\documentclass[twoside,11pt]{article}

\PassOptionsToPackage{table,dvipsnames}{xcolor}

\usepackage[preprint]{jmlr2e}
\usepackage[utf8]{inputenc}
\usepackage[T1]{fontenc}
\usepackage{url}
\usepackage{booktabs}
\usepackage{amsfonts}
\usepackage{amsmath, bbm}
\usepackage{amssymb}
\usepackage{nicefrac}
\usepackage{microtype}
\usepackage{xcolor}
\usepackage{graphicx, subcaption}
\graphicspath{{./}}
\usepackage{adjustbox}
\usepackage{algorithm}
\usepackage{algorithmic}

\hypersetup{colorlinks=true, allcolors=blue}
\usepackage[nameinlink,noabbrev,capitalise]{cleveref}

\crefname{thm}{theorem}{theorems}
\Crefname{thm}{Theorem}{Theorems}
\crefname{prop}{proposition}{propositions}
\Crefname{prop}{Proposition}{Propositions}
\crefname{lemma}{lemma}{lemmas}
\Crefname{lemma}{Lemma}{Lemmas}
\crefname{rmk}{remark}{remarks}
\Crefname{rmk}{Remark}{Remarks}

\DeclareMathOperator{\E}{\mathbb{E}}

\def\d{\text{d}}

\begin{document}

\title{Optimizing Irreversible Perturbations of the Unadjusted Langevin Algorithm}

\author{\name Qianyu Julie Zhu \\
       \addr Massachusetts Institute of Technology \\
       \addr Cambridge, MA, USA \\
       \addr qianyu\_z@mit.edu \\
       \AND
       \name Youssef Marzouk \\
       \addr Massachusetts Institute of Technology\\
       \addr Cambridge, MA, USA \\
       \addr ymarz@mit.edu \\
       \AND
       \name Konstantinos Spiliopoulos\\
       \addr Boston University\\
       \addr Boston, MA, USA \\
       \addr kspiliop@bu.edu \\
       \AND
       \name Benjamin Zhang\\
       \addr University of North Carolina at Chapel Hill \\
       \addr Chapel Hill, NC, USA \\
       \addr bjz@unc.edu \\
}

\maketitle

\begin{abstract}
Irreversible perturbations accelerate the convergence of Langevin dynamics, breaking detailed balance while preserving the invariant measure.
The design of optimal irreversible perturbations has been studied in the continuous-time Gaussian setting, but extensions to non-Gaussian target distributions---and the impact of \emph{time discretization} on the design of optimal perturbations---have not been well understood.
Numerical discretizations of Langevin dynamics introduce bias, which is typically exacerbated by irreversible perturbations; handling this interaction demands a joint treatment of acceleration and accuracy.
This paper develops a systematic framework for optimizing position-independent irreversible perturbations of the unadjusted Langevin algorithm (ULA).
We formulate a constrained optimization problem that simultaneously accounts for mixing efficiency and discretization bias, where the former is characterized by a spectral gap analogue and the latter is quantified via a weighted expected squared jump distance.
Within this framework, we derive an explicit characterization of the optimal position-independent irreversible perturbation.
Extensive numerical experiments demonstrate that our design yields faster convergence with controlled bias, and improves mean squared estimation errors compared to other choices of irreversible perturbation.
\end{abstract}


\section{Introduction}

In modern Bayesian inference and scientific computing, a central challenge is to compute expectations under a target distribution $\pi$ known only up to normalization \citep{robert1999monte, box2011bayesian, james1980monte, durmus2019high}.
Markov chain Monte Carlo (MCMC) methods address this problem by constructing a Markov process with stationary distribution $\pi$, enabling estimation via ergodic averages without explicit integration (e.g., quadrature)
in high dimensions \citep{andrieu2003introduction, geyer1992practical, geyer2011introduction, brooks2011handbook, gilks1995markov}.
This paper focuses on the \textit{unadjusted Langevin algorithm} (ULA), a gradient-based MCMC method derived from discretizing overdamped Langevin dynamics. ULA is widely used for its simplicity and scalability, and its mixing rate and discretization bias have been studied extensively \citep{rossky1978brownian, roberts1998optimal, bussi2007accurate, cheng2018underdamped}.

Despite its simplicity, ULA faces two key limitations: slow mixing on challenging targets and discretization bias that grows with step size \citep{roberts1996exponential, roberts1998optimal}.
To address these issues, two complementary strategies have been developed to accelerate convergence by enlarging spectral gaps and reducing asymptotic variance. \textit{Reversible perturbations} reshape the dynamics using local or global geometric information \citep{gilks1996strategies, mira2001ordering, girolami2011riemann, leimkuhler2018ensemble, parno2018transport, dalalyan2017theoretical, gabrie2022adaptive, hird2023quantifying, Ti2023, cui2024optimal}, while \textit{irreversible perturbations} introduce additional drift to enhance exploration of the state space \citep{hwang1993accelerating, diaconis2000analysis, mira2000non, neal2004improving, hwang2005accelerating, sun2010improving, bierkens2016non, zhang2022geometryinformed,ReyBelletSpiliopoulos_2015a,ReyBelletSpiliopoulos_2015b,ReyBelletSpiliopoulos_2016}.
Both strategies are well studied in continuous time: reversible perturbations, also understood as preconditioning, are supported by rigorous optimality principles, and irreversible perturbations have been shown to accelerate a wide range of diffusion-based samplers.
For irreversible perturbations, however, most analyses remain qualitative and/or confined to the continuous-time limit. \citet{Lelivre2013} provide a principled characterization of \emph{optimal} perturbations in the continuous-time Gaussian setting. \citet{zhang2022geometryinformed} propose a way of constructing geometry-informed irreversible perturbations for more general non-Gaussian targets, but do not invoke a particular notion of optimality.
It is also important to note that irreversible perturbations can introduce stiff, multiscale dynamics, which in turn degrade the accuracy of standard time discretization schemes, as observed in numerical studies of irreversible Langevin samplers \citep{lu2018analysis}. A principled framework for designing irreversible perturbations that simultaneously accounts for discretization error and extends beyond Gaussian target distributions $\pi$ remains largely unrealized.

This paper develops a principled framework for designing \emph{state-independent} irreversible perturbations of ULA, explicitly balancing two competing goals: accelerating mixing and controlling discretization bias.
Our framework first imposes a constraint that enforces a mixing-improvement criterion, motivated by spectral-gap optimality in the Gaussian setting but extended to the non-Gaussian case via the \textit{Fisher information matrix}. Then we optimize within this admissible class of perturbations, minimizing a discretization error proxy based on a \textit{weighted expected squared jump distance} (ESJD$^W$). The solution set
of this optimization problem can be characterized explicitly, and yields an irreversible perturbation that achieves the best trade-off between these two objectives.

We demonstrate our method on several representative target distributions: anisotropic Gaussians, Gaussian mixtures, and the posterior distributions of Bayesian logistic regression and Bayesian independent component analysis problems. In these experiments, ULA equipped with our optimal irreversible perturbation achieves a favorable trade-off between mixing efficiency (estimator variance) and estimator bias, and lower mean-squared estimation errors overall. Additionally, it remains stable over a broad range of step sizes, in regimes where alternative constructions of irreversible perturbations can diverge or stagnate.

The rest of the paper is organized as follows:
\Cref{sec:background} reviews irreversible perturbations of the overdamped Langevin dynamics.
\Cref{sec:method} presents our method for constructing optimal irreversible perturbations, informed by the Fisher information of the target distribution, together with a practical approximation algorithm based on online Fisher information matrix estimation.
\Cref{sec:numerical} reports numerical experiments comparing our method to alternative designs of irreversible perturbation.
\Cref{sec:conclusion} provides a discussion of future research directions.

The appendix contains derivations of formulas, proofs of the theoretical results of the paper, and additional numerical examples. In particular, Appendix \ref{app:reversible} discusses reversible perturbations in the context of choosing a preconditioner that minimizes the expected square jump distance (ESJD), in the spirit of \citet{Ti2023}, which has partially motivated the approach in this paper. Appendix \ref{app:derivations} contains explicit derivations of various formulas used throughout the paper, including the mean and covariance of Gaussian ULA iterates and the discretization error originating from the It\^o--Taylor expansion. Appendix \ref{app:proofs} contains rigorous and complete proofs of the theoretical results of the paper. Appendix \ref{app:fisher-matrix-visual} discusses aspects of Fisher information-based preconditioning. Appendix \ref{app:add-experiment-details} contains additional experiments and numerical details. Appendix \ref{app:robustness} offers an empirical and target-dependent comparison of irreversible and reversible perturbations.
Code for our algorithms and experiments is available at \href{https://github.com/qianyu-zhu/IrrevULA}{https://github.com/qianyu-zhu/IrrevULA}.

\section*{Acknowledgments}\label{sec:acknowledgments}

JZ and YM acknowledge support from the US Department of Energy, Office of Advanced Scientific Computing Research, under grants DE-SC0023188 and DE-SC0026245. KS was partially supported by the National Science Foundation (NSF) under grant DMS-2311500. BZ was supported by AFOSR grant FA9550-21-1-0354.

\section{Background on sampling from equilibrium}\label{sec:background}

Let $\pi\propto e^{-V}$ denote a target density on $\Omega \subset \mathbb{R}^d$.
The overdamped Langevin diffusion targeting $\pi$ is
\begin{equation}\label{eq:overdamped_LD}
    \mathrm{d} X_t = -\nabla V(X_t) \, \mathrm{d}t + \sqrt{2} \, \mathrm{d}W_t,
\end{equation}
where $W_t$ is a standard Brownian motion and the initial condition of the process is $X_0$.
Let $\pi_t$ denote the law of $X_t$. (When $\pi_t$ has a density with respect to the Lebesgue measure on $\mathbb{R}^d$, we use the same notation for both measure and density.)
Under mild regularity conditions on the potential $V$ (e.g., it is smooth and confining), the process admits a unique invariant measure $\pi$ and $\pi_t$ converges to $\pi$ weakly \citep[Proposition 4.2]{pavliotis2014stochastic}.
The ergodic theorem then ensures
\begin{equation*}
    \E_\pi[f(X)] = \int_{\mathbb R^d}f(x)\pi(x)\d x = \lim_{T\rightarrow \infty}\frac{1}{T}\int_0^T f(X(t))\d t.
\end{equation*}
For a detailed treatment, see \citet[Section 2.4]{pavliotis2014stochastic}.

As discussed in the introduction, we focus on irreversible perturbations, which introduce an additional drift term $\phi: \mathbb{R}^d \to \mathbb{R}^d$:
\begin{equation}\label{eq:irr-LD}
    \mathrm{d}X_t \;=\; \big[\,\nabla\log\pi(X_t)\;+\;\phi(X_t)\,\big]\mathrm{d}t \;+\; \sqrt{2}\,\mathrm{d}W_t.
\end{equation}
The process leaves $\pi$ invariant whenever $\phi$ is $\pi$–divergence free, i.e.\ $\nabla\cdot\big(\phi(x)\,\pi(x)\big)=0$.
A natural choice satisfying this constraint is $\phi(x)=J\,\nabla\log\pi(x)$ with a constant skew-symmetric matrix $J^\top=-J$, since
\begin{equation*}
    \nabla\!\cdot\!\big(\phi\,\pi\big)=\nabla\!\cdot\!\big(J\nabla\pi\big)=\mathrm{tr}\big(J\,\nabla^2\pi\big)=0.
\end{equation*}
We call this the \textit{state-independent} case; see Remark~\ref{rmk:state-dependent-perturbation} for the general case.

Such perturbations break detailed balance by inducing a rotational vector field which enhances state-space exploration \citep{bierkens2016non, diaconis2000analysis}.
Prior work has shown that irreversible drift accelerates convergence to equilibrium, with optimal rates and corresponding optimality criteria characterized for Gaussian targets \citep{hwang1993accelerating, hwang2005accelerating}. Building on this line of work, \citet{Lelivre2013} provide an explicit construction of optimal irreversible
perturbations achieving these rates.
Behavior of the generator under strongly scaled nonreversible drift has also been analyzed \citep{berestycki2005elliptic, constantin2008diffusion, franke2010behavior}, and subsequent work established that irreversible perturbations for overdamped Langevin dynamics reduce the asymptotic variance of the ergodic estimators \citep{ReyBelletSpiliopoulos_2015a,ReyBelletSpiliopoulos_2015b, ReyBelletSpiliopoulos_2016, duncan2016variance}.

In this work, we focus on state-independent case and identify a class of perturbations $J$ that jointly reduces mixing time and discretization error. For brevity, we refer to $J$ itself as the \textit{irreversible perturbation}.
\begin{rmk}\label{rmk:state-dependent-perturbation}
    More generally, irreversible perturbations take the form
\begin{equation*}
    \phi(x) = C(x) \, \nabla \log \pi(x) + \nabla \cdot C(x),
\end{equation*}
where the matrix field $C(x) = -C(x)^\top$ may depend on $x$.
This form encompasses the state-independent case but also enables geometry-aware designs that adapt to the local structure \citep{zhang2022geometryinformed, wu2023accelerating}.
A notable instance, proposed by \citet{zhang2022geometryinformed}, sets \(C(x)=\tfrac12\!\left(JA(x)+A(x)J\right)\) with \(J^\top=-J\), aligning the irreversible drift with the metric induced by a symmetric positive definite preconditioner \(A(x)\). We leave the study of such \textit{state-dependent} extensions to future work.
\end{rmk}

\subsection{Optimal irreversible perturbation for Gaussian targets in continuous time}
For Gaussian targets, the impact of irreversible perturbations can be quantified exactly through the spectral gap, enabling explicit characterization of optimal perturbations.
Consider $\pi \propto \exp(-x^\top S x / 2)$ with irreversible perturbations $\phi(x) = JSx$, $J = -J^\top$ \citep{hwang1993accelerating, Lelivre2013}. The infinitesimal generator of the corresponding Langevin dynamics is
\begin{equation*}
    \mathcal{L} = -(I + J) S x \cdot \nabla + \Delta \, .
\end{equation*}
The spectrum of $\mathcal{L}$ in $L^p(\pi)$, $p>1$, consists of negative integer linear combinations of the eigenvalues of $(I + J) S$ \citep{metafune2002spectrum}:
\begin{equation*}
    \sigma(\mathcal{L}) = \left\{ -\sum_{j=1}^r n_j \sigma_j \; : \; \forall\, n_j \in \mathbb{N}_0 \right\},
\end{equation*}
where $\{\sigma_j\}_{j=1}^r$ are the (generally complex) eigenvalues of $(I + J) S$, all satisfying $\Re(\sigma_j)>0$ (where $\Re$ denotes the real part).
The spectral gap is thus determined by the smallest among $\{\Re(\sigma_j)\}_{j=1}^r$.
Since $J$ is skew-symmetric, $\mathrm{Tr}((I+J)S) = \mathrm{Tr}(S)$, so the sum of the real parts of the eigenvalues $\{\sigma_j\}_{j=1}^r$ is fixed.
Thus, maximizing the spectral gap corresponds to maximizing the smallest real part under a fixed sum constraint, which is achieved when all $\Re(\sigma_j)$ are equal
\citep{Lelivre2013}:
\begin{equation}\label{eq:J_spec}
    \mathcal J_{\mathrm{s}}(S)
    :=
    \left\{J\in\mathbb R^{d\times d}: J^\top=-J,\;
    \min_i \Re \big( \sigma_i( (I + J) S ) \big)
    = \frac{\mathrm{Tr}(S)}{d}\right\}.
\end{equation}
where the subscript ``s'' denotes \emph{spectral}. When $S$ is clear, we write $\mathcal J_{\mathrm{s}}$.
The maximizer of the spectral gap in continuous time is generally not unique; \citet{Lelivre2013} further proposed a constructive algorithm that yields admissible maximizers within this class.
However, this analysis is confined to continuous time and does not account for discretization error, a critical consideration in practice.
We address this gap in \Cref{sec:method}.

\section{Design of optimal position-independent irreversible perturbations}
\label{sec:method}

In practice, the continuous-time Langevin dynamics are approximated by a discrete-time scheme, typically Euler--Maruyama \citep[Section~3]{higham2001algorithmic}:
\begin{equation}\label{eq:ULA}
    X_{t} = X_{t-1} + h(I+J)\nabla\log\pi(X_{t-1})+ \sqrt{2h} \, Z_{t-1},\qquad t = 1,2,\dots
\end{equation}
where $h>0$ is the step size and $\{Z_t\}_{t\geq 0}$ are i.i.d.\ standard Gaussian.
This first-order discretization, also referred to as the unadjusted Langevin algorithm (ULA), admits an invariant distribution $\pi_h$ that generally differs from $\pi$.
While a Metropolis–Hastings correction can recover $\pi$, the asymmetry induced by irreversible perturbations lowers acceptance rates and adds computational cost, especially in high dimensions \citep{xifara2014langevin, OttobrePillaiSpiliopoulos2020}. We therefore focus on unadjusted dynamics and design $J$ to balance mixing acceleration against discretization bias.

Our approach is one of constrained optimization. Among choices of $J$ that satisfy the Gaussian spectral optimality defined by \cite{Lelivre2013}, we will minimize a discretization-error proxy based on the weighted expected squared jump distance. We begin with the Gaussian setting, where spectral gap can be calculated exactly, before extending these principles to general targets.

For a Gaussian target with mean $\mu$ and covariance matrix $S^{-1}$, the ULA update in \eqref{eq:ULA} yields
\begin{equation}\label{eq:ULA-gaussian}
    X_{t} = X_{t-1} - h (I+J)S(X_{t-1}-\mu) + \sqrt{2h}\,Z_{t-1}.
\end{equation}
With $X_0 \sim \mathcal{N}(\mu_0, \tau^2 I)$, the iterates remain Gaussian with mean and covariance
\begin{align*}
    \mu_t(h) = \mu+B^t(\mu_0-\mu),\ \
    \Sigma_t(h) = \tau^2B^t(B^\top)^t + 2h\sum_{i=0}^{t-1}B^i(B^\top)^i, \ \  B:=I-h(I+J)S.
\end{align*}
See \Cref{app:gaussian-example} for a detailed derivation. As $t\rightarrow\infty$, $\pi_t$ converges to the Gaussian distribution $\pi_h$, whose  moments as functions of $h$ are
\begin{align}
\label{eq:momentswithbias}
    \mu(h) = \mu,\qquad
    \Sigma(h) = S^{-1}+h\Sigma^{(1)}+ \mathcal O(h^2),
\end{align}
where $\Sigma^{(1)}$ satisfies
\begin{align}
    A\Sigma^{(1)}+\Sigma^{(1)}A^\top=(I+J) S (I+J)^\top, \ \text{with} \ A := (I+J)S.
    \end{align}
While the mean is unbiased, the covariance incurs bias driven by $(I+J)S(I+J)^\top$, reflecting the discrepancy between the invariant distribution $\pi_h$ of the discretized dynamics and the target distribution $\pi$.

Moment estimation error provides useful intuition for how discretization affects downstream statistics, but it is not a sufficiently general design objective.
Directly optimizing the bias of a particular moment estimator is both observable-dependent and less principled as a general design criterion.
Since our goal is to choose perturbations that improve the overall performance of the discretized sampler, we instead introduce tractable proxies for discretization error.
These proxies capture the leading-order dependence on $J$ while remaining general enough to support analysis and optimization.

\subsection{Discretization error proxies}\label{sec:disc_err_criteria_ESJD}

We introduce three proxies for the discretization error, evaluated in closed form for the ULA~\eqref{eq:ULA-gaussian} with $X_t\sim\pi=\mathcal N(\mu, S^{-1})$.

\paragraph{(i) ESJD.}
The expected squared jump distance, used as a discretization error criterion in \citet{Ti2023}, measures the average squared displacement per iteration:
\begin{equation}\label{eq:ESJD-def}
    \mathrm{ESJD}(X_t, h)\;:=\;\mathbb{E}\big[\|X_{t+h}-X_t\|^2\big].
\end{equation}
Substituting the Gaussian ULA step~\eqref{eq:ULA-gaussian} at stationarity,
\begin{align*}
    \mathrm{ESJD}(X_t,h)
&= \mathrm{Tr}\!\left(h^2 A\, S^{-1}A^\top\right)+2h d
= -h^2\mathrm{Tr}\big[ J S J \big] +g_{1}(S,h),\qquad A:=(I+J)S,
\end{align*}
where $g_{1}(S,h):=h^2\mathrm{Tr}[S]+2h d$ collects $J$-independent terms and we used $\text{Tr}[JS+SJ^\top]=0$ by skew-symmetry.
Minimizing ESJD over $J$ thus reduces to minimizing $\mathcal E_1(J) :=-\mathrm{Tr}\big[ J S J \big]$.

This proxy has a direct connection to the moment estimation error in \eqref{eq:momentswithbias}: the $J$-dependent part of ESJD exactly equals the $J$-dependent part of the forcing term driving the leading covariance bias. In particular, $\mathrm{Tr}\bigl( (I+J) S (I+J)^\top \bigr ) = \mathrm{Tr}(S) + \mathrm{Tr}(JSJ^\top)$. Larger irreversible perturbations produce larger jumps and, simultaneously, larger bias.

\paragraph{(ii) ESJD$^W$.}
For anisotropic targets, a large step along a high-curvature direction inflates bias more than the same step along a flat direction.
A natural remedy is to measure jumps in coordinates that \emph{isotropize} the target.
This generally motivates a weighted ESJD with weight $W \succ 0$:
\begin{equation}\label{eq:w-ESJD-def}
    \mathrm{ESJD}^W(X_t, h)
    \;:=\;\mathbb{E}\big[\|X_{t+1}-X_t\|_{W}^2\big].
 \end{equation}
where we use the notation $\|a \|_W^2 := a^\top W a$.
At stationarity,
\begin{align*}
   \mathrm{ESJD}^W(X_t,h)
   &= h^2\mathrm{Tr}\!\left(W(I+J)S(I+J)^\top\right)+2h\mathrm{Tr}[W].
\end{align*}
Taking $W=S$ makes the cross terms vanish. In particular, we get that
\begin{align*}
\mathrm{ESJD}^S(X_t,h)
   &= - h^2\mathrm{Tr}(JSJS)+g_{2}(S,h),
\end{align*}
with $g_{2}(S,h)=-h^2\mathrm{Tr}(SS)+2h\mathrm{Tr}[S]$. This calculation
gives the proxy $\mathcal E_2(J) :=  -\mathrm{Tr}(JSJS)$. Compared with $\mathcal E_1$, this proxy favors larger effective steps along directions where the target geometry is flatter.

\paragraph{(iii) It\^o--Taylor remainder.}
A complementary viewpoint comes from the numerical analysis of SDEs: the local truncation error of the Euler--Maruyama scheme directly quantifies the per-step approximation quality.
The It\^o--Taylor expansion at $X_t$ (the stochastic analogue of a Taylor series) gives the leading error term as
\begin{equation*}
    X_{t+h} - \tilde{X}_{t+h} = -\sqrt{2}(I+J)S\,Z_{t,h}+ \mathcal{O}(h^2),
\end{equation*}
where $\tilde{X}_{t+h}$ denotes the EM approximation of $X_{t+h}$ and $Z_{t,h}\sim\mathcal{N}(0,\tfrac{h^3}{3}I)$.
Its expected squared norm $\mathbb{E}\big[\|X_{t+h}-\tilde{X}_{t+h}\|^2\big]$ is proportional to $\mathrm{Tr}[(I+J)SS(I+J)^\top]$, yielding $\mathcal E_3(J) := -\mathrm{Tr}(J S S J)$ after isolating $J$-dependent terms; see \Cref{app:ito-taylor} for the full derivation.

\begin{figure}[htbp]
    \centering
        \centering
        \includegraphics[width=0.7\textwidth]{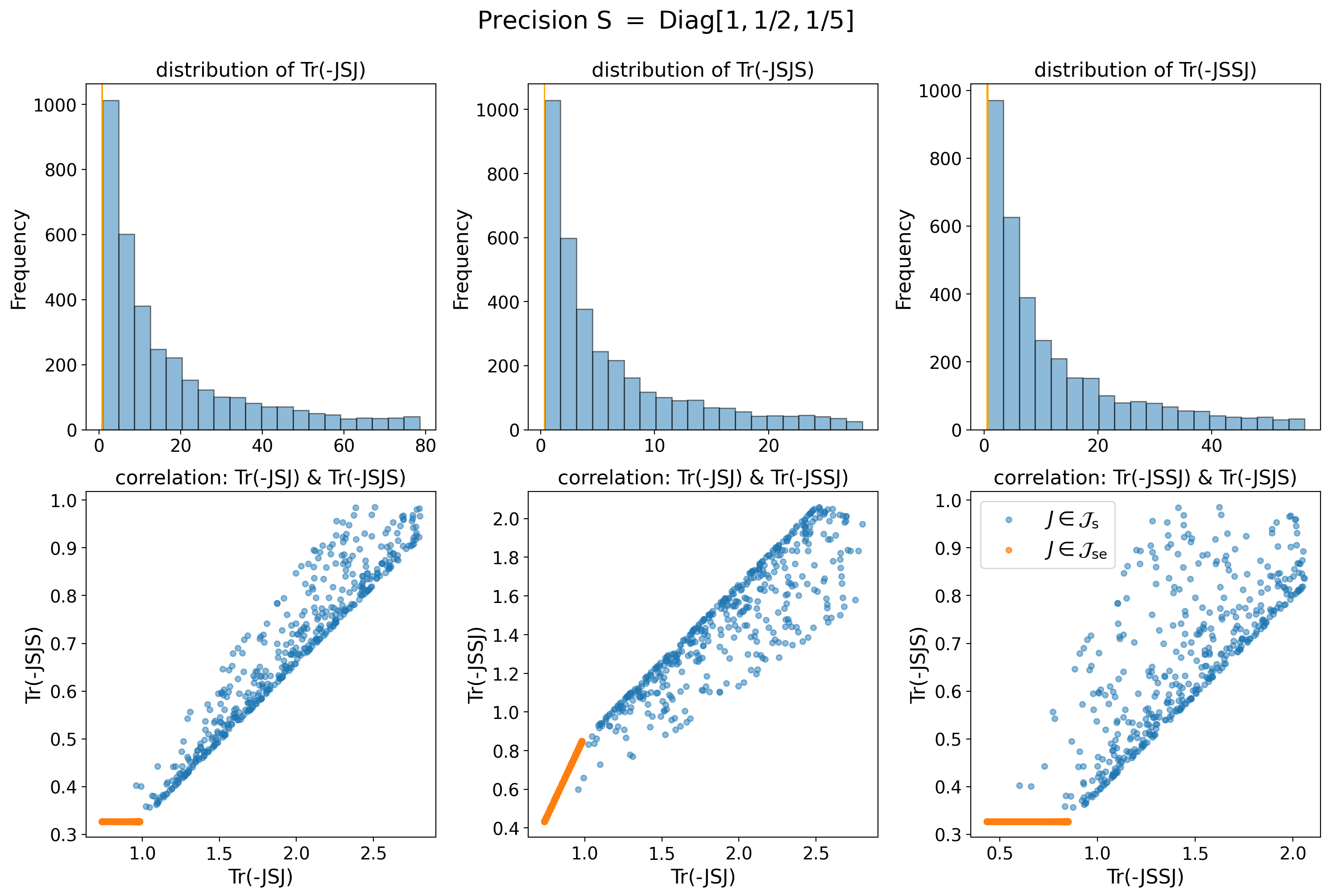}
    \caption[Relationship between different discretization criteria]{Marginal distributions (top) and pairwise correlations (bottom) of $\text{Tr}(-JSJ)$, $\text{Tr}(-JSJS)$, and $\text{Tr}(-JSSJ)$ for spectrally optimal perturbations $J\in\mathcal J_{\mathrm{s}}$ (blue). The orange points correspond to perturbations $J\in\mathcal J_{\mathrm{se}}$ that additionally minimize $\mathcal E_2$; they are discussed in \Cref{sec:optimal-J-constraint}.}
    \label{fig:relation_different_target}
\end{figure}

To summarize, we proposed the following three discretization error proxies:
$$\mathcal E_1(J)  = -\mathrm{Tr}(J S J), \qquad \mathcal E_2(J) =-\mathrm{Tr}(J S J S),\qquad \mathcal E_3(J)  = -\mathrm{Tr}(J S S J).$$
The next section compares these functions and examines how well they track the mean-square estimation error that ultimately determines sampling performance.

\subsection{Comparison of proxies and relationship with MSE}
\label{sec:proxy_comparison}

As a numerical illustration, we fix a Gaussian target with precision $S=\operatorname{diag}(1,1/2,1/5)$, and randomly sample $J\in\mathcal J_\mathrm{s}$ using the constructive algorithm of \citet{Lelivre2013}. \Cref{fig:relation_different_target} and \Cref{fig:relation_different_target_mse} reveal the following key findings.

\begin{itemize}

\item The histograms in the top row of \Cref{fig:relation_different_target} show that none of $\mathcal E_{\{1,2,3\}}(J)$ is bounded over $\mathcal J_\mathrm{s}$. Thus, spectral optimality alone does not control discretization error: arbitrary choices of spectrally optimal perturbation can yield arbitrarily large discretization error. This motivates a second-stage optimization of $\mathcal E_i(J)$ within $\mathcal J_\mathrm{s}$.

\item The pairwise scatter plots in the bottom row of \Cref{fig:relation_different_target} show strong positive correlation among all three criteria. This suggests that optimizing a single representative proxy is sufficient in practice.
  Proposition~\ref{prop:proxy-equivalence} below formalizes this equivalence.

\item \Cref{fig:relation_different_target_mse}  connects the proxies to estimator accuracy. In this example, the proxies show little correlation with the MSE of the estimator of the first moment $\mathbb E[\mathrm{sum}(X^{(i)})]$, but strong positive correlation with the MSE of the estimator of the second moment $\mathbb E[\|X\|_2^2]$. This matches the discretization-bias analysis in the opening of \Cref{sec:method}: mean estimation is asymptotically unbiased even after discretization, whereas covariance-related quantities are affected by discretization.
\end{itemize}

The next proposition establishes equivalence of $\mathcal{E}_{ \{1, 2, 3\} }$  up to constants depending only on the spectrum of $S$. See \Cref{app:proof:prop:proxy-equivalence} for the proof.
\begin{prop}\label{prop:proxy-equivalence}
Let $S = S^\top\succ 0$ and let $J\in\mathbb R^{d\times d}$ be skew-symmetric. Let $\sigma_{\min}(S)$ and $\sigma_{\max}(S)$ denote the smallest and largest eigenvalues of $S$. Then $\mathcal E_i(J)\ge 0$ for $i=1,2,3$, and
\[
\sigma_{\min}(S)\,\mathcal E_1(J)
\ \le\
\mathcal E_2(J)
\ \le\
\sigma_{\max}(S)\,\mathcal E_1(J),
\qquad
\sigma_{\min}(S)\,\mathcal E_1(J)
\ \le\
\mathcal E_3(J)
\ \le\
\sigma_{\max}(S)\,\mathcal E_1(J).
\]
Consequently, $\mathcal E_2(J)$ and $\mathcal E_3(J)$ are equivalent up to the condition number of $S$:
\[
\frac{\sigma_{\min}(S)}{\sigma_{\max}(S)}\,\mathcal E_2(J)
\ \le\
\mathcal E_3(J)
\ \le\
\frac{\sigma_{\max}(S)}{\sigma_{\min}(S)}\,\mathcal E_2(J).
\]
\end{prop}

Together, these findings justify optimizing a single discretization-error proxy within $\mathcal J_\mathrm{s}$. We will see in \Cref{sec:optimal-J-constraint} that only  $\mathcal E_2$ admits a tractable closed-form optimization under the spectral constraint.

\begin{figure}[htbp]
        \centering
        \includegraphics[width=0.8\textwidth]{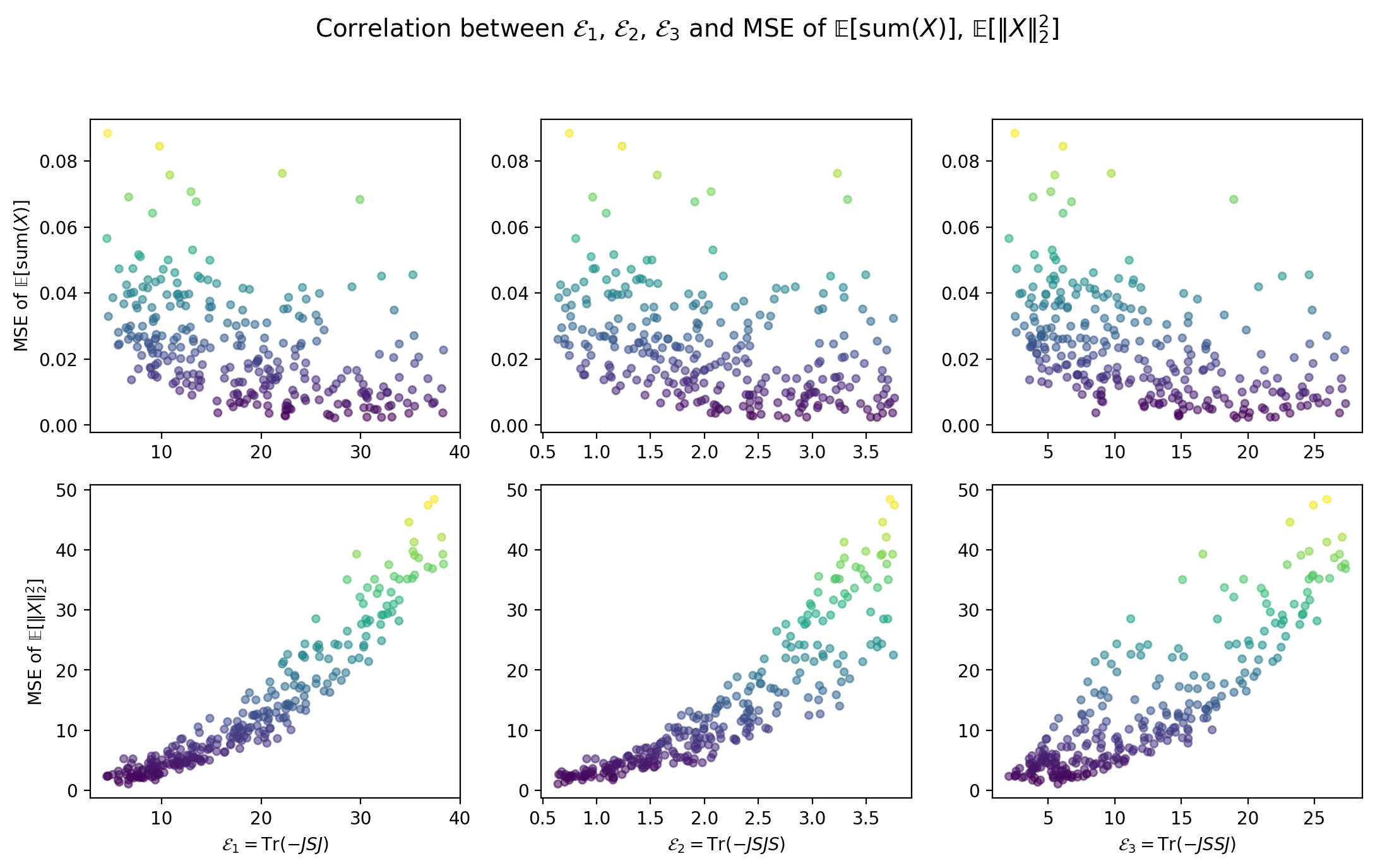}
    \caption[Correlation between statistics MSE and discretizations]{Correlation between the MSE of ULA estimators of $\mathbb E[\text{sum}(X)]$ and $\mathbb E[\|X\|^2_2]$, and three discretization error proxies, $\mathcal E_1(J)$, $\mathcal E_2(J)$, $\mathcal E_3(J)$.  $J$ are sampled from $\mathcal J_\mathrm{s}$,
     and the target distribution is Gaussian with precision $S=\operatorname{diag}(1,1/2,1/5)$. All three proxies are positively correlated with the MSE of the $J$-dependent estimators of $\mathbb E[\|X\|^2_2]$, but not with the MSE of the estimators of $\mathbb E[\text{sum}(X)]$. The color represents MSE.}
    \label{fig:relation_different_target_mse}
\end{figure}

\subsection{Illustration: two-dimensional Gaussian}\label{sec:2D:Gauss}
Before presenting our general method for selecting $J\in\mathcal J_\mathrm{s}$ while controlling discretization error,
we first illustrate the key trade-off in a two-dimensional setting with a Gaussian target, where every skew-symmetric $J\in\mathbb R^{2\times 2}$ is of the form
\[
    J = \begin{pmatrix} 0 & a \\ -a & 0 \end{pmatrix} \qquad\text{ with }\quad a \in \mathbb{R}.
\]
For the precision matrix $S = \text{diag}(1,\varepsilon)$, the spectral optimality condition in \citet{Lelivre2013} is attained by every $a$ satisfying $|a|\geq (1-\varepsilon)/2\sqrt{\varepsilon}$.
Thus, continuous-time spectral optimality alone does not identify a unique perturbation.
Once discretization error is taken into account, however, the ambiguity is removed: minimizing $\mathcal E_2$ over this spectrally optimal class selects the two boundary values $a^* =\pm (1-\varepsilon)/2\sqrt{\varepsilon}$.

\Cref{fig:ESJD} visualizes the first part of this trade-off. The blue shaded region denotes the spectrally optimal class: once \(|a|\) exceeds the threshold, the spectral gap has reached its maximal value, so all larger perturbations are equivalent from the continuous-time spectral perspective.
The discretization proxy, however, continues to increase with \(|a|\).
Thus, minimizing \(\mathrm{ESJD}^{S}\), equivalently \(\mathcal E_2\), selects the smallest perturbation that is still spectrally optimal, namely the bifurcation point where the spectral constraint first becomes active.
In this two-dimensional Gaussian example, this same boundary value also minimizes \(\mathcal E_1 \) and \(\mathcal E_3\), although this coincidence is special to this setting.

\Cref{fig:J_scale_comparison} then shows why this distinction matters for finite-step estimation.
For linear statistics, such as \(\mathbb E[\operatorname{sum}(X)]\), larger \(|a|\) primarily improves mixing and therefore reduces variance and hence MSE.
For nonlinear statistics, however, the behavior is no longer monotone: increasing \(|a|\) reduces variance at first, but excessive perturbation strength introduces discretization bias.
Thus, perturbations that are equivalent in continuous time can lead to different finite-sample accuracy after discretization.

This example identifies the basic selection problem.
Spectral optimality determines a class of accelerating perturbations $\mathcal{J}_s$, but discretization error determines which member of that class is preferable in practice.
In two dimensions this reduces to choosing the boundary point \(a^\ast\); in higher dimensions, the spectrally optimal set has many more degrees of freedom because the space of skew-symmetric matrices has dimension \(d(d-1)/2\).
This motivates the general optimization framework developed next.
\begin{figure}[htbp]
\centering
\includegraphics[width=0.95\linewidth]{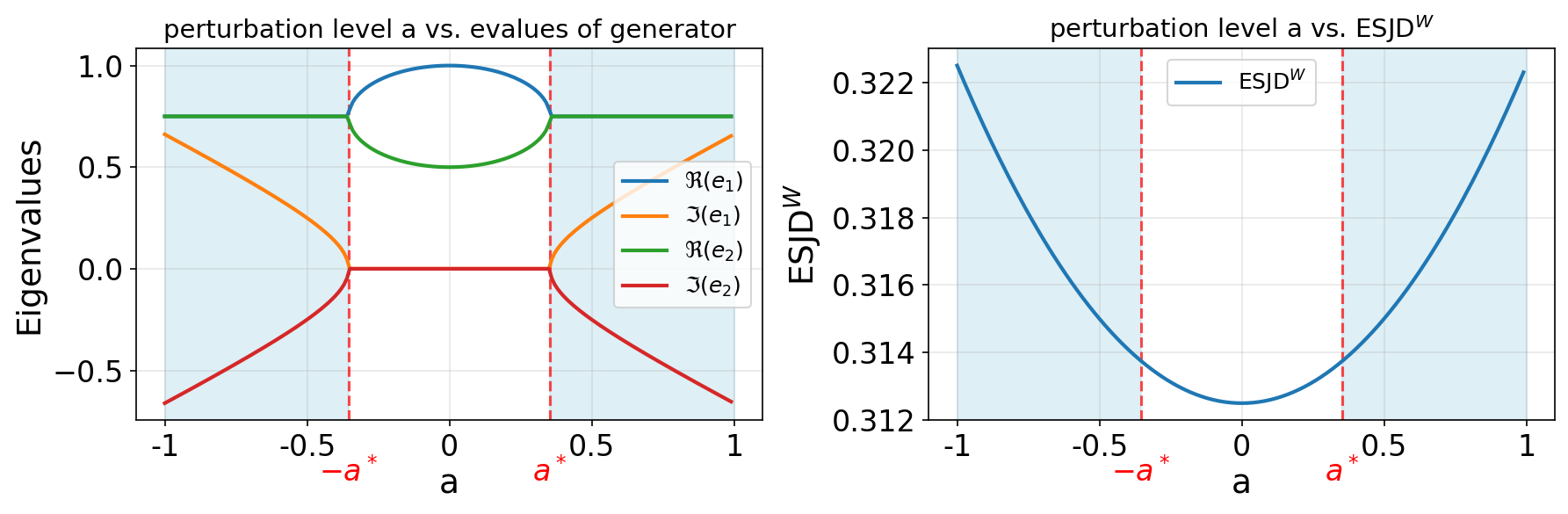}
\caption[2-D ESJD$^W$ vs. spectral gap]{In 2-D, the spectral gap increases with perturbation strength until it reaches its maximal value at \(a^\ast\), while ESJD$^W$ (with $W=S$) continues to increase beyond this point. Thus \(a^\ast\), up to sign, is the discretization-optimal choice within the spectrally optimal class.}
\label{fig:ESJD}
\end{figure}

\begin{figure}[htbp]
    \centering
    \includegraphics[width=0.95\linewidth]{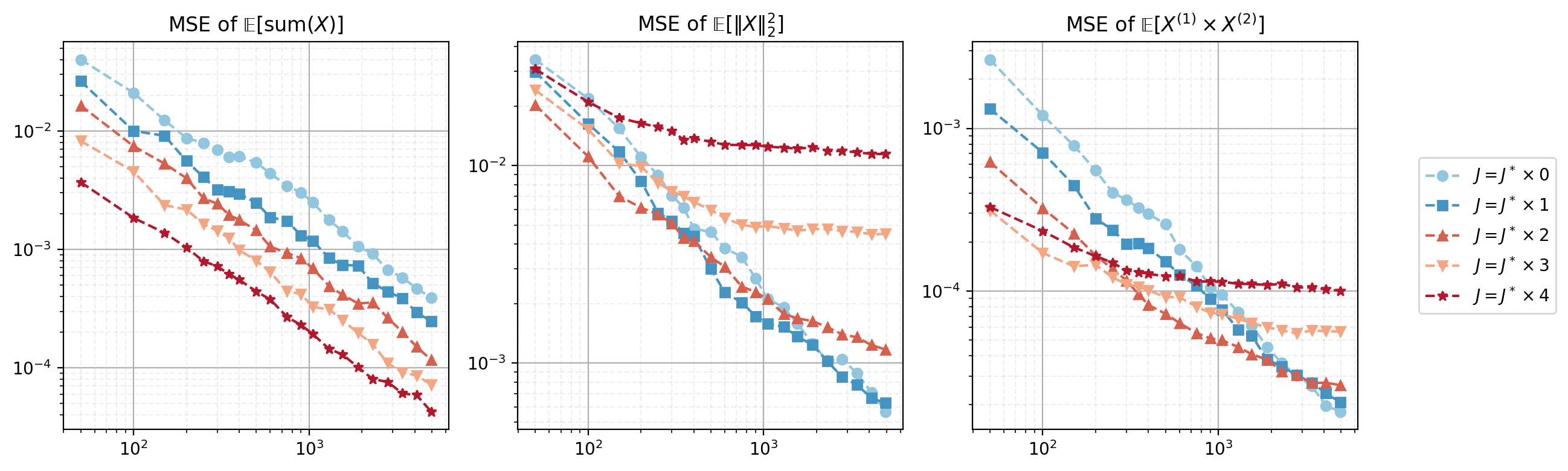}
    \caption[2D Gaussian, MSE vs. Time]{MSE of three statistics for different perturbation scales, in 2-D. The curve with square markers corresponds to the optimal $J^* = [0 \ \ a^\ast; -a^\ast \ \ 0]$. The other curves correspond to $J$ with different scaling. The errors corresponding to larger $J$ are small when the total number of steps is small, but they plateau at higher values earlier due to numerical bias. The bias-variance tradeoff is less present for symmetric observables such as $\mathbb E[\text{sum}(X)]$.}
    \label{fig:J_scale_comparison}
\end{figure}

\subsection{Optimal irreversible perturbations as solution to constrained optimization}
\label{sec:optimal-J-constraint}

We now formalize the design of optimal irreversible perturbations as a constrained optimization problem that minimizes discretization error while maximizing spectral gap.
Although all three proxies correlate positively with MSE and are equivalent up to spectral constants, as shown in the preceding subsections, only $\mathcal E_2$ leads to an explicit characterization of the optimal solutions.

We write \(\gamma_\star=\mathrm{Tr}(S)/d\) and minimize $\mathcal E_2(J)$ subject to the spectral gap constraint of \citet{Lelivre2013}:
\begin{equation}\label{eq:P1}
\begin{aligned}
    \min_{J:J = -J^\top} \quad &\mathcal E_2(J) := \mathrm{Tr}(-JSJS) \\
    \text{s.t.}  \qquad &\Re\!\left[\sigma_i\big((I+J)S\big)\right] = \gamma_\star, \ \  \forall i.
\end{aligned}
\tag{P1}
\end{equation}
The constraint enforces spectral optimality: all eigenvalues of $(I+J)S$ share the same real part, maximizing the spectral gap for the continuous-time process. The objective penalizes large jumps in high-curvature directions, controlling discretization bias. Together, they identify perturbations that accelerate mixing without sacrificing numerical accuracy.

\paragraph{Equivalent spectral coordinate formulation.}
Now define
\[
    \tilde J=S^{1/2}JS^{1/2},
    \qquad
    \tilde B_J=S+\tilde J
    =S^{1/2}(I+J)S\,S^{-1/2}.
\]
Then \(\tilde B_J\) is similar to \((I+J)S\), and
\[
    \mathrm{Tr}(-JSJS)=\|\tilde J\|_F^2.
\]
For any orthonormal matrix \(\Psi=[\psi_1,\ldots,\psi_d]\), write
\[
    B_\Psi=\Psi^\top S\Psi,
    \qquad
    M=\Psi^\top \tilde J\Psi .
\]
Orthogonal invariance gives \(\|\tilde J\|_F=\|M\|_F\), and
\(\Psi^\top \tilde B_J\Psi=B_\Psi+M\). Hence \eqref{eq:P1} is equivalent to the
following problem in different coordinates:
\begin{equation}\label{eq:P2}
\begin{aligned}
    \min_{\Psi,M}\qquad &\|M\|_F^2\\
    \text{s.t.}\qquad
    &\Psi^\top\Psi=I,\quad M^\top=-M,\\
    &\Re\!\left[\sigma_i(B_\Psi+M)\right]=\gamma_\star,\quad i=1,\ldots,d .
\end{aligned}
\tag{P2}
\end{equation}
Given a feasible pair \((\Psi,M)\), the corresponding perturbation is
\[
    J=S^{-1/2}\Psi M\Psi^\top S^{-1/2}.
\]
Conversely, every feasible \(J\) in \eqref{eq:P1} gives feasible coordinates
\((\Psi,M)\) in \eqref{eq:P2} by setting \(M=\Psi^\top S^{1/2}JS^{1/2}\Psi\).

\begin{rmk}
The formulation \eqref{eq:P2} is only a change of coordinates from \eqref{eq:P1}, so the two
problems have the same feasible set and the same optima.
\end{rmk}

\paragraph{Optimal value and solution set.}
The form of \eqref{eq:P2} makes the attained optimum explicit.
\begin{thm}
\label{thm:P1-P2-relation}
Let $S = S^\top \succ0$ and define $\gamma_\star=\mathrm{Tr}(S)/d$. There exists an orthonormal basis
\(\Psi=[\psi_1,\ldots,\psi_d]\) satisfying
\[
    \psi_k^\top S\psi_k=\gamma_\star,
    \qquad k=1,\ldots,d .
\]
For any such basis, define \(B=\Psi^\top S\Psi\) and let \(M_\star\) be the
skew-symmetric matrix
\[
    (M_\star)_{jk}
    =
    \begin{cases}
        B_{jk}, & j<k,\\
        0, & j=k,\\
        -B_{jk}, & j>k.
    \end{cases}
\]
Equivalently, $B+M_\star$ is upper triangular with diagonal entries $\gamma_\star$.
Let
\[
    \tilde J_\star=\Psi M_\star\Psi^\top,
    \qquad
    J_\star=S^{-1/2}\tilde J_\star S^{-1/2}.
\]
Then \((\Psi,M_\star)\) is feasible for \eqref{eq:P2}, i.e., \(J_\star\in\mathcal J_\mathrm{s}\), and
\[
    \min(\mathrm{P2})
    =
    \min(\mathrm{P1})
    =
    \mathrm{Tr}(-J_\star S J_\star S)
    =
    \|S\|_F^2-d\gamma_\star^2
    =
    \|S\|_F^2-\frac{[\mathrm{Tr}(S)]^2}{d}.
\]
Moreover, the global minimizers of \eqref{eq:P1} are exactly the matrices
\[
    \left\{
    S^{-1/2}\Psi M_\star(\Psi)\Psi^\top S^{-1/2}
    :
    \Psi^\top\Psi=I,\
    \operatorname{diag}(\Psi^\top S\Psi)=\gamma_\star\mathbf 1
    \right\}.
\]
\end{thm}
The proof is in \Cref{proof:thm:P1-P2-relation}.

We define the solution set to \eqref{eq:P1} as
\[
    \mathcal J_\mathrm{se}
    :=
    \operatorname*{argmin}_{J\in\mathcal J_\mathrm{s}}\;\mathrm{Tr}(-JSJS).
\]
\Cref{algo:algorithm} provides a constructive procedure to obtain an element of $\mathcal J_{\mathrm{se}}$.
It is inspired by the constructive approach of
\citet{Lelivre2013}, whose goal is to build an orthonormal basis
$\Psi=(\psi_1, \dots, \psi_d)$ satisfying $\psi_k^\top S\psi_k = \gamma_\star$.
This condition ensures that the resulting perturbation can satisfy the
optimal spectral constraint. However, this construction is aimed at spectral optimality alone, whereas our goal is to select, among spectrally optimal perturbations,
one that also optimizes the ESJD$^W$ discretization-error proxy. Our algorithm
thus modifies the previous constructive procedure in two respects.

First, instead of repeatedly selecting directions with maximal and minimal values of $(\psi_i,S\psi_i)$, we use a threshold-based selection rule.
This avoids discontinuous changes in the selected directions when $S$ is slightly perturbed, which is useful in adaptive settings where the precision matrix is estimated and updated online; see Section~\ref{sec:adaptive_FIM}.

Second, after constructing such a basis $\Psi$, spectral optimality still leaves degrees of freedom in the skew-symmetric
perturbation. These degrees of freedom do not affect the spectral
constraint, but they do affect the discretization error proxy.
We resolve this ambiguity by choosing the triangular perturbation in \Cref{thm:P1-P2-relation}, namely
\[
    M_{jk}=\mathrm{sign}(k-j)\psi_j^\top S\psi_k,\quad j\neq k,
    \qquad M_{kk}=0.
\]
This choice preserves spectral optimality while attaining the minimum of $\mathcal{E}_2(J)$.

\begin{algorithm}
\caption{Construct $J \in \mathcal{J}_{\mathrm{se}}$ satisfying spectral optimality and minimizing ESJD$^W$.\label{algo:algorithm}}
\begin{algorithmic}
    \STATE \textbf{Input:} Matrix $\mathcal F \succ 0$ (for Gaussian target $\mathcal F =S$)
    \STATE Set $\gamma_{\mathcal F}=\mathrm{Tr}(\mathcal F)/d$.
    \STATE \textbf{Initialize:} Orthonormal basis $\Psi:=(\psi_1, \dots, \psi_d)$.
    \FOR{$n = 1$ to $d-1$}
        \IF{there exists $r\geq n$ such that $\psi_r^\top\mathcal F\psi_r=\gamma_{\mathcal F}$}
            \STATE Swap $\psi_n \leftrightarrow \psi_r$.
        \ELSE
            \STATE Find indices $p, q \geq n$ such that:
            \[
                ( \psi_p, \mathcal F \psi_p ) > \gamma_{\mathcal F},
                \quad
                ( \psi_q, \mathcal F \psi_q )< \gamma_{\mathcal F} .
            \]
            \STATE Compute $\theta^\star \in (0, \pi/2)$ such that $\psi_{\theta^\star} = \cos(\theta^\star)\psi_p + \sin(\theta^\star)\psi_q$ satisfies:
            \[
                ( \psi_{\theta^\star}, \mathcal F \psi_{\theta^\star} ) = \gamma_{\mathcal F}
            \]
            \STATE Set $\psi_p \gets \psi_{\theta^\star}$ and swap $\psi_n \leftrightarrow \psi_p$.
            \STATE Apply Gram--Schmidt to $(\psi_{n+1}, \dots, \psi_d)$ to restore orthonormality.
        \ENDIF
    \ENDFOR
    \STATE Let $\Psi = (\psi_1, \dots, \psi_d)$ be the resulting orthonormal basis.
    \STATE Construct skew-symmetric $M \in \mathbb{R}^{d \times d}$:
    \[
        M_{jk} = \mathrm{sign}(k - j) \cdot \psi_j^\top \mathcal F\psi_k, \quad \forall j \neq k, \qquad M_{kk} = 0.
    \]
    \STATE \textbf{Return:} $J = \mathcal F^{-1/2} \Psi M \Psi^\top \mathcal F^{-1/2}$.
\end{algorithmic}
\end{algorithm}

Returning to \Cref{fig:relation_different_target}, the orange points show perturbations $J\in\mathcal J_{\mathrm{se}}$ generated by \Cref{algo:algorithm}. These perturbations, which minimize $\mathcal E_2$ by construction, simultaneously achieve low values of $\mathcal E_1$ and $\mathcal E_3$, confirming the equivalence of \Cref{prop:proxy-equivalence}: for well-conditioned $S$, optimizing any single proxy effectively controls the others.

\subsection{Extension to non-Gaussian targets}
The constrained optimization problem~\eqref{eq:P1} is formulated for Gaussian
targets, where the precision matrix $S$ fully determines the geometry of the
distribution. For a general target $\pi$, we can still ask what the best
position-independent irreversible perturbation should be. A natural way to
characterize the target distribution for this purpose is through its \textit{Fisher information matrix}:
\begin{equation}
    \mathcal{F} \;=\; \mathbb{E}_{\pi}\!\left[\nabla\log\pi(X)\,\nabla\log\pi(X)^\top\right].\label{Eq:FisherGeneral}
\end{equation}
This choice is principled in two ways: first,  $\mathcal{F}$ captures the average local curvature of $\log\pi$, and for a Gaussian with precision $S$ it recovers $\mathcal{F}=S$ exactly.
Second, the weighted ESJD for general $\pi$ is a function of $\mathcal{F}$:
\begin{align*}
    \mathrm{ESJD}^W(X_t,h)
    :&= \mathrm{Tr} \, \mathbb{E}\big[ W^{1/2}(X_{t+h} - X_t)(X_{t+h} - X_t)^\top W^{1/2}\big] \\
    &=h^2\mathrm{Tr}\!\left(W(I+J)\mathcal F(I+J)^\top\right)+2h\mathrm{Tr}[W].
\end{align*}
Setting $W = \mathcal F$ makes the cross terms vanish. Writing
\(\gamma_{\mathcal F}=\mathrm{Tr}(\mathcal F)/d\), minimizing this $\mathcal{F}$-weighted ESJD over $J \in \mathcal{J}_s(\mathcal{F})$ yields
the following optimization problem:
\begin{equation}\label{eq:P3}
\begin{aligned}
    &\min_{J = -J^\top} \qquad-\mathrm{Tr}\ \left( J \mathcal F J \mathcal F\right)\\
    &\text{s.t.}\qquad \qquad
    \Re[\sigma_i\left((I+J)\mathcal{F}\right)] = \gamma_{\mathcal F},\qquad\forall i.
\end{aligned}
\tag{P3}
\end{equation}
This is precisely the Gaussian problem \eqref{eq:P1} with $S$ replaced by $\mathcal F$,
so the same closed-form solution and \Cref{algo:algorithm} apply directly. We note that the Fisher information matrix~\eqref{Eq:FisherGeneral} is also used by
\citet{Ti2023} to construct optimal position-independent \emph{reversible}
perturbations; here, we use the same target summary for the irreversible setting.

\subsubsection{When does the Fisher-based approach work?}
Since $\mathcal{F}$ averages local geometry over the entire state space, it is most informative when this geometry is sufficiently consistent.
We now analyze two regimes where design based on the global Fisher information matrix $\mathcal{F}$ is most effective and describe how this matrix can be estimated adaptively in practice.

\paragraph{Near-Gaussian targets.}
We first consider target distributions $\pi \propto \exp(-V)$ that are close to Gaussian, in the sense that $V$ is a small perturbation of a quadratic form. Write $\mathcal F_\pi$ for the Fisher information matrix \eqref{Eq:FisherGeneral}  of such a \(\pi\). The natural question is whether an irreversible perturbation induced by $J^* \in \mathcal J_\mathrm{s}(\mathcal F_\pi)$, i.e., designed from $\mathcal{F}_\pi$, will still accelerate the convergence of Langevin dynamics for $\pi$. The next result shows that, under a diagonalizability assumption, this spectral acceleration is indeed stable to perturbations.

\begin{thm}[Spectral gap stability]\label{thm:spectral-gap-stability}
Let $\pi\propto \exp(-V)$ and let $V$ take the form
\begin{equation*}
    V(x) = \frac{1}{2}x^\top Sx + r(x),
\end{equation*}
with $S = S^\top \succ0$, $r \in C^2(\mathbb R^d)$, \(\nabla r(0)=0\), and
\[
    \delta := \sup_{x \in \mathbb R^d} \|\nabla^2 r(x)\|_{2} < \infty,
\]
where (without loss of generality) we have centered the coordinates $x$. Then
\[
    \mathcal F_\pi=S+\mathbb E_\pi[\nabla^2 r(X)],
    \ \ \text{and} \
    \|\mathcal F_\pi-S\|_2\leq \delta .
\]
Now let \(J^* \in \mathcal J_\mathrm{s}(\mathcal F_\pi)\), the set of skew-symmetric matrices defined in~\eqref{eq:J_spec} with \(S\) replaced by \(\mathcal F_\pi\). Assume that the drift matrix \((I+J^*)\mathcal F_\pi\) is diagonalizable over \(\mathbb C\).
Assume also
$2\delta < c_0 \sigma_{\min}(\mathcal F_\pi)$ for a constant $c_0 = c_0(\mathcal F_\pi,J^*) \in (0,1)$.
Then the spectral gap of the nonreversible generator $\mathcal{L}_{J^*} f
    =
    \Delta f-(I+J^*)\nabla V\cdot\nabla f$, denoted by $\gamma(J^*)$, satisfies
\begin{equation*}
    \gamma(J^*) \geq \frac{\mathrm{Tr}(\mathcal F_\pi)}{d} - 2C\delta,\qquad \text{or equivalently,}\qquad  \gamma(J^*) \geq \frac{\mathrm{Tr}(S)}{d} - (2C+1)\delta ,
\end{equation*}
where $C = C(\mathcal F_\pi,J^*,d) > 0$.

\end{thm}
The proof is provided in \Cref{app:thm:spectral-gap-stability}.

\begin{rmk}
This theorem is a stability statement for diagonalizable drifts \((I+J^*)\mathcal F_\pi\) produced by \(J^* \in \mathcal J_\mathrm{s}(\mathcal F_\pi)\), which also can be thought of as spectrally optimal perturbations for a Gaussian \(\mathcal N(0,\mathcal F_\pi^{-1})\). It should not be read as a guarantee for every $J^\ast$ that additionally solves \eqref{eq:P3}, as the drift matrix produced by such $J^\ast$ might not be diagonalizable over $\mathbb{C}$.
\end{rmk}

\begin{rmk}[Fisher estimation error]
The matrix $\mathcal F_\pi$ is often replaced by an approximation $\widehat{\mathcal F}\succ0$. Set $\eta=\|\widehat{\mathcal F}-\mathcal F_\pi\|_2$ and rewrite
\[
    V(x)=\frac12 x^\top \widehat{\mathcal F}x+\widehat r(x),
    \qquad
    \nabla^2\widehat r(x)=\nabla^2 r(x)+S-\widehat{\mathcal F}.
\]
Then $\sup_x\|\nabla^2\widehat r(x)\|_2\leq 2\delta+\eta$. The proof of \Cref{thm:spectral-gap-stability}
relies on properties of the quadratic term of $V$, so the same argument applies with the estimate $\widehat{\mathcal F}$.
Hence, if $\widehat{J}^* \in \mathcal J_\mathrm{s}(\widehat{\mathcal F})$, the drift matrix $(I+\widehat{J}^*)\widehat{\mathcal F}$ is diagonalizable, and $2\delta+\eta$ is sufficiently small relative to $\sigma_{\min}(\widehat{\mathcal F})$, the same argument gives
\[
    \gamma(\widehat{J}^*) \geq \frac{\mathrm{Tr}(\widehat{\mathcal F})}{d} - C(2\delta+\eta).
\]
\end{rmk}

\paragraph{Adaptive estimation of the Fisher information matrix.}\label{sec:adaptive_FIM}
In practice, the Fisher information matrix $\mathcal F$ is typically unknown and must be estimated during sampling.
We employ a streaming estimator that updates the matrix online, using score observations:
\begin{equation}\label{eq:adapted-fisher}
\begin{aligned}
    \text{Initialization:} \quad &\widehat{\mathcal{F}}_0 = I_d, \\
    \text{Iteration:} \quad &\widehat{\mathcal{F}}_n =
        \widehat{\mathcal{F}}_{n-1} + \frac{1}{K+n}\left( s_n s_n^\top - \widehat{\mathcal{F}}_{n-1} \right), & n \geq 1,
\end{aligned}
\end{equation}
where $s_n = \nabla\log\pi(X_n)$ is the score at iteration $n$ and $K>0$ controls the transition from initialization to online updates. This estimator is consistent and numerically stable, as illustrated in the numerical experiments of \Cref{sec:numerical}.

\paragraph{Aligned local geometries.}
Beyond the near-Gaussian regime, the Fisher information matrix can still be informative in cases where multiple modes share similar local geometry, even though obtaining formal guarantees in this setting is challenging. When the target distribution has multiple modes with similar shapes and orientations, the global Fisher information matrix aggregates their local curvature information in a coherent manner.
\Cref{fig:fisher_geometry} illustrates this phenomenon in two dimensions: when modes are misaligned (left), their local contributions cancel and $\mathcal{F}$ loses directional information; when modes are geometrically aligned (middle and right), $\mathcal{F}$ reflects their shared local structure.
\begin{figure}[htbp]
\centering
\begin{minipage}[t]{0.3\textwidth}
    \centering
    \begin{subfigure}[b]{\textwidth}
        \includegraphics[width=\textwidth]{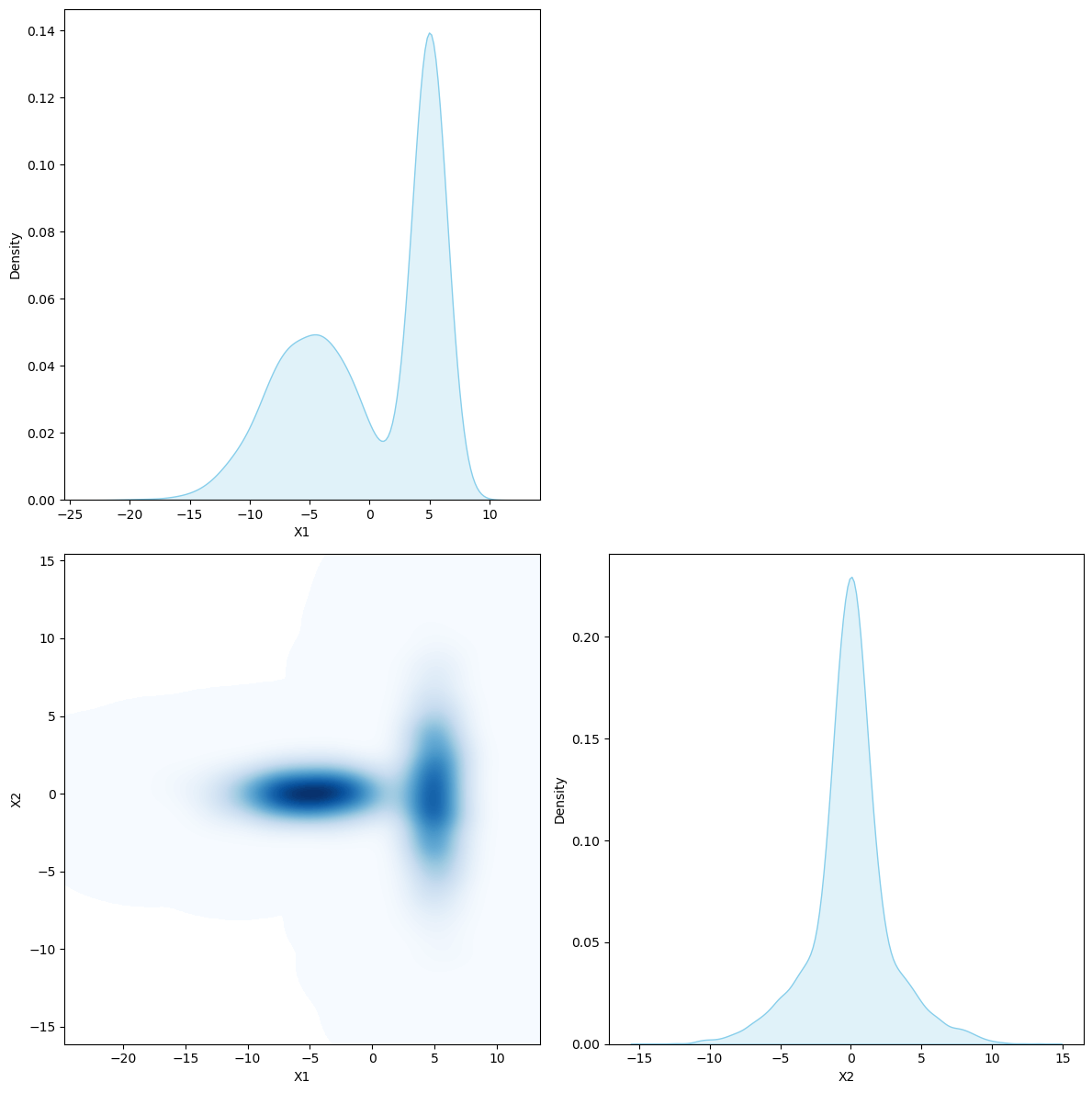}
    \end{subfigure}
\end{minipage}
\begin{minipage}[t]{0.3\textwidth}
    \centering
    \begin{subfigure}[b]{\textwidth}
        \includegraphics[width=\textwidth]{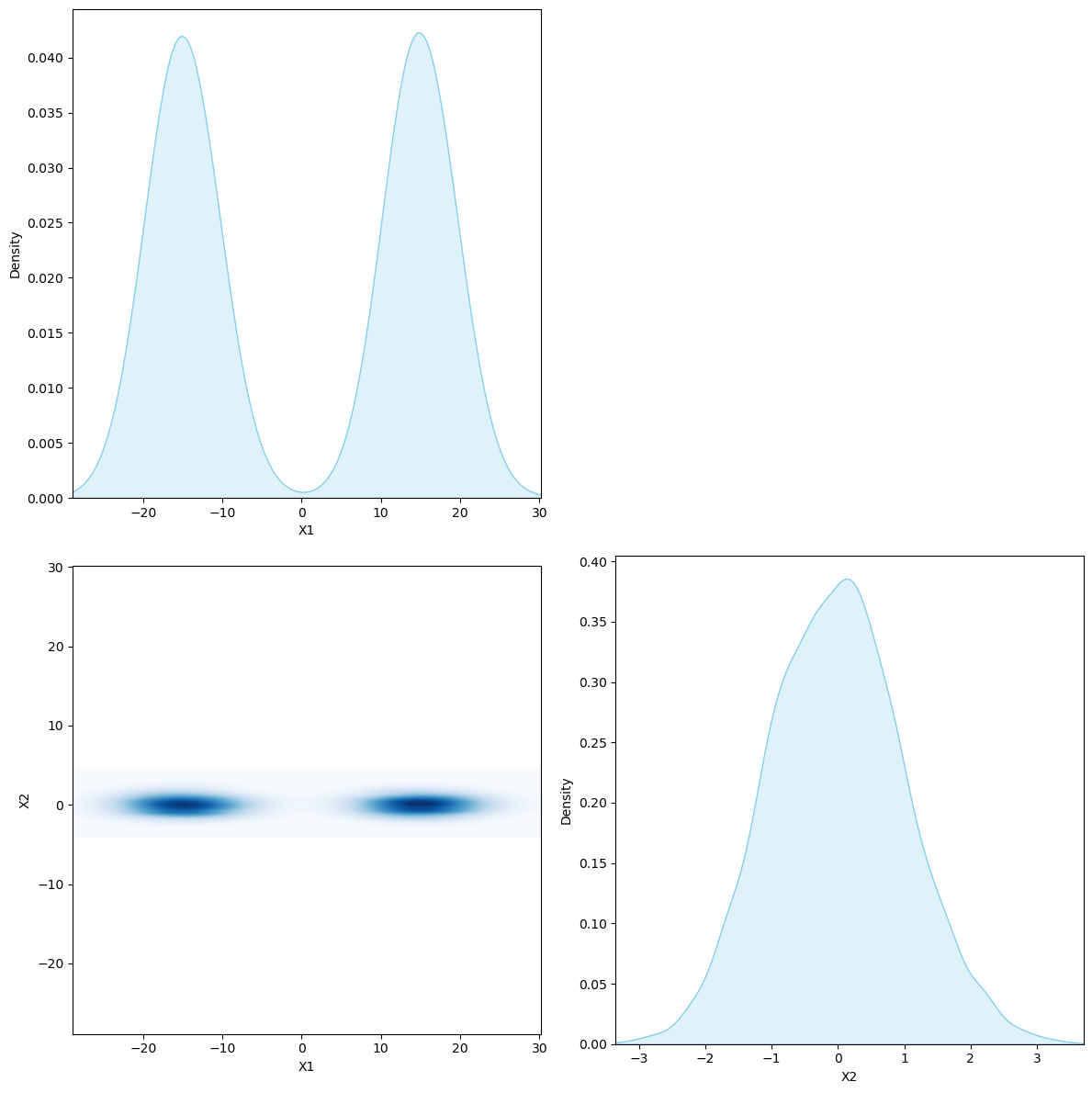}
    \end{subfigure}
\end{minipage}
\begin{minipage}[t]{0.3\textwidth}
    \centering
    \begin{subfigure}[b]{\textwidth}
        \includegraphics[width=\textwidth]{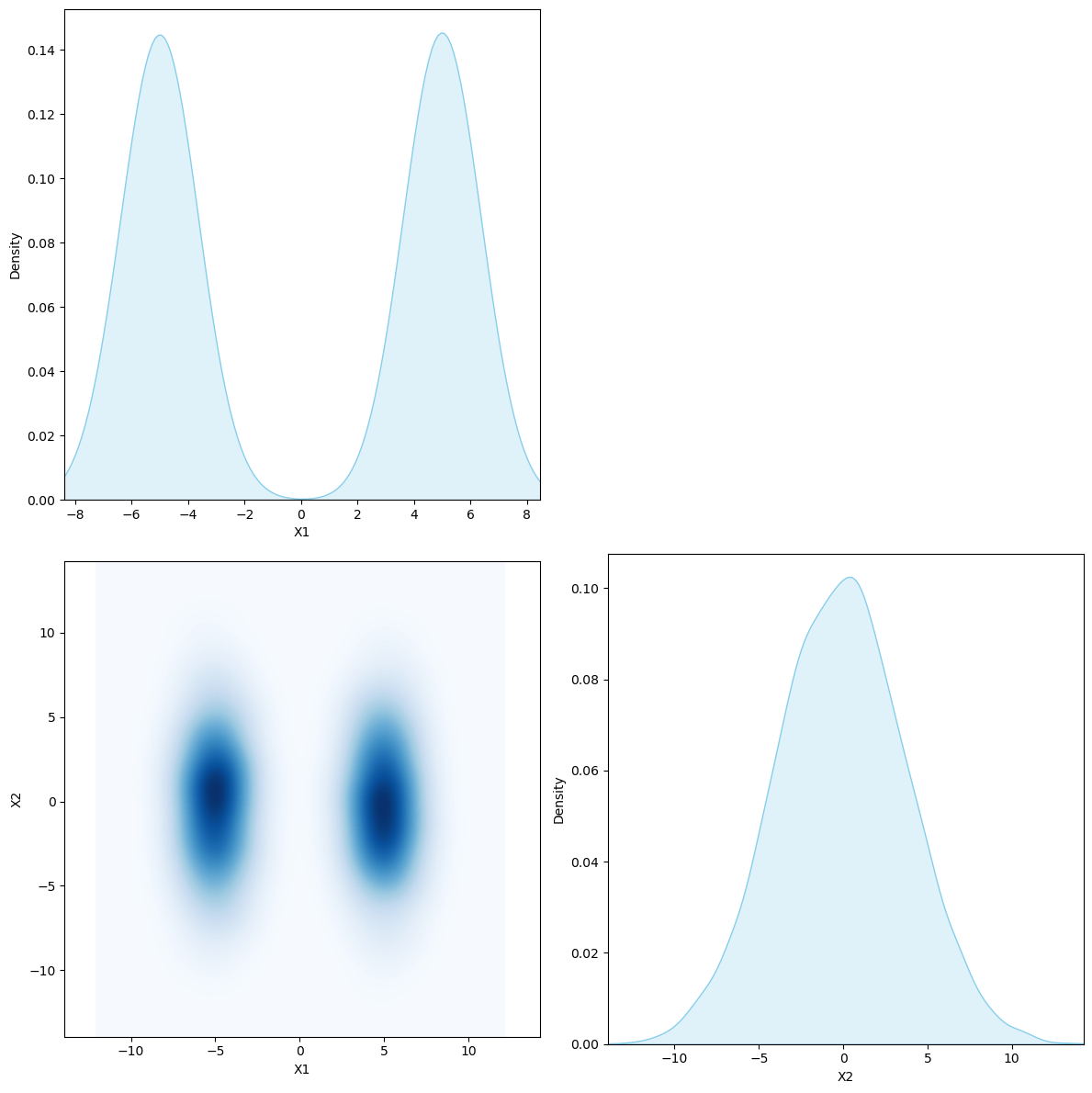}
    \end{subfigure}
\end{minipage}
\caption[Effect of mode geometry on Fisher matrix effectiveness]{Effect of mode geometry on Fisher information matrix effectiveness. Each panel shows a bimodal distribution and its corresponding global Fisher information matrix \eqref{Eq:FisherGeneral}. Left: misaligned modes yield $\mathcal{F}_1 = \text{diag}(0.5, 0.5)$; local geometries cancel and $\mathcal{F}_1$ loses directional information. Middle and right: aligned modes yield $\mathcal{F}_2 = \text{diag}(0.07, 1)$ and $\mathcal{F}_3 = \text{diag}(1, 0.07)$, respectively; the global Fisher matrix reflects the shared local curvature.}
\label{fig:fisher_geometry}
\end{figure}

Together, these results suggest that the $\mathcal F$-based optimization is most effective for targets that are either (i) approximately Gaussian, or (ii) multi-modal with geometrically aligned modes. We will verify these guidelines empirically in \Cref{sec:numerical}.

\section{Numerical experiments}\label{sec:numerical}

In this section, we evaluate our method against several alternative irreversible perturbations on four benchmark problems: a non-isotropic Gaussian (\Cref{sec:Gaussian}), a Gaussian mixture (\Cref{sec:mix:of:Gaussian}), Bayesian logistic regression (\Cref{sec:BLR}), and a Bayesian formulation of independent component analysis (\Cref{sec:ICA}). For each problem, we compare MCMC estimates of representative observables against ground truth values obtained analytically or from high-accuracy MALA simulations.

\paragraph{Methods.} We run ULA with the six types of irreversible perturbations summarized in \Cref{tab:methods}. These comprise the unperturbed case ($J=0$) and five other choices of $J$.
\begin{table}[h]
\centering
\begin{tabular}{l p{12cm}}
\toprule
\textbf{Name} & \textbf{Description} \\
\midrule
\texttt{unperturbed} & $J=0$. \\[3pt]

\texttt{rand-S/M/L} & $J$ with random entries and rescaled such that
$\|J\|_F \in \{\frac12 C, C, \frac32 C\}.$ \\[3pt]

\texttt{spec} & $J \in \mathcal{J}_s(\mathcal{F})$, i.e., with generalized spectral optimality. \\[3pt]

\texttt{spec-E} & $J \in \mathcal{J}_{\mathrm{se}}(\mathcal{F})$, i.e., with generalized spectral optimality and minimal ESJD$^{\mathcal{F}}$. \\
\bottomrule
\end{tabular}
\begin{minipage}{\linewidth}
\smallskip
\end{minipage}
\caption{Summary of perturbations in our numerical experiments.}
\label{tab:methods}
\end{table}
In this table, $C$ is the  Frobenius norm of $J$ from \texttt{spec-E}, averaged over randomness in Algorithm~\ref{algo:algorithm}. $\mathcal F$ is given by \eqref{Eq:FisherGeneral}. Note that neither \texttt{spec} nor \texttt{spec-E} is unique: both depend on the random initial basis supplied to the constructive algorithm, so each run produces a different $J$ from the solution sets $\mathcal J_{\mathrm{s}}$ or $\mathcal J_{\mathrm{se}}$.

\paragraph{Evaluation metrics.}
For an observable $\varphi:\mathbb{R}^d \to \mathbb{R}$,
the ULA estimator of $\mathbb{E}_\pi [\varphi(X)]$ based on $K$ steps is $\bar{\varphi}_K \;=\; \frac{1}{K} \sum_{k=0}^{K-1} \varphi(X_k)$.
We evaluate the performance of the estimator using bias and variance estimated from $M$ independent chains:
\begin{align*}
    \widehat{\mathrm{Bias}}(\bar{\varphi}_K) &= \frac{1}{M}\sum_{m=1}^{M}\bar{\varphi}_K^{(m)} - \mathbb{E}_\pi[\varphi(X)], \quad\quad
     \widehat{\mathrm{Var}}(\bar{\varphi}_K) = \frac{1}{M}\sum_{m=1}^{M}\left(\bar{\varphi}_K^{(m)}\right)^2 - \left(\frac{1}{M}\sum_{m=1}^{M}\bar{\varphi}_K^{(m)}\right)^2,
\end{align*}
where $\bar{\varphi}_K^{(m)}$ denotes the estimate from the $m^\mathrm{th}$ chain. The MSE of the estimator is then simply $\widehat{\mathrm{Bias}}(\bar{\varphi}_K)^2 + \widehat{\mathrm{Var}}(\bar{\varphi}_K) $.
The values of $K$ and of the ULA step size $h$ will be varied systematically in the experiments below.

The chains for the first two experiments are initialized randomly from the target distribution; those for the last two experiments are initialized at $\mathbf{0}$ and at a random diagonal sign matrix $\text{diag}(\pm1)$, with burn-in intervals of $0$ and $5000$ steps, respectively, discarded before computing any $\bar{\varphi}_K$.

In addition, we evaluate numerical stability by examining the probability that an estimator fails to converge, in the sense that it exceeds a prescribed threshold chosen to be significantly larger than the true value of $\mathbb{E}_\pi[\varphi(X)]$. In plots of MSE versus step size $h$, we exclude points for which the averaged estimates exceed this threshold, as this provides a sufficient (though not necessary) indication of divergence. Across the various experiments, we find that \texttt{spec}, \texttt{rand-M}, and \texttt{rand-L} are substantially more numerically unstable, a phenomenon that is particularly evident in Experiment 2. Accordingly, we classify these methods as ``discretization-sensitive,'' in contrast to the remaining ``discretization-robust'' methods.

\subsection{Non-isotropic Gaussian}\label{sec:Gaussian}
We begin with a zero-mean Gaussian target with covariance $S^{-1} \;=\;\text{diag}(1,4,16,64)$, providing a controlled setting where $\mathcal F = S$ is known exactly. Let the components of the state vector $X$ be denoted as $X = (X^{(1)}, X^{(2)}, X^{(3)}, X^{(4)})$.

\paragraph{Experiment 1: MSE versus step size.}
\Cref{fig:gaussian_fix_time} and~\Cref{fig:gaussian_fix_budget} compare the MSE, bias, and variance of estimators of the expected $\ell_1$ norm $\|X\|_1=\sum_{i=1}^4 |X^{(i)}|$
and the tail probability $P(X^{(4)}>16)$ ,
across step sizes $h \in [0.02, 0.4]$, with either fixed simulation time ($T = hK =10\,000$) \textit{or} fixed computational budget ($K = 10^5$).
We also test our method \texttt{spec-E} with adaptive estimation of the Fisher information matrix as described in \Cref{sec:adaptive_FIM}, which is labeled as \texttt{adaptive}.
Results for additional observables $\phi$ are reported in \Cref{app:numerical:iso-gaussian}.

The methods split into two groups. The discretization-unaware methods \texttt{spec}, \texttt{rand-M}, and \texttt{rand-L}, which do not involve $\text{ESJD}^W$ minimization, diverge for $h>0.05$ (red/orange curves in both figures).
Other methods, namely  \texttt{rand-S}, \texttt{spec-E}, and \texttt{unperturbed}, remain well-behaved across all $h$ (blue curves) with \texttt{spec-E} converging the most quickly, achieving MSE roughly an order of magnitude lower than the other stable methods. Note that the $J$ of \texttt{rand-M} has exactly the Frobenius norm as that of \texttt{spec-E}, showing that stability is not only a matter of the scale of the perturbation.
Also notably, \texttt{adaptive} nearly matches the performance of the latter despite having no access to the true Fisher information matrix, estimating it on-the-fly from samples alone. The fact that the MSE of \texttt{adaptive} remains only marginally larger than that of \texttt{spec-E} demonstrates that our method is {robust} to Fisher information matrix misspecification---an issue that arises in practice when the target distribution is not analytically tractable.

The fixed-$T$ and fixed-$K$ experiments highlight complementary trade-offs. Under fixed simulation time (\Cref{fig:gaussian_fix_time}), \texttt{spec-E} achieves lower variance for all $h$ (right column) without sacrificing bias (middle column), reflecting its maximized spectral gap. Under fixed computational budget (\Cref{fig:gaussian_fix_budget}), increasing $h$ reduces variance through longer simulation time but increases bias through discretization error. The resulting U-shaped MSE curves (left column) show that  \texttt{spec-E} achieves the lowest MSE across a wide range of $h$, demonstrating robustness to step size selection.

\paragraph{Experiment 2: Proxy comparison}
\Cref{fig:adaptive-ellipse} compares the performance of \texttt{spec-E} with that of two alternative discretization-aware pertubations, obtained by solving \eqref{eq:P1} but with its objective \emph{replaced} by $\mathcal{E}_1$ or $\mathcal{E}_3$. Since the latter optimization problems are not analytically tractable, we approximate their minimizers by Monte Carlo search over $\mathcal{J}_s$. Using $\mathcal E_2$ yields the most accurate estimation, for most values of $h$. Improved performance over $\mathcal{E}_{ \{ 1,3 \} }$ may be due to two complementary factors: first, minimizing \(\mathcal E_2\) admits an explicit and numerically stable construction; second, \(\mathcal E_2\) itself seems to better correlate with discretization error and thus provides a more effective criterion for selecting \(J\) in this example.

\begin{figure}[htbp]
\centering
\begin{minipage}[t]{0.89\textwidth}
    \centering
    \includegraphics[width=\textwidth]{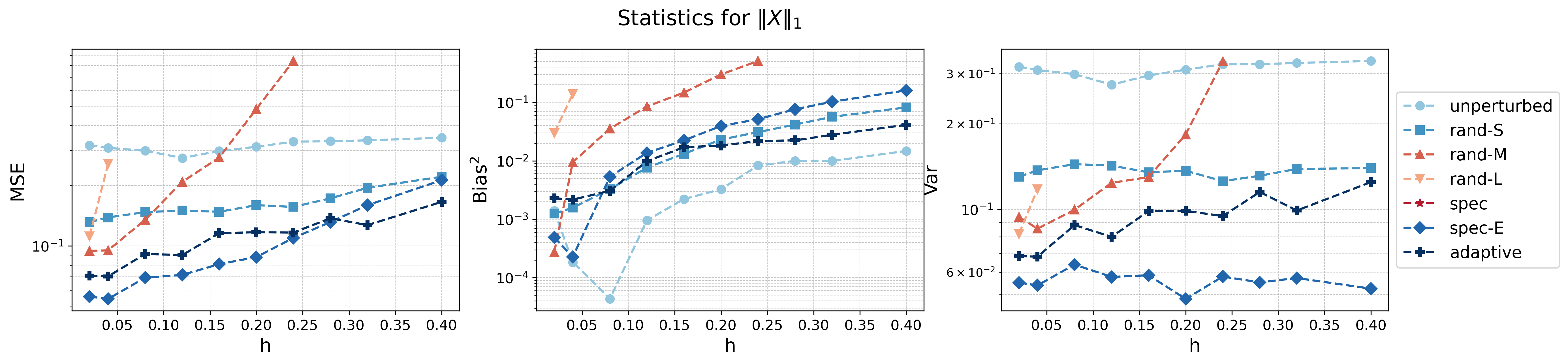}
    \includegraphics[width=\textwidth]{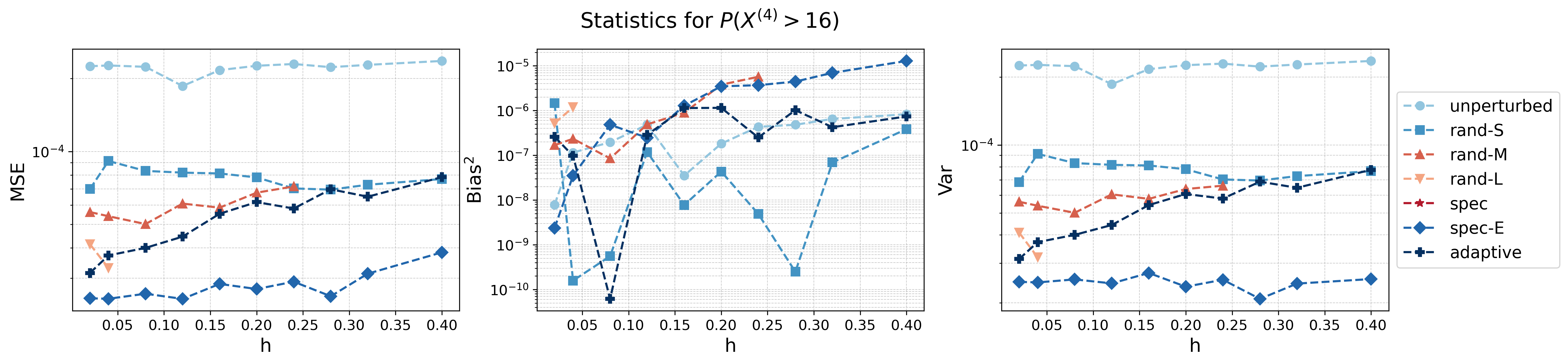}
\end{minipage}
\caption[Non-isotropic Gaussian, MSE vs. step size, fixed time]{Non-isotropic Gaussian, fixed simulation time $T = hK= 10{,}000$. Plots show MSE, squared bias, and variance for step sizes $h \in [0.02, 0.4]$ and two different observables. Statistics are computed using $M = 512$ replicated MCMC chains. Irreversible perturbations accelerate convergence but introduce additional bias and tend to be unstable unless they are kept small or, more explicitly, account for discretization effects via ESJD$^W$ minimization.}
\label{fig:gaussian_fix_time}
\end{figure}

\begin{figure}[htbp]
\centering
\begin{minipage}[t]{0.89\textwidth}
    \centering
    \includegraphics[width=\textwidth]{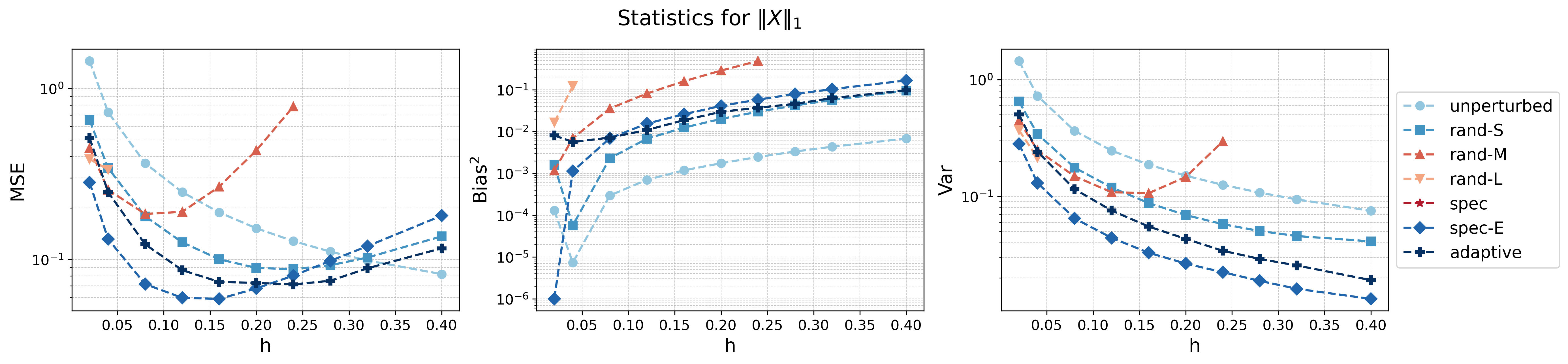}
    \includegraphics[width=\textwidth]{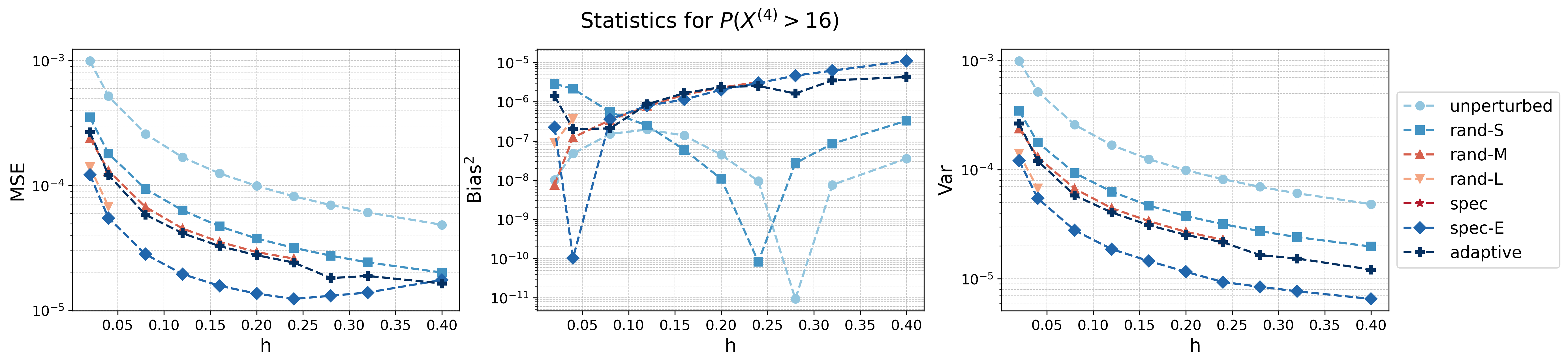}
\end{minipage}
\caption[Non-isotropic Gaussian, MSE vs. step size, fixed budget]{Non-isotropic Gaussian, fixed computational budget $K = 10^5$. Plots show MSE, squared bias, and variance for varying $h$ and two different observables. Statistics are computed using $M = 512$ replicated MCMC chains. Unstable methods (\texttt{spec}, \texttt{rand-M}, \texttt{rand-L}) diverge for $h>0.1$. Among stable methods, \texttt{spec-E} achieves the lowest MSE across all step sizes.}
\label{fig:gaussian_fix_budget}
\end{figure}

\begin{figure}[htbp]
    \centering
    \includegraphics[width=0.95\linewidth]{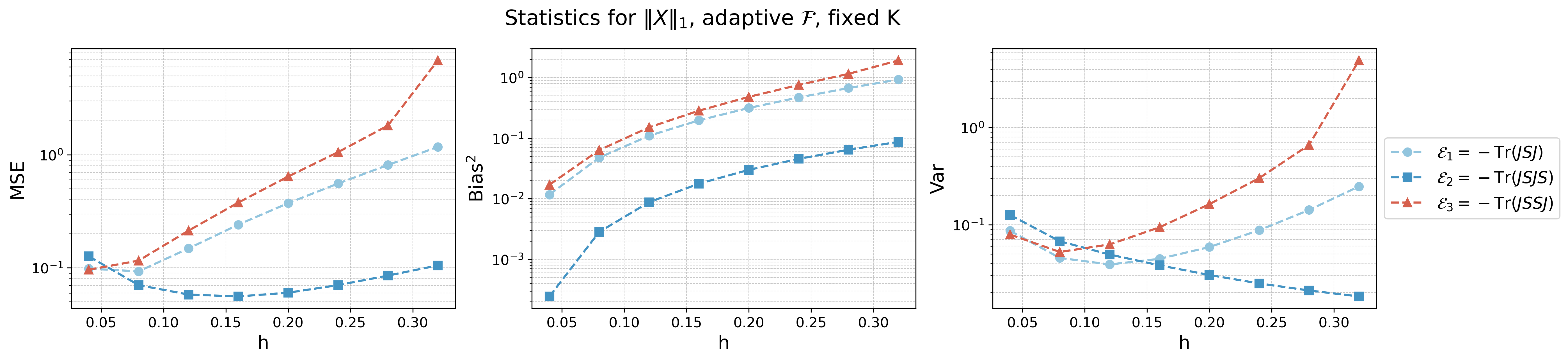}
    \includegraphics[width=0.95\linewidth]{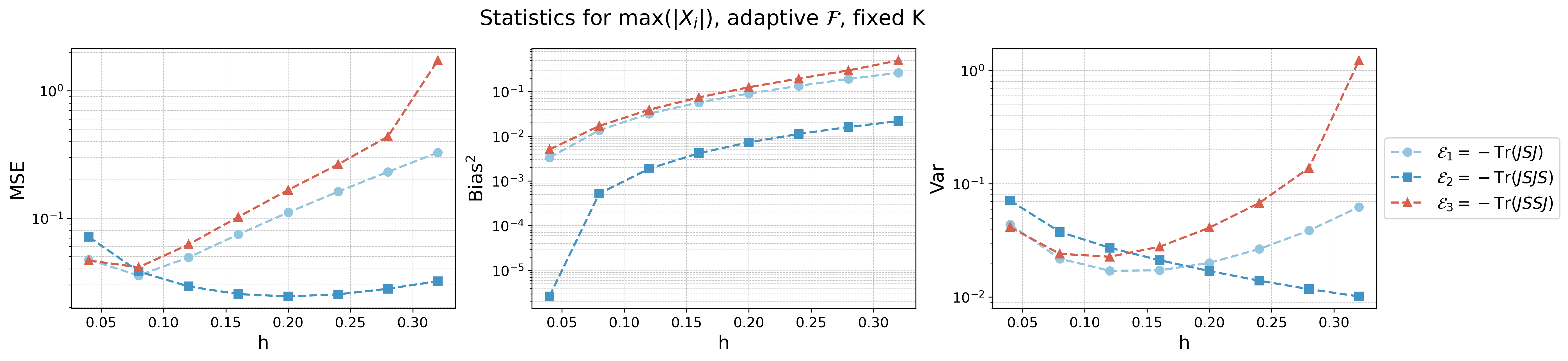}
    \caption{Non-isotropic Gaussian, fixed budget $K = 10^5$. Plots show MSE, squared bias, and variance of two different observables (top and bottom) for varying $h$.
    Statistics are computed using $M = 512$ replicated MCMC chains.
    Lines show the performance of \texttt{spec} perturbations additionally minimizing different objectives: $\mathcal{E}_1$, $\mathcal{E}_2$ (same as \texttt{spec-E}), and $\mathcal{E}_3$.}
    \label{fig:adaptive-ellipse}
\end{figure}

\subsection{Mixture of Gaussians}\label{sec:mix:of:Gaussian}

We next consider a random vector $X=(X^{(1)},X^{(2)},X^{(3)})\in\mathbb R^3$ whose target distribution is a mixture of five Gaussian components (\Cref{fig:multi-gaussian-density}).
The modes are separated along the first coordinate, so efficient exploration requires transitions between distant regions.
At the same time, the local geometry within each mode is similar, and the global Fisher information matrix remains aligned with this shared local structure.
This example therefore isolates a favorable multi-modal setting in which the Fisher information matrix provides an effective global summary of the target geometry, consistent with the discussion in \Cref{sec:adaptive_FIM}.
Further details on the problem setup are provided in \Cref{app:numerical:mix-gaussian}.

\paragraph{Experiment 1: Stability and trajectory behavior.}
\Cref{fig:multi-gaussian-ESJD-diverge} reports the number of divergent chains among 128 independent runs, where divergence is declared if the empirical estimator of $\E_\pi [\|X\|_1]$ or $\E_\pi [\max_i(|X^{(i)}|)]$ exceeds \(50\) after \(T=10{,}000\) iterations, a threshold chosen to be more than twice the corresponding true value.
The discretization-sensitive methods \texttt{spec}, \texttt{rand-M}, and \texttt{rand-L} become increasingly unstable as the step size grows: by $h=0.3$, nearly all chains diverge for \texttt{spec} and \texttt{rand-L}. In contrast, the discretization-robust methods \texttt{unperturbed}, \texttt{rand-S}, and \texttt{spec-E} rarely diverge.

The trajectory visualizations in \Cref{fig:multi-gaussian-trajectory} further explain these stability trends. The
\texttt{unperturbed} and \texttt{rand-S/M} chains often fail to visit all modes, while \texttt{rand-L/spec} can over-explore and produce unstable trajectories.
The properly regularized \texttt{spec-E} achieves the desired balance: it improves inter-modal exploration while controlling discretization-induced instability.

\paragraph{Experiment 2: MSE versus step size.}
\Cref{fig:multi-gaussian-estimator} reports MSE, bias, and variance for estimators of $\mathbb{E}_\pi[ \max_i  |X^{(i)}|  ]$
and $\Pr(X^{(1)}>20)$ under fixed simulation time $T = 4000$.
The pattern mirrors the Gaussian results: \texttt{spec-E} achieves the lowest MSE across a broad range of $h$, while \texttt{unperturbed/rand-S} mix poorly at small $h$ and \texttt{rand-M/rand-L/spec} accumulate bias at large $h$.
Additional statistics and fixed-budget experiments are in \Cref{app:numerical:mix-gaussian}.

\begin{figure}[htbp]
\centering
\begin{minipage}[t]{0.8\textwidth}
    \centering
    \begin{subfigure}[b]{\textwidth}
        \includegraphics[width=\textwidth]{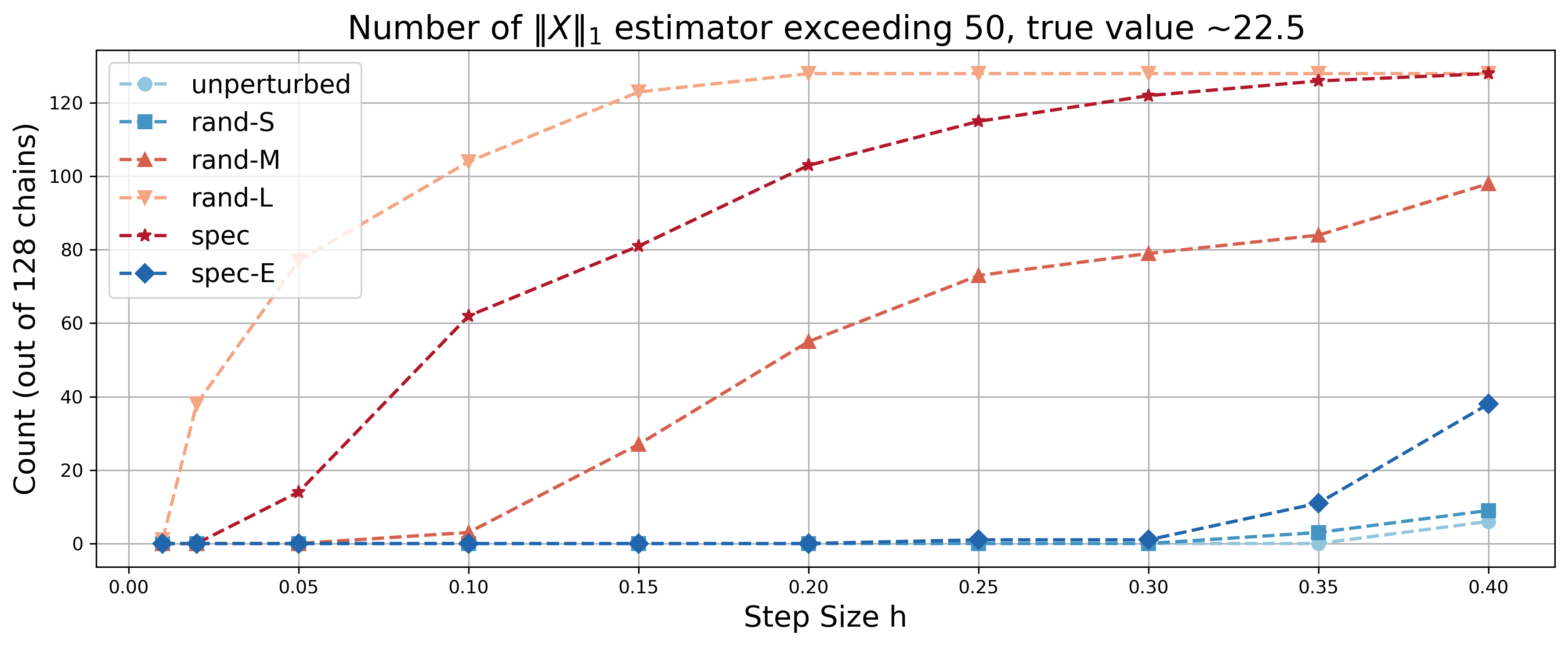}

        \includegraphics[width=\textwidth]{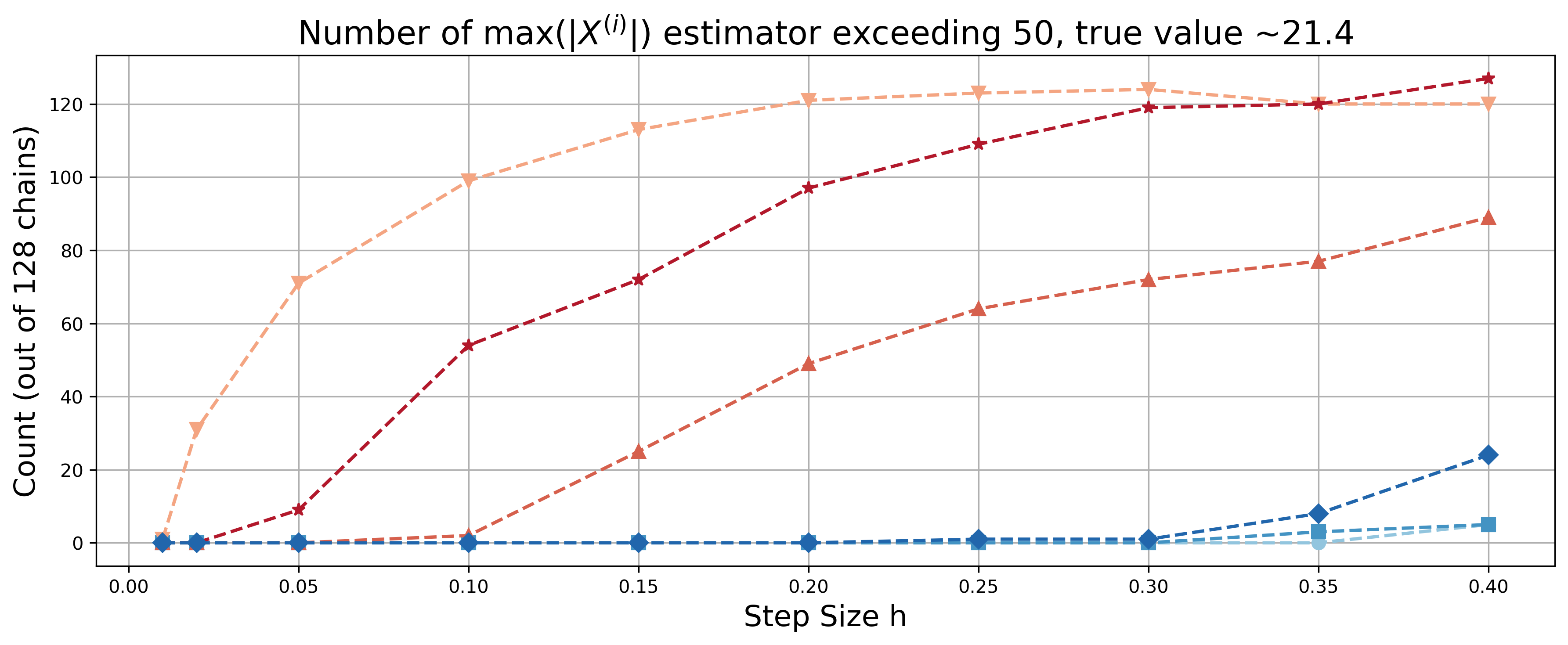}
    \end{subfigure}
\end{minipage}
\caption{Mixture of Gaussians example. Divergence frequency for estimators of the expectation of $\|X\|_1$ (top) and of maximum absolute entry $\max_i(|X^{(i)}|)$ (bottom), versus step size $h$. The methods (\texttt{spec}, \texttt{rand-L}) show frequent divergence; \texttt{spec-E} remains stable.}
\label{fig:multi-gaussian-ESJD-diverge}
\end{figure}

\begin{figure}[htbp]
\centering
\begin{minipage}[t]{\textwidth}
    \centering
    \begin{subfigure}[b]{\textwidth}
        \includegraphics[width=\textwidth]{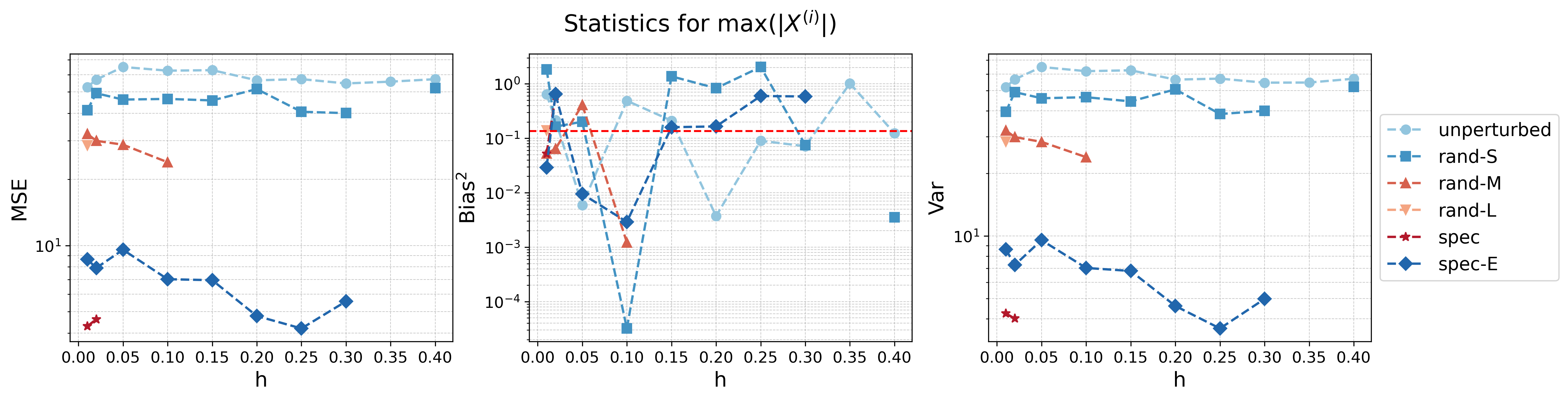}
        \includegraphics[width=\textwidth]{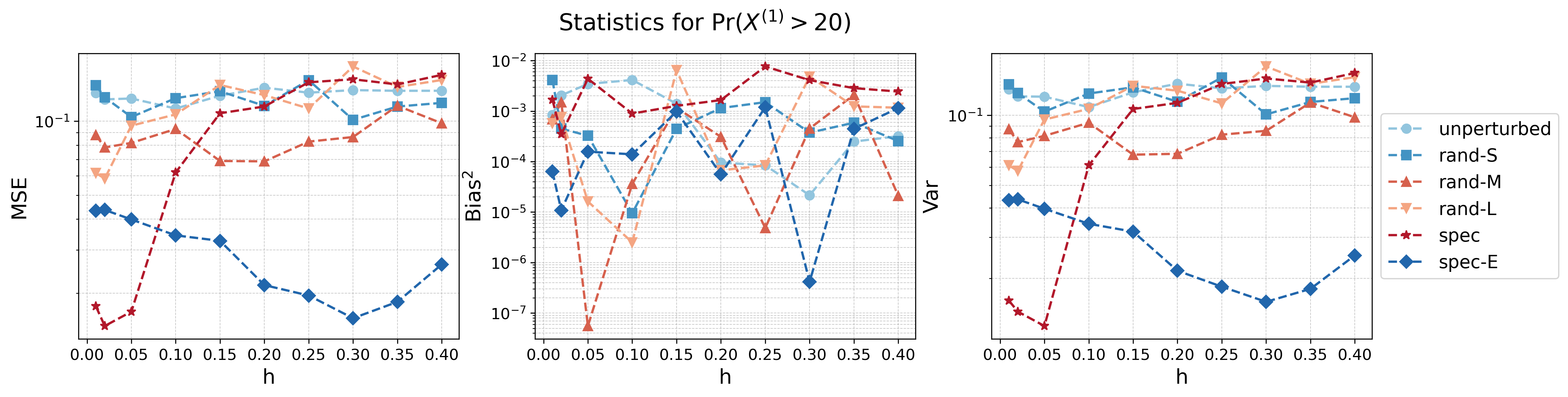}
    \end{subfigure}
\end{minipage}
\caption{Mixture of Gaussians, fixed simulation time $T = 4000$.
Plots show MSE, bias, and variance of estimators of $\E_\pi [\max_i( |X^{(i)}|)]$ (top) and $\Pr(X^{(1)} > 20)$ (bottom) for varying $h$. Statistics are computed using $M=128$ chains. \texttt{spec-E} achieves the lowest MSE across step sizes; \texttt{spec, rand-M/L} accumulate bias and quickly blow up as $h$ grows.}
\label{fig:multi-gaussian-estimator}
\end{figure}

\begin{figure}[htbp]
\centering
\begin{minipage}[t]{0.35\textwidth}
    \centering
    \begin{subfigure}[b]{\textwidth}
        \includegraphics[width=\textwidth]{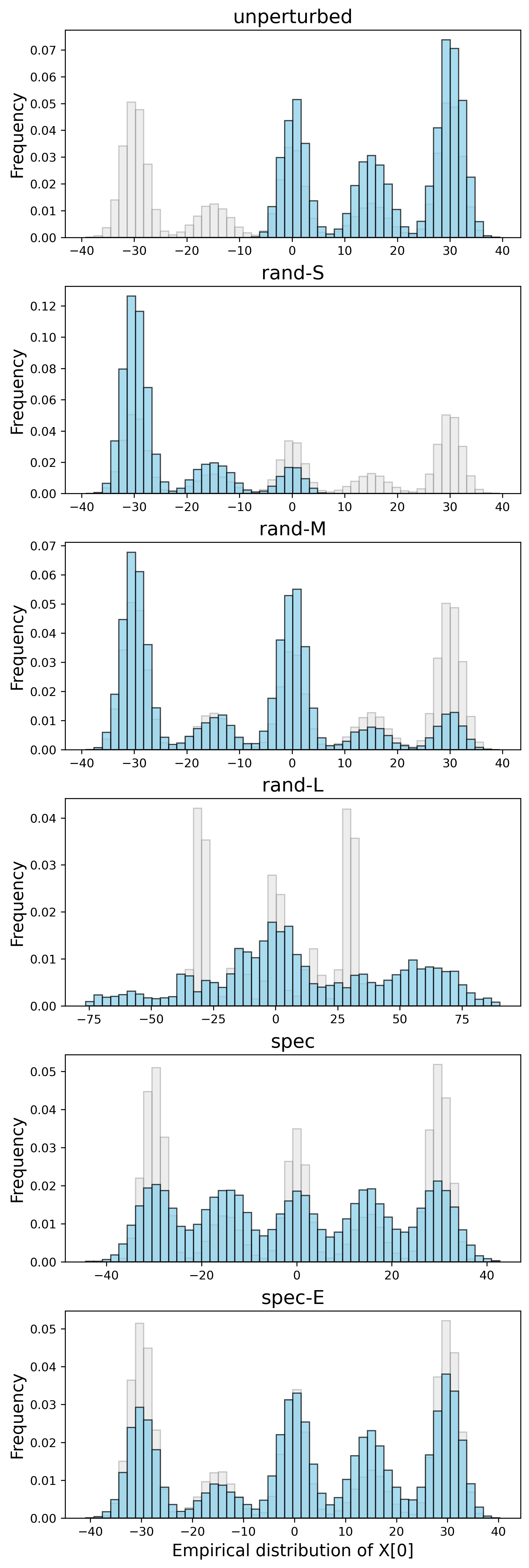}
    \end{subfigure}
\end{minipage}
\begin{minipage}[t]{0.35\textwidth}
    \centering
    \begin{subfigure}[b]{\textwidth}
        \includegraphics[width=\textwidth]{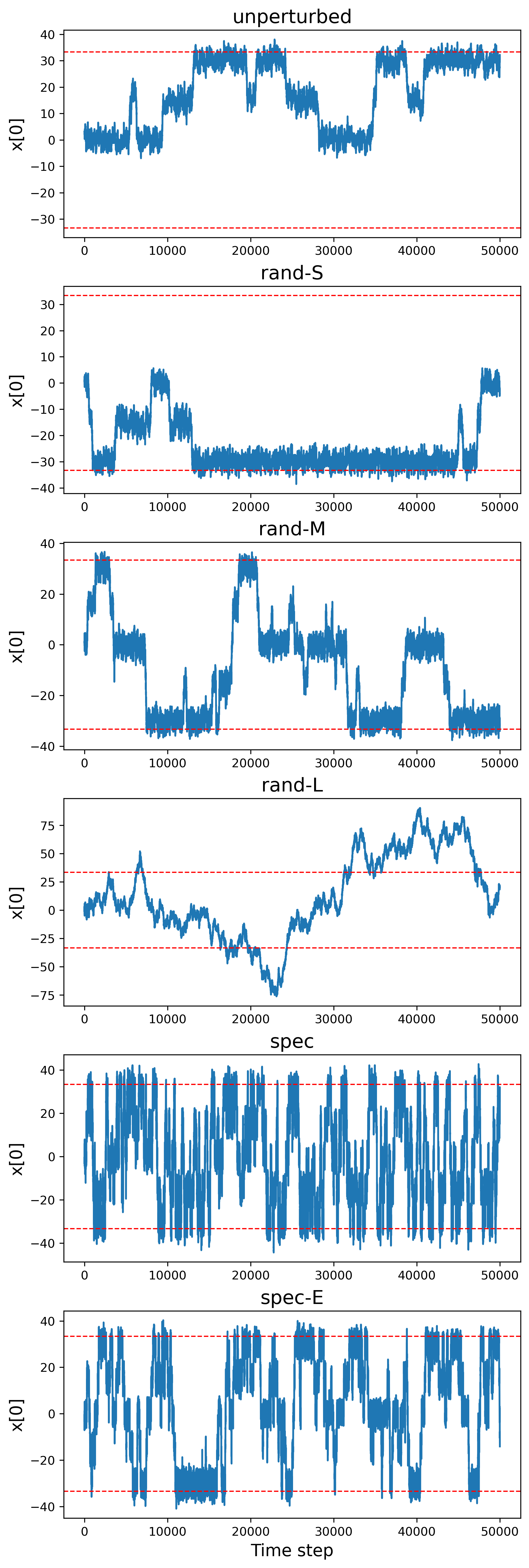}
    \end{subfigure}
\end{minipage}
\caption{Mixture of Gaussians: trajectory visualizations. Left: true marginal of the first coordinate (gray) versus a sample histogram (blue). Right: trajectory of a single representative chain along the first coordinate, with a centered 90\% interval marked by red dashed lines. \texttt{spec-E} visits all modes with balanced exploration.}
\label{fig:multi-gaussian-trajectory}
\end{figure}

\subsection{Bayesian logistic regression}\label{sec:BLR}
Next we apply our method to Bayesian logistic regression on the \texttt{german} dataset from the benchmark collection used by \citet{gershman2012nonparametric} and compiled by \citet{mika1999fisher} ($d = 20$ features, $N = 400$ observations).
Given data $\{ (Z_i, t_i) \}_{i=1}^N$ with $ Z = [z_1,\dots, z_N]\in\mathbb R^{d\times N}$, $t_i \in \{0, 1\}$, the log posterior for weights $X \in \mathbb{R}^d$ under a Gaussian prior $X\sim\mathcal N(0,\alpha^{-1} I)$ is
\begin{equation}
    \log\pi(X\mid Z) = -\frac{\alpha}{2}\|X\|^2 + \sum_{i=1}^N \left[ t_i z_i^\top X - \log(1 + \exp(z_i^\top X)) \right] + \mathrm{const}.
\end{equation}
Define the logistic function $\varphi(y) = (1+\exp(-y))^{-1}$. The gradient of the log-posterior is
\begin{equation}
    \nabla_{X} \log \pi(X \mid Z)
    = -\alpha X + \sum_{i=1}^N t_i z_i
      - \sum_{i=1}^N \varphi(z_i^\top X) \, z_i.
\end{equation}
The posterior of $X$ is unimodal but anisotropic (see \Cref{fig:blr_density} in \Cref{app:numerical:blr}).

For the numerical experiments, we use stochastic gradient Langevin dynamics \citep{Welling2011, teh2016consistency} and subsample at a rate of $n=10$ data points per likelihood computation, leading to the following drift
\begin{equation}
    \nabla_{X} \log \tilde \pi(X \mid Z)
    = -\alpha X + \frac{N}{n}\sum_{i=1}^n t_{\tau_i} z_{\tau_i}
      - \frac{N}{n}\sum_{i=1}^n \varphi(z_{\tau_i}^\top X) \, z_{\tau_i}
\end{equation}
where $\{\tau_i\}_{i=1}^n$ are the indices of the subsampled data.

\paragraph{Experiment: MSE versus step size.}\Cref{fig:blr_statistics} reports MSE, bias, and variance of estimators of  $\E_\pi[ |X^{(5)}|]$ and $P(X^{(5)}>2)$ under fixed computational budget $K=100,000$; estimator is estimated using $M=128$ independent chains.
Unlike the Gaussian targets considered above, this posterior has asymmetric tail behavior: along some directions, the density has flatter tails on one side and steeper decay on the other.
Thus, while $\mathcal F$ no longer exactly characterizes the target geometry, it still captures directions in which the log-density changes rapidly.

All methods remain numerically stable in this example, unlike in the Gaussian experiments where several unregularized methods diverged. Nevertheless, the unregularized spectral perturbation \texttt{spec} performs poorly.
The random perturbations and \texttt{unperturbed} exhibit moderate MSE, limited by higher variance.
\texttt{spec-E} retains much of the variance reduction from other irreversible perturbations while controlling bias in steep directions, and therefore achieves the lowest MSE across step sizes. Results for additional statistics and under fixed simulation time $T$ are in \Cref{app:numerical:blr}.

\begin{figure}[htbp]
\centering
\begin{minipage}[t]{0.85\textwidth}
    \centering
    \begin{subfigure}[b]{\textwidth}
        \includegraphics[width=\textwidth]{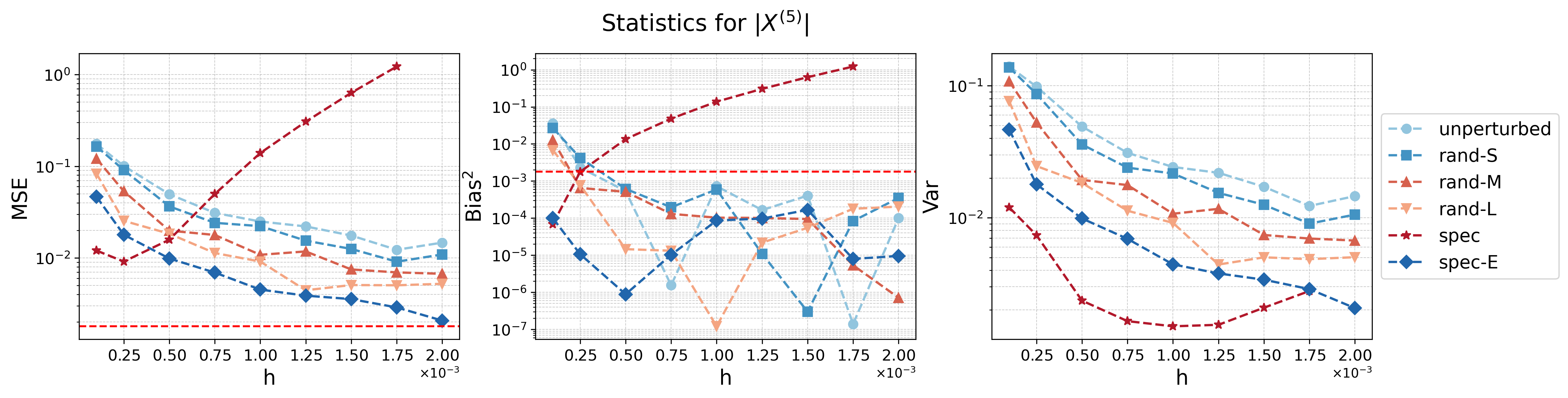}
        \includegraphics[width=\textwidth]{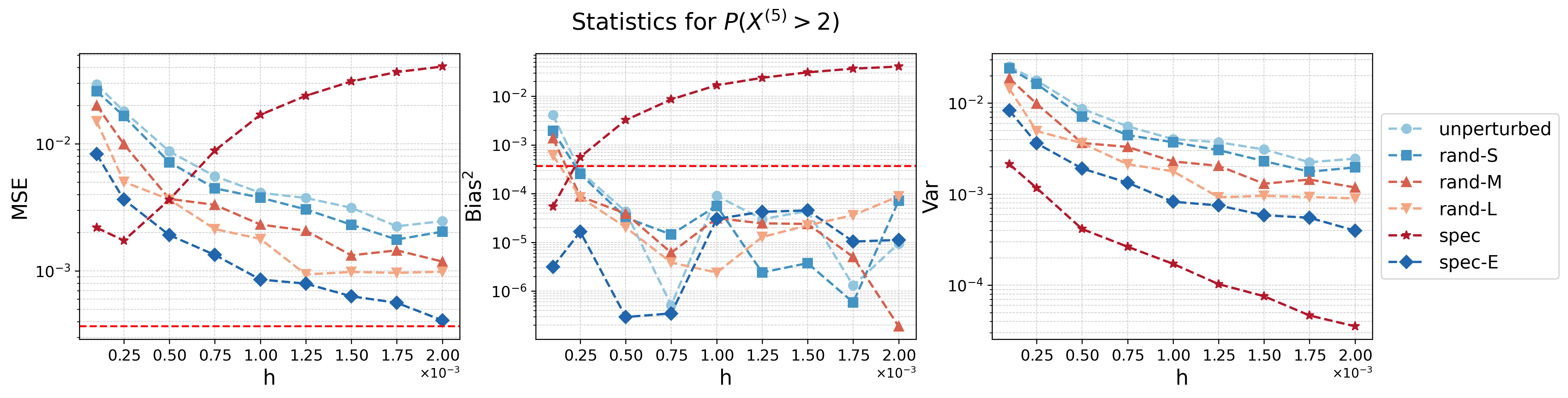}
    \end{subfigure}
\end{minipage}
\caption[Bayesian logistic regression, MSE versus step size]{Bayesian logistic regression, fixed simulation time $T = 2000$. Plots show the MSE, bias, and variance of estimators of $\E_\pi[ |X^{(5)}|]$ (top) and $P(X^{(5)}>2)$ (bottom) under varying $h$, computed using $M=128$ chains.
Our method, \texttt{spec-E}, achieves the smallest MSE for all tested values of $h$. \texttt{spec} is numerically unstable; the other baselines remain stable but exhibit larger estimation variance. \label{fig:blr_statistics}}
\end{figure}

\subsection{Independent component analysis}\label{sec:ICA}

The last application is a blind signal separation problem studied in \citet{Welling2011, Amari1996}. Given $N$ observations of $m$ mixed signals $x_t = Ms_t\in\mathbb R^m$, where $M\in\mathbb R^{m\times m}$ is an unknown mixing matrix and $s_t\in\mathbb R^m$ contains the original source signals, the goal is to infer a de-mixing matrix $W\in\mathbb R^{m\times m}$ such that $Wx_t$ recovers the sources up to scaling and permutation.

We adopt a Bayesian formulation with Gaussian prior $ W_{ij} \stackrel{\mathrm{i.i.d.}}{\sim} \mathcal{N}(0, \lambda^{-1})$.
Following standard ICA practice, each recovered component $y = [Wx]_i$ is assumed to follow a hyperbolic secant distribution $p(y) = \frac14 \,\mathrm{sech}^2\!\left(\frac{y}{2}\right)$, yielding the dataset likelihood
\begin{equation}
    p(X| W) = \prod_{t=1}^N|\det W| \prod_{i=1}^m p([W x_t]_i) .
\end{equation}
The log-posterior gradient is
\begin{equation}
    \nabla_{W} \log p(W|X) = N (W^\top)^{-1}
- \sum_{t=1}^N \tanh\!\left(\frac12 Wx_t\right) x_t^\top
- \lambda W ,
\end{equation}
where the $\tanh$ is applied element-wise, and $X = [x_1,\dots, x_N]\in\mathbb R^{m\times N}$.
We use $m=3$ independent source signals $s_t^1, s_t^2, s_t^3$ and $N=400$ observations, yielding a $d = m^2 = 9$ dimensional posterior.
Stochastic gradients use minibatch size $n = 40$.
The initial condition is a diagonal matrix with entries $\pm 1$ chosen uniformly at random.
The posterior is strongly multimodal with symmetric modes (\Cref{fig:ICA_density}), making this the most challenging test case. Reference samples and expectations are obtained from long MALA chains with $K=10^7$ steps and step size $h=5\times 10^{-5}$, using Riemannian manifold preconditioning together with a fixed irreversible perturbation.
Our test runs use $K=10^{7}$ steps with step size $h=2\times10^{-4}$.

\begin{figure}[htbp]
  \centering
  \includegraphics[width=0.8\textwidth]{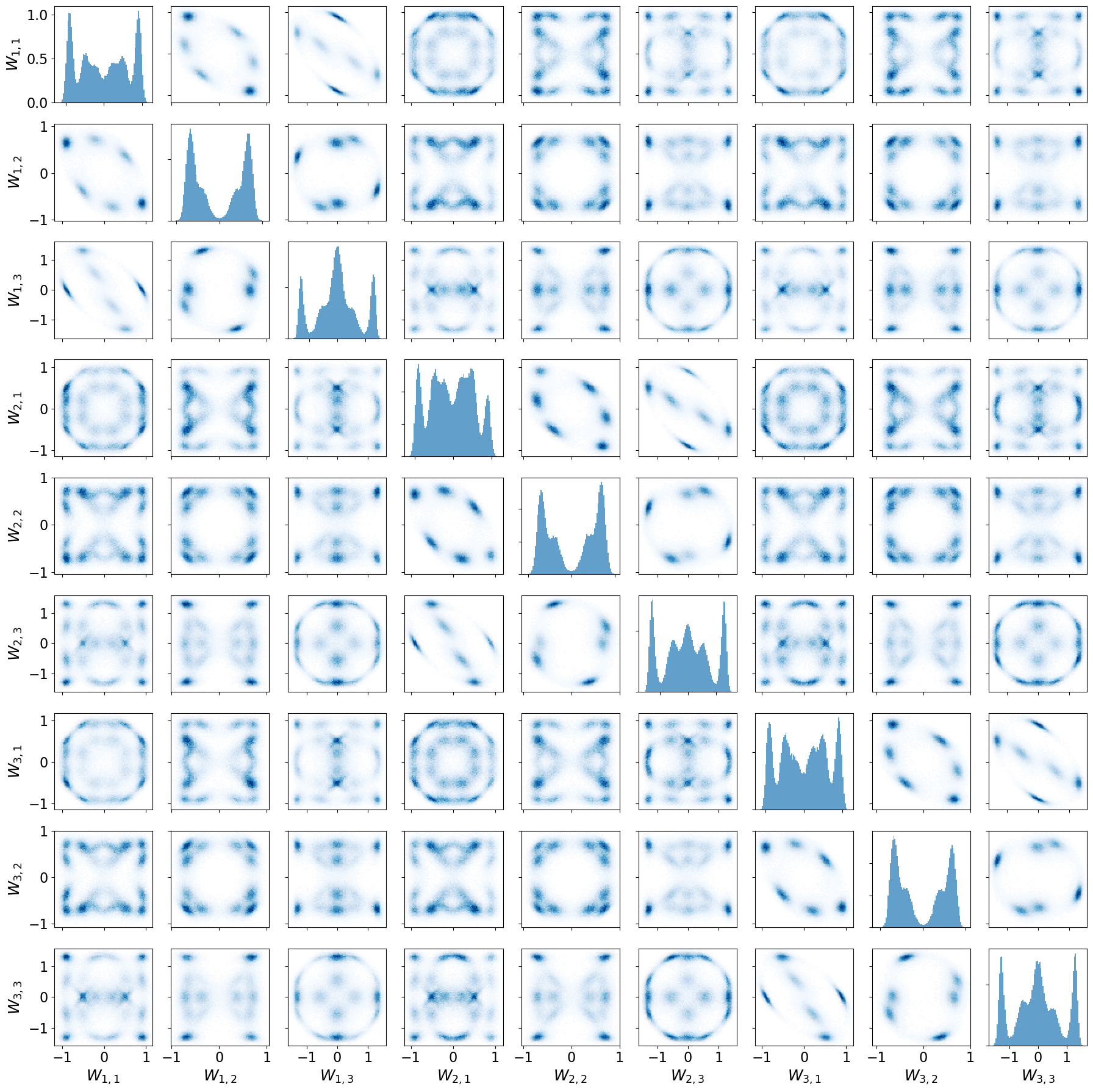}
  \caption[ICA, posterior distribution]{Posterior visualization of the ICA problem (9-dimensional),
    using reference MCMC samples.
  Notice the strong symmetry and presence of multiple modes.}
  \label{fig:ICA_density}
\end{figure}

The Fisher information matrix $\mathcal{F}$, estimated offline for this example, exhibits a $9\times 9$ block-diagonal structure with strong intra-block correlations and weak inter-block coupling, reflecting the three independent source signals in the ICA model. See \Cref{app:numerical:ica} for more details.

\paragraph{Experiment 1: Trajectory visualization.}
\Cref{fig:ICA_marginal_trajectory} shows estimated marginal distributions and chain trajectories along the first coordinate $W_{11}$.
\texttt{unperturbed} stays trapped in a subset of the modes; \texttt{rand-S/M/L} switch frequently between nearby modes but still exhibit visible bias; and \texttt{spec} leads to blowing-up chains. \texttt{spec-E} better captures the global symmetry and promotes more balanced transitions among the modes, yielding a closer approximation to the target distribution.
These qualitative results show, again, that irreversibility alone does not guarantee global exploration; the perturbation must also be chosen to avoid discretization-induced bias.

\paragraph{Experiment 2: MSE versus iteration.}
\Cref{fig:ICA-statistics} shows the evolution of MSE, squared bias, and variance of estimators of various posterior expectations, for increasing MCMC simulation time $T$ (and fixed step size $h$).
Across the statistics considered, \texttt{spec-E} is consistently among the best-performing methods, achieving fast convergence and low MSE.
This quantitative comparison reinforces the qualitative findings above.

\begin{figure}[htbp]
\centering
    \centering
    \includegraphics[width=\textwidth]{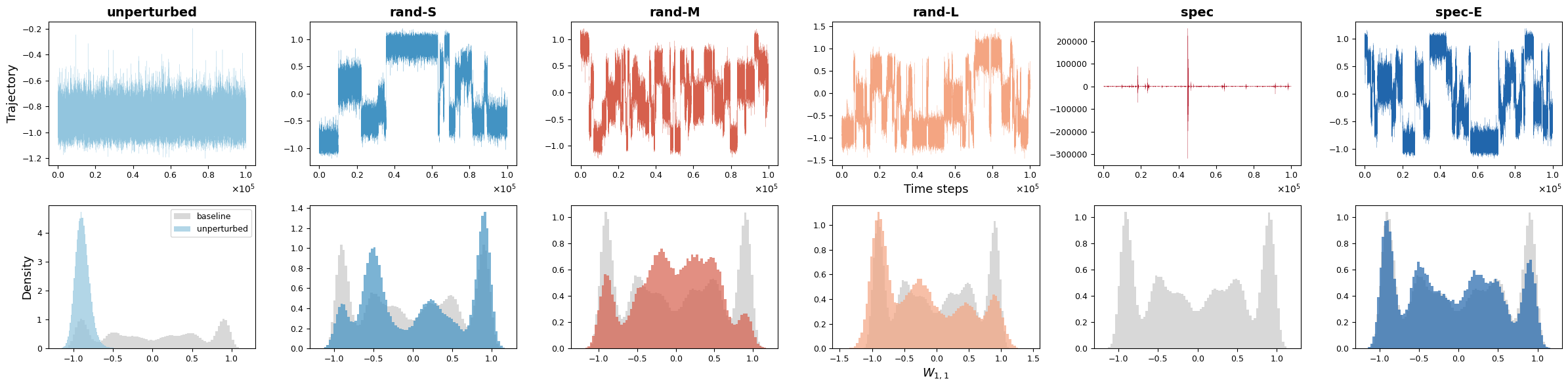}
\caption[ICA, marginal distribution and trajectory plots]{ICA: marginal distributions and first-coordinate trajectories. \texttt{spec-E} visits all modes and best matches the true distribution.}
    \label{fig:ICA_marginal_trajectory}
\end{figure}

\begin{figure}
    \centering
    \includegraphics[width=\linewidth]{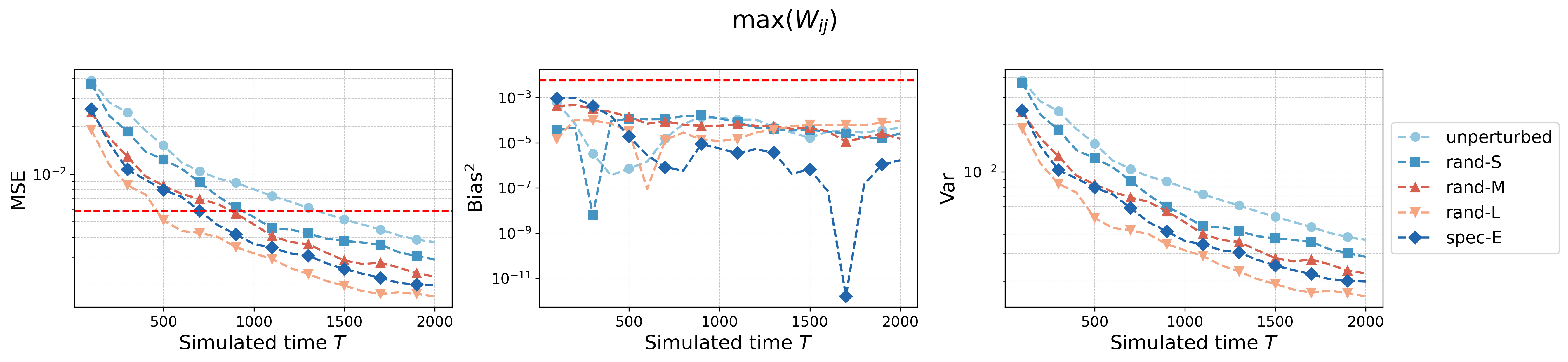}
    \includegraphics[width=\linewidth]{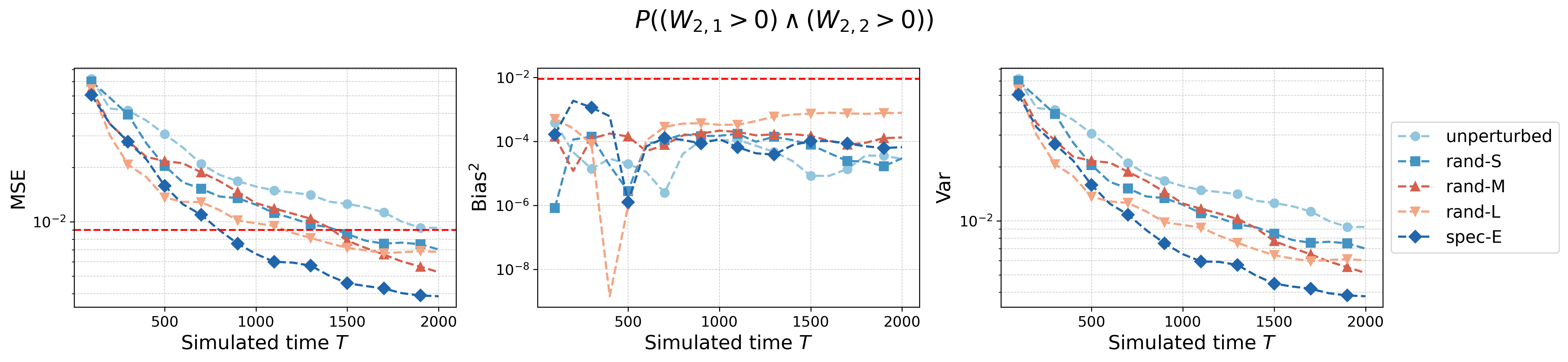}
    \includegraphics[width=\linewidth]{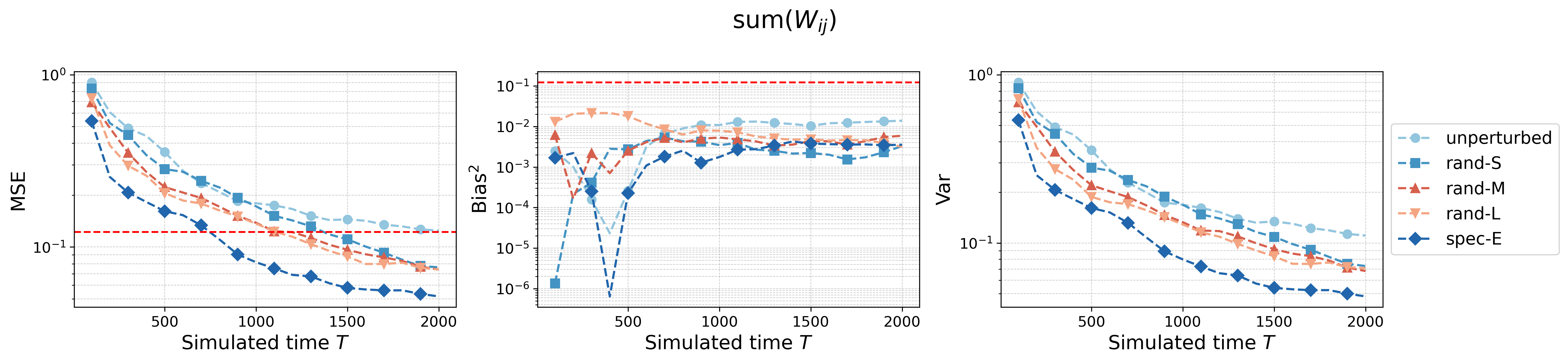}
    \caption[ICA, MSE v.s. time]{ICA: MSE versus simulated time. \texttt{spec} omitted due to instability. \texttt{spec-E} converges fastest; random perturbations lag despite similar $\|J\|_F$.}
    \label{fig:ICA-statistics}
\end{figure}

\section{Conclusion}\label{sec:conclusion}
We have presented a systematic method for constructing irreversible perturbations of time-discretized Langevin dynamics, designed to accelerate mixing while controlling numerical error. The perturbations we propose are solutions to a constrained optimization problem, wherein we minimize a weighted expected squared jump distance (with weight equal to the global Fisher information matrix) subject to a spectral-gap constraint motivated by the Gaussian setting. We show how to construct explicit solutions to this problem.
A key insight is that spectral optimality alone is insufficient for good performance of ULA: spectrally optimal perturbations that do not account for discretization can amplify bias and often lead to divergence at even moderate step sizes. By jointly optimizing for mixing speed and numerical stability, our method achieves substantially lower mean squared error than alternative irreversible perturbations, across diverse target distributions.

Several directions merit further investigation. First, the irreversible perturbations $J$ developed here are state-independent. Extending to \textit{state-dependent} perturbations $J(x)$ opens up numerous possibilities: adapting to local geometry, controlling higher-order discretization errors, and leveraging structure-preserving integrators.
Second, though the present paper is focused solely on the design of irreversible perturbations, it is natural to consider combining these perturbations with reversible preconditioning. These different classes of perturbations may have complementary strengths; indeed, our preliminary experiments in \Cref{app:robustness-analysis} suggest that irreversible perturbations can be more robust than reversible preconditioning when local and global geometry diverge. A complete theoretical analysis is beyond the scope of this work, but would be interesting in the future.
Third, our analysis has focused on overdamped Langevin dynamics; lifting to the underdamped setting could exploit hypocoercivity through skew-symmetric position-momentum couplings for further mixing improvements.
Finally, interacting particle systems offer another opportunity for the design of more effective algorithms. Ensemble-informed irreversible perturbations, e.g., using ensemble estimates of the Fisher information, could enhance exploration of complex multi-modal landscapes.

\appendix

\section{Reversible perturbations}\label{app:reversible}
Reversible perturbations precondition the Langevin dynamics by introducing a symmetric positive-definite matrix $A(x) \succ 0$ that rescales and reshapes the state space. The preconditioned overdamped Langevin diffusion is
\begin{equation}\label{eq:preconditioned-LD}
    \mathrm{d}X_t = \big[-A(X_t)\nabla V(X_t) + \nabla \cdot A(X_t)\big]\mathrm{d}t + \sqrt{2A(X_t)}\,\mathrm{d}W_t,
\end{equation}
where the divergence correction $\nabla \cdot A(X_t)$ ensures that $\pi \propto e^{-V}$ remains the invariant measure. The diffusion coefficient $\sqrt{2A(X_t)}$ denotes any matrix square root satisfying $\sqrt{2A}\,\sqrt{2A}^\top = 2A$. When $A(x) = I$, equation~\eqref{eq:preconditioned-LD} reduces to the standard overdamped Langevin dynamics~\eqref{eq:overdamped_LD}. This framework is also known as Riemannian manifold Langevin dynamics (RMLD) \citep{girolami2011riemann}.

For a \emph{constant} preconditioner $A(x) \equiv A$, the divergence term vanishes and the dynamics simplify to
\begin{equation}\label{eq:const-precond-LD}
    \mathrm{d}X_t = -A\nabla V(X_t)\,\mathrm{d}t + \sqrt{2A}\,\mathrm{d}W_t.
\end{equation}
The corresponding ULA discretization is
\begin{equation*}
    X_{t+1} = X_t + hA\nabla\log\pi(X_t) + \sqrt{2hA}\,Z_t, \qquad Z_t \sim \mathcal{N}(0,I).
\end{equation*}
The preconditioner $A$ reshapes the effective geometry of the target: choosing $A$ to align with the local curvature of $V$ can significantly accelerate convergence. \citet{Ti2023} showed that the ESJD-optimal constant preconditioner is $A^* \propto \mathcal{F}^{-1}$, where
\begin{equation*}
    \mathcal{F} = \mathbb{E}_\pi\!\left[\nabla\log\pi(X)\,\nabla\log\pi(X)^\top\right]
\end{equation*}
is the Fisher information matrix, which captures the average local curvature of $\log\pi$. For a Gaussian target with precision $S$, the Fisher information matrix recovers $\mathcal{F} = S$ exactly.

Our approach extends this ESJD-based framework from reversible to irreversible perturbations. While reversible preconditioning rescales the geometry, irreversible perturbations introduce rotational dynamics that break detailed balance and induce probability currents. The two strategies are complementary, and both benefit from Fisher-informed design.

\section{Derivations}\label{app:derivations}

\subsection{Mean and covariance of Gaussian ULA}\label{app:gaussian-example}
Consider the ULA iteration for a Gaussian target with mean $\mu$ and covariance $S^{-1}$:
\begin{equation*}
    X_t = X_{t-1} - h(I+J)S(X_{t-1} - \mu) + \sqrt{2h}\,Z_{t-1}, \qquad Z_{t-1} \sim \mathcal{N}(0, I).
\end{equation*}
Define $A :=(I+J)S$ and $B := I - h A$.
Subtracting $\mu$ and unrolling the recursion gives
\begin{equation*}
    X_t - \mu = \big( I - h (I+J) S \big)^t (X_0 - \mu)
+ \sqrt{2h} \sum_{i=0}^{t-1} \big( I - h (I+J) S \big)^{t-1-i} \, Z_i.
\end{equation*}

\paragraph{Mean.}
For Gaussian initialization $X_0 =\mathcal N(\mu_0,\tau^2 I)$, taking expectations yields:
\begin{equation*}
    \mathbb{E}[X_t]  - \mu = B^t (\mu_0 - \mu).
\end{equation*}
The mean converges geometrically to $\mu$ at a rate determined by the spectral radius $\rho(B)$, whenever the initial mean $\mu_0$ differs from $\mu$.
For small \(h\), \(\rho(B)\) is governed to first order by the real parts of the eigenvalues of \(A=(I+J)S\), matching the continuous-time spectral-gap criterion.

\paragraph{Covariance.}
\begin{align*}
    \Sigma_t &= \E[(X_t-\mu_t)(X_t-\mu_t)^\top]\\
    &=\E\left[\Bigl(B^t(X_0-\mu)+\sqrt{2h}\sum_{i=0}^{t-1}B^{t-1-i}\xi_i-(\mu_t-\mu)\Bigr)\Bigl(\cdots\Bigr)^\top\right]\\
    &=\E\left[\Bigl(B^t(X_0-\mu_0)+\sqrt{2h}\sum_{i=0}^{t-1}B^{t-1-i}\xi_i\Bigr) \Bigl(\cdots \Bigr)^\top\right]\\
    &=\tau^2B^t(B^\top)^t + 2h\sum_{i=0}^{t-1}B^i(B^\top)^i.
\end{align*}
The covariance satisfies the recursion $\Sigma_{t+1} = B \Sigma_t B^\top + 2h I$.
As $t\rightarrow\infty$, if $\rho(B)< 1$, then $\lim_{t\rightarrow\infty}\Sigma_t = \Sigma_h$, which satisfies the discrete Lyapunov equation
\begin{equation}\label{eq:discrete-lyapunov}
    \Sigma_h = B \Sigma_h B^\top + 2h I.
\end{equation}

\paragraph{Expansion in $h$.}
Substituting $B = I-h A$ into \eqref{eq:discrete-lyapunov} yields
\begin{equation*}
    \Sigma_h = \Sigma_h - h(A\Sigma_h + \Sigma_h A^\top) + h^2 A\Sigma_h A^\top + 2h I.
\end{equation*}
Rearranging and dividing by $h$:
\begin{equation*}
    A\Sigma_h + \Sigma_h A^\top = 2I + h A\Sigma_h A^\top.
\end{equation*}
Expanding $\Sigma_h = \Sigma^{(0)}+h\Sigma^{(1)}+\mathcal O(h^2)$ and
matching powers of $h$,
\begin{align*}
     \mathcal O(1):
     &\quad A\Sigma^{(0)} + \Sigma^{(0)}A^\top = 2I, \\
     \mathcal O(h): &\quad A\Sigma^{(1)} + \Sigma^{(1)}A^\top = A\Sigma^{(0)}A^\top.
\end{align*}
For the $\mathcal O(1)$ equation, one can verify $\Sigma^{(0)} = S^{-1}$ is a valid (and unique) solution.
For the $\mathcal O(h)$ equation, substituting $\Sigma^{(0)} = S^{-1}$,
\begin{align*}
    A\Sigma^{(1)} + \Sigma^{(1)}A^\top &= AS^{-1}A^\top = (I+J)S(I+J)^\top =: \Delta,\\
    \Sigma^{(1)} &= \int_0^\infty e^{-tA} \Delta \, e^{-tA^\top} \, dt.
\end{align*}
In summary, $\Sigma_h = S^{-1} + h\Sigma^{(1)} + O(h^2)$,
where the leading discretization error $\Sigma^{(1)}$ is driven by $\Delta = (I+J)S(I+J)^\top$.
Taking the trace:
$$\mathrm{Tr}(\Delta) = \mathrm{Tr}(S) + \mathrm{Tr}(JSJ^\top),$$
where $\mathrm{Tr}(JSJ^\top)$ is the $J$-dependent term.
This motivates minimizing $\mathcal E_1(J)=\mathrm{Tr}(JSJ^\top)$ or related quantities to control discretization bias.

\subsection{Discretization error in It\^o-Taylor expansion}\label{app:ito-taylor}
In this section we derive the local truncation error of Euler--Maruyama scheme using the It\^o--Taylor expansion, providing an alternative perspective on discretization bias.
\paragraph{General It\^o--Taylor expansion}
Consider the SDE
\begin{equation*}
    \mathrm{d}X_t = b(X_t) \, \mathrm{d}t + \sigma(X_t) \, \mathrm{d}W_t, \qquad X_0 \sim \pi_0.
\end{equation*}
For a smooth test function $f$, define
\begin{equation*}
    \mathcal{L}_0 f(x) = b(x) \cdot \nabla f(x) + \frac{1}{2} \sigma(x)\sigma(x)^\top : \Delta f(x),
    \qquad
    \mathcal{L}_1 f(x)= \sigma(x) \cdot \nabla f(x).
\end{equation*}
For a fixed time $t$ and step size $h>0$, let $\Delta W_t:=W_{t+h}-W_t$. Applying It\^o's lemma to the identity function $f(x) = x$ over the interval $[t,t+h]$ gives
\begin{align}\label{eq:ito_formula}
    X_{t+h}
    &= X_t + b(X_t)h+\sigma(X_t)\Delta W_t+R_1,\\
    R_1
    &= \int_0^h \int_0^s \mathcal{L}_0 b(X_{t+z})\,\mathrm{d}z\,\mathrm{d}s
    + \int_0^h \int_0^s \mathcal{L}_1 b(X_{t+z})\,\mathrm{d}W_{t+z}\,\mathrm{d}s \notag\\
    &\quad
    + \int_0^h \int_0^s \mathcal{L}_0 \sigma(X_{t+z})\,\mathrm{d}z\,\mathrm{d}W_{t+s}
    + \int_0^h \int_0^s \mathcal{L}_1 \sigma(X_{t+z})\,\mathrm{d}W_{t+z}\,\mathrm{d}W_{t+s}.
\end{align}
The four terms have $L^2$ orders $\mathcal{O}(h^2)$, $\mathcal{O}(h^{3/2})$, $\mathcal{O}(h^{3/2})$, and $\mathcal{O}(h)$, respectively.

\paragraph{Application to Langevin dynamics}
For the Langevin dynamics considered in this paper, $\sigma(x) = \sqrt{2}I$ is constant, so $\mathcal L_0\sigma = 0$ and $\mathcal L_1\sigma = 0$. Consequently, the last two terms vanish, and we obtain
\begin{equation*}
    X_{t+h} = X_t + b(X_t)\,h + \sqrt{2}\Delta W_t + \int_0^h \int_0^s \mathcal{L}_1 b(X_{t+z})\,\mathrm{d}W_{t+z}\,\mathrm{d}s  + O(h^2).
\end{equation*}
The dominant remainder term is $\mathcal{O}(h^{3/2})$.

\paragraph{Linear Langevin dynamics}
For the irreversibly perturbed linear Langevin SDE, we have $b(X)=-(I+J) SX$, $\sigma = \sqrt{2}I$. Therefore,
$\nabla b = -(I+J)S$ and $\mathcal L_1 b(X_0) = -\sqrt{2}(I+J)S$. Since $L_1 b$ is constant, the $\mathcal{O}(h^{3/2})$ term simplifies to
\begin{equation*}
    \int_0^h \int_0^s \mathcal{L}_1 b(X_{t+z})\,\mathrm{d}W_{t+z}\,\mathrm{d}s  = -\sqrt{2}(I+J)S
    \int_0^h (W_{t+s}-W_t)\,\mathrm{d}s =: -\sqrt{2}(I+J)S\,Z_{t,h},
\end{equation*}
where $Z_{t,h}
    :=
    \int_0^h (W_{t+s}-W_t)\,\mathrm{d}s
    \sim
    \mathcal{N}\!\left(0,\frac{h^3}{3}I\right).$

Therefore, the local expansion from time $t$ is
\begin{equation*}
    X_{t+h}
    =
    X_t
    -
    (I+J)SX_t\,h
    +
    \sqrt{2}\Delta W_t
    -
    \sqrt{2}(I+J)S Z_{t,h}
    +
    \mathcal{O}(h^2).
\end{equation*}

\paragraph{Local truncation error}
The Euler--Maruyama scheme over one step is $\widetilde X_{t+h}
    =
    X_t + b(X_t)h + \sqrt{2}\Delta W_t$. For the linear Langevin dynamics, the leading local truncation error is therefore
\begin{equation*}
    X_{t+h}-\widetilde X_{t+h}
    =
    -\sqrt{2}(I+J)S Z_{t,h}
    +
    \mathcal{O}(h^2).
\end{equation*}
Consequently,
\begin{align*}
    \mathbb{E}\!\left[
        \left\|
        \sqrt{2}(I+J)S Z_{t,h}
        \right\|^2
    \right]
    &=
    2\,\mathrm{Tr}\!\left[
        (I+J)S
        \mathbb{E}[Z_{t,h}Z_{t,h}^\top]
        S(I+J)^\top
    \right] \\
    &=
    \frac{2h^3}{3}
    \mathrm{Tr}\!\left[(I+J)S^2(I+J)^\top\right] \\
    &=
    \frac{2h^3}{3}\mathrm{Tr}(S^2)
    +
    \frac{2h^3}{3}\mathrm{Tr}(JS^2J^\top),
\end{align*}
where the cross terms vanish because $S^2$ is symmetric and $J$ is skew-symmetric. The first term is independent of $J$, so the $J$-dependent contribution is
$\mathcal E_3(J) := \mathrm{Tr}(JS^2J^\top) =-\mathrm{Tr}\!\left[ J S S J \right]$.
This motivates minimizing $\mathcal{E}_3(J)$ as a proxy for controlling the leading local discretization error.

\section{Proofs}\label{app:proofs}

\subsection{Proof of Proposition~\ref{prop:proxy-equivalence}}\label{app:proof:prop:proxy-equivalence}
\begin{proof}
We first rewrite each discretization error proxy as a Frobenius norm.
Since $J^\top=-J$,
\[
\mathcal E_1(J)=\text{Tr}(-JSJ)=\text{Tr}(JSJ^\top)=\|J S^{1/2}\|_F^2\geq 0.
\]
Moreover, with $W=S$,
\[
\mathcal E_2(J)= \|S^{1/2} J S^{1/2}\|_F^2\geq 0,\qquad \mathcal E_3(J)=\|J S\|_F^2\geq 0.
\]
Next, set $M:=J S^{1/2}$. Then $M^\top M\succeq 0$ and $\mathcal E_1(J)=\text{Tr}(M^\top M)$. Also,
\[
\mathcal E_3(J)=\|M S^{1/2}\|_F^2
=\text{Tr}\big((MS^{1/2})^\top (MS^{1/2})\big)
=\text{Tr}\big(S^{1/2} M^\top M S^{1/2}\big)
=\text{Tr}\big(M^\top M\,S\big).
\]
Since $M^\top M\succeq 0$ and $S\succeq 0$, the inequalities
$\sigma_{\min}(S)I\preceq S\preceq \sigma_{\max}(S)I$
and trace monotonicity yield
\[
\sigma_{\min}(S)\text{Tr}(M^\top M)\le \text{Tr}(M^\top M S)\le \sigma_{\max}(S)\text{Tr}(M^\top M),
\]
i.e.,
\[
\sigma_{\min}(S)\,\mathcal E_1(J)\le \mathcal E_3(J)\le \sigma_{\max}(S) \,\mathcal E_1(J).
\]
For $\mathcal E_2$, set $N:=S^{1/2}J S^{1/2} = S^{1/2} M$. Then
\[
\mathcal E_2(J)=\|N\|_F^2=\|S^{1/2}M\|_F^2
=\text{Tr}\big((S^{1/2}M)^\top(S^{1/2}M)\big)
=\text{Tr}\big(M^\top S M\big)
=\text{Tr}\big(M M^\top S\big).
\]
Apply the same eigenvalue sandwich with $A:=MM^\top\succeq 0$ and $B:=S$ to obtain
\[
\sigma_{\min}(S)\text{Tr}(MM^\top)\le \text{Tr}(MM^\top S)\le \sigma_{\max}(S)\text{Tr}(MM^\top).
\]
Finally, $\text{Tr}(MM^\top)=\text{Tr}(M^\top M)=\mathcal E_1(J)$, and hence
\[
\sigma_{\min}(S)\,\mathcal E_1(J)\le \mathcal E_2(J)\le \sigma_{\max}(S)\,\mathcal E_1(J).
\]
Combining the two bounds gives the equivalence between $\mathcal E_2$ and $\mathcal E_3$ up to $\kappa(S)=\sigma_{\max}(S)/\sigma_{\min}(S)$.
\end{proof}

\subsection{Proof of Theorem \ref{thm:P1-P2-relation}}\label{proof:thm:P1-P2-relation}
\begin{proof}
Let $J\in\mathcal J_\mathrm{s}$, and set
\[
    \tilde J=S^{1/2}JS^{1/2},
    \qquad
    A=S+\tilde J .
\]
Since $A$ is similar to $(I+J)S$, spectral feasibility gives
$\sigma_j(A)=\gamma_\star+i\omega_j$, where $\gamma_\star=\mathrm{Tr}(S)/d$. Therefore
\(\sum_j \omega_j=0\), since $\mathrm{Tr}(A)=d\gamma_\star$ is real, and
\[
    \mathrm{Tr}(A^2)
    =
    \sum_j (\gamma_\star+i\omega_j)^2
    =
    d\gamma_\star^2-\sum_j\omega_j^2
    \le d\gamma_\star^2 .
\]
On the other hand, since $S$ is symmetric, $\tilde J$ is skew-symmetric,
and $\mathrm{Tr}(S\tilde J)=0$,
\[
    \mathrm{Tr}(A^2)
    =
    \mathrm{Tr}\bigl((S+\tilde J)^2\bigr)
    =
    \|S\|_F^2-\|\tilde J\|_F^2 .
\]
Moreover,
\[
    \|\tilde J\|_F^2
    =
    \mathrm{Tr}(-JSJS).
\]
Thus every feasible point of \eqref{eq:P1} satisfies
\begin{align}\label{Eq:LowerBoundP1}
\mathrm{Tr}(-JSJS)
    \ge
    \|S\|_F^2-d\gamma_\star^2 .
\end{align}

The existence of an orthonormal basis \(\Psi\) with
\(\operatorname{diag}(\Psi^\top S\Psi)=\gamma_\star\mathbf 1\) follows from the
Schur--Horn theorem: the constant vector \((\gamma_\star,\ldots,\gamma_\star)\) is majorized by the
eigenvalues of \(S\).

We next show that the triangular construction attains this bound. By definition,
\[
    B=\Psi^\top S\Psi,
    \qquad
    M_\star=\Psi^\top \tilde J_\star \Psi .
\]
Since \(B+M_\star\) is upper triangular with diagonal entries \(\gamma_\star\),
\[
    \Psi^\top (S+\tilde J_\star)\Psi
    =
    B+M_\star
\]
has all eigenvalues equal to \(\gamma_\star\). Hence \(S+\tilde J_\star\), and therefore
\((I+J_\star)S\), has all eigenvalues with real part \(\gamma_\star\). Thus
\(J_\star\in\mathcal J_\mathrm{s}\).

Furthermore, orthogonal invariance of the Frobenius norm gives
\[
    \mathrm{Tr}(-J_\star S J_\star S)
    =
    \|\tilde J_\star\|_F^2
    =
    \|M_\star\|_F^2 .
\]
Since \(M_\star\) agrees with \(B\) up to signs on the off-diagonal entries and
vanishes on the diagonal,
\[
    \|M_\star\|_F^2
    =
    \sum_{j\neq k} B_{jk}^2
    =
    \|B\|_F^2-\sum_{k=1}^d B_{kk}^2
    =
    \|S\|_F^2-d\gamma_\star^2 .
\]
Thus, comparing with (\ref{Eq:LowerBoundP1}), \(J_\star\) attains the lower bound, and hence
\[
    \min(\mathrm{P1})
    =
    \|S\|_F^2-d\gamma_\star^2.
\]

It remains to relate this value to \eqref{eq:P2}. This follows from the coordinate
change defining \eqref{eq:P2}. Given any feasible \(J\) in \eqref{eq:P1}, set
\(\tilde J=S^{1/2}JS^{1/2}\) and \(M=\Psi^\top\tilde J\Psi\) for any
orthonormal \(\Psi\). Then \(M^\top=-M\), \(B_\Psi+M\) is orthogonally similar
to \(S+\tilde J\), and
\[
    \|M\|_F^2
    =
    \|\tilde J\|_F^2
    =
    \mathrm{Tr}(-JSJS).
\]
Conversely, any feasible pair \((\Psi,M)\) in \eqref{eq:P2} induces
\[
    J=S^{-1/2}\Psi M\Psi^\top S^{-1/2},
\]
which is feasible for \eqref{eq:P1} with the same objective value. Therefore
\[
    \min(\mathrm{P2})
    =
    \min(\mathrm{P1})
    =
    \|S\|_F^2-d\gamma_\star^2.
\]

Finally, we characterize the minimizers of \eqref{eq:P1}. Let \(J\in\mathcal J_\mathrm{s}\)
be a global minimizer. Equality in the lower bound above forces
\[
    \sum_j \omega_j^2=0,
\]
so all eigenvalues of \(A=S+\tilde J\) are equal to \(\gamma_\star\). By the real Schur
decomposition, there exists an orthonormal basis
\(\Psi=[\psi_1,\ldots,\psi_d]\) such that $\Psi^\top A\Psi$
is upper triangular with diagonal entries \(\gamma_\star\). Write
\[
    B=\Psi^\top S\Psi,
    \qquad
    M=\Psi^\top \tilde J\Psi .
\]
Then \(B\) is symmetric, \(M\) is skew-symmetric, and \(B+M\) is upper triangular.
In particular, \(B_{kk}=\gamma_\star\) for each \(k\), and for \(j>k\),
\[
    0=(B+M)_{jk}=B_{jk}+M_{jk}.
\]
Thus
\[
    M_{jk}
    =
    \begin{cases}
        B_{jk}, & j<k,\\
        0, & j=k,\\
        -B_{jk}, & j>k,
    \end{cases}
\]
which is precisely the triangular construction in the statement. Conversely,
the construction above shows that every such triangular form is feasible for
\eqref{eq:P1} and attains the minimum.
\end{proof}

\subsection{Spectral stability under small perturbations}\label{app:prop:stability:spectrum}
The following proposition shows that the spectral gap is stable under target perturbations with small Hessian. Although it is not used directly in the main paper, it provides a useful construction and key intuition for the proof of \Cref{prop:stability:spectrum} in the next subsection.

\begin{prop}[Spectral stability under small perturbations]\label{prop:stability:spectrum}
Let $S \succ 0$ and let $V(x) = \frac{1}{2}x^\top Sx + r(x)$ with $r \in C^3(\mathbb{R}^d)$ satisfying
\[
\varepsilon\;:=\;\sup_{x\in\mathbb R^d}\bigl\|\nabla^2 r(x)\bigr\|_2\;<\;\min\{\sigma_{\min}(S), 1\}.
\]
Here $\sigma_{\min}(S)$ denotes the smallest eigenvalue of $S$.
This assumption makes \(V\) strongly convex, so \(-\mathcal L_\varepsilon\) has discrete spectrum.
Let $\mathcal{L}_\varepsilon = \Delta - \nabla V \cdot \nabla$ be the Langevin generator with invariant measure $\pi \propto e^{-V}$, and let $0 = \sigma_0(\varepsilon) < \sigma_1(\varepsilon) \leq \sigma_2(\varepsilon) \leq \cdots$ denote the eigenvalues of $-\mathcal{L}_\varepsilon$.
For the Gaussian case $r \equiv 0$, the spectrum of $-\mathcal{L}_0$ is $\{\sum_{i=1}^d n_i \sigma_i(S) : n_i \in \mathbb{N}_0\}$, where $\sigma_1(S), \ldots, \sigma_d(S)$ are the eigenvalues of $S$. In particular, the spectral gap is $\gamma(0)=\sigma_{\min}(S)$.

Then there exists $C = C(S, d) > 0$ such that for all $k \geq 0$,
\[
\bigl|\sigma_k(\varepsilon)-\sigma_k(0)\bigr|
\;\le\; C\,\varepsilon\bigl(1+\sigma_k(0)\bigr).
\]
In particular, the spectral gap $\gamma(\varepsilon):=\sigma_1(\varepsilon)$ satisfies
\[
\sigma_{\min}(S)-\varepsilon
\;\le\; \gamma(\varepsilon)
\;\le\; \sigma_{\min}(S)+C\,\varepsilon.
\]
\end{prop}

\begin{proof}[Proof]
We begin by recalling that
\begin{equation*}
    \mathcal L_\varepsilon f = \Delta f-\nabla V\cdot \nabla f \;\;\text{ on }\;\;L^2(\pi), \qquad \pi\propto e^{-V(x)}dx,
\end{equation*}
Using the ground-state transform $(Uf)(x) = e^{-V(x)/2}f(x)$, which maps $L^2(\pi)$ unitarily onto $L^2(dx)$ up to the normalizing constant, we define the conjugated operator
\begin{equation*}
    \tilde{L}_\varepsilon := -U\mathcal L_\varepsilon U^{-1} = - e^{-V/2}\mathcal L_\varepsilon e^{V/2} = -\Delta+\frac{1}{4}\|\nabla V\|_2^2-\frac{1}{2}\Delta V.
\end{equation*}
Thus $-\mathcal L_\varepsilon$ and $\tilde{L}_\varepsilon$ have identical spectra.

\paragraph{Normalization and decomposition.}
Write $\nabla V = Sx+R(x)$, where $R(x) = \nabla r(x)$ and $\sup_x\|\nabla^2 r(x)\|_2\leq \varepsilon$ by assumption. Without loss of generality, we may assume $\nabla r(0) = 0$. Indeed, let $x_*$ be the unique solution to $Sx_* + \nabla r(x_*) = 0$, which exists for $\varepsilon < \sigma_{\min}(S)$ by the contraction mapping theorem: the map $x \mapsto -S^{-1}\nabla r(x)$ has Lipschitz constant $\varepsilon/\sigma_{\min}(S) < 1$. The change of variables $y = x - x_*$ gives $V(y + x_*) = \frac{1}{2}y^\top Sy + \hat{r}(y) + \mathrm{const}$, where $\hat{r}(y) = y^\top Sx_* + r(y + x_*)$ satisfies $\nabla\hat{r}(0) = Sx_* + \nabla r(x_*) = 0$ and $\|\nabla^2\hat{r}\|_\infty = \|\nabla^2 r\|_\infty \leq \varepsilon$. Since the spectrum of $-\mathcal{L}_\varepsilon$ is invariant under translation and additive constants, we may replace $r$ by $\hat{r}$. With this normalization, $\|R(x)\|_2 \leq \varepsilon\|x\|_2$ by the mean value theorem.

Decompose $\tilde{L}_\varepsilon = \tilde{L}_0 + W_\varepsilon$, where
\begin{align*}
    \tilde{L}_0 &= -\Delta+\frac{1}{4}\|Sx\|_2^2-\frac{1}{2}\text{Tr}(S),\qquad W_\varepsilon  = \frac{1}{2}x^\top SR(x)+\frac{1}{4}\|R(x)\|^2_2-\frac{1}{2}\text{Tr}\nabla^2 r(x)
\end{align*}
where $W_\varepsilon$ collects all perturbation terms.

\paragraph{Perturbation bound.}
We next derive a growth bound on the perturbation. Since $\|R(x)\|_2\leq \varepsilon\|x\|_2$,
\begin{equation*}
    |x^\top SR(x)|\leq\|S\|_2\|x\|_2\|R(x)\|_2\leq\|S\|_2\varepsilon\|x\|^2_2,\quad \|R(x)\|^2_2\leq \varepsilon^2\|x\|_2^2,\quad |\text{Tr}\nabla^2 r(x)|\leq d\varepsilon.
\end{equation*}
Hence for $\varepsilon\leq 1$,
\begin{equation}\label{eq:upper-bound-W-eps}
    |W_\varepsilon(x)| \leq \left(\frac{\|S\|}{2} + \frac{1}{4}\right) \varepsilon \|x\|^2 + \frac{d}{2}\varepsilon \leq C \varepsilon (1 + \|x\|^2), \qquad C = C(S,d).
\end{equation}
To control the effect of $W_\varepsilon$, we compare it to the quadratic form of $\tilde{L}_0$. Since $\tilde{L}_0 = -\Delta + \frac{1}{4}\|Sx\|_2^2 - \frac{1}{2}\mathrm{Tr}(S)$, integration by parts gives, for $u$ in the form domain,
\[
\langle u, \tilde{L}_0 u \rangle = \int \|\nabla u\|_2^2 \, dx + \frac{1}{4}\int x^\top S^2 x \, |u|^2 \, dx - \frac{1}{2}\mathrm{Tr}(S)\|u\|^2 = \mathcal{E}_0[u] - \tfrac{1}{2}\mathrm{Tr}(S)\|u\|^2,
\]
where we define the reference energy
\begin{equation*}
    \mathcal E_0[u] := \int\left(\|\nabla u\|^2_2+\frac{1}{4}x^\top S^2 x\,|u|^2\right)\d x.
\end{equation*}
Since $S \succeq \sigma_{\min}(S) I$, we have $x^\top S^2 x \geq \sigma_{\min}(S)^2 \|x\|^2$. Dropping the nonnegative gradient term from $\mathcal{E}_0[u]$ and using this bound gives
\begin{equation}\label{eq:form-bound-h0}
    \int \|x\|^2 |u|^2 \, dx \leq \frac{1}{\frac{1}{4}\sigma_{\min}(S)^2}\int \frac{1}{4}x^\top S^2 x\, |u|^2\,dx \leq \frac{4}{\sigma_{\min}(S)^2} \mathcal{E}_0[u].
\end{equation}
Combining \eqref{eq:upper-bound-W-eps} and \eqref{eq:form-bound-h0}, we obtain the key energy estimate:
\begin{align*}
    |\langle u, W_\varepsilon u \rangle|
    &\leq C \varepsilon \int (1 + \|x\|^2) |u|^2 \, dx \\
    &= C \varepsilon \left( \|u\|^2 + \int \|x\|^2 |u|^2 \, dx \right) \\
    &\leq C \varepsilon \left( \|u\|^2 + \frac{4}{\sigma_{\min}(S)^2} \mathcal{E}_0[u] \right) \\
    &= a\ \varepsilon \, \mathcal{E}_0[u] + b\ \varepsilon \, \|u\|^2,
\end{align*}
for all finite energy $u$. Thus $W_\varepsilon$ is small of order $\mathcal O(\varepsilon)$ in the energy sense.

Since $\varepsilon<\sigma_{\min}(S)$ and $R(0)=0$, we have $\|\nabla V(x)\|_2\geq(\sigma_{\min}(S)-\varepsilon)\|x\|_2$ and $|\Delta V(x)|\leq \mathrm{Tr}(S)+d\varepsilon$. Thus the Schr\"odinger potential $\frac14\|\nabla V\|_2^2-\frac12\Delta V$ is confining and comparable to $1+\|x\|^2$. Hence the Friedrichs realization of $\tilde L_\varepsilon$ is self-adjoint with compact resolvent and the same quadratic-form domain as $\tilde L_0$.

\paragraph{Eigenvalue bounds.}
Denote the eigenvalues of $\tilde{L}_{\bullet}$ as $\sigma_{k}(\bullet)$.
Since $\langle u, \tilde{L}_\varepsilon u\rangle = \langle u, \tilde{L}_0 u\rangle + \langle u, W_\varepsilon u\rangle$ and $|\langle u, W_\varepsilon u\rangle| \leq a\varepsilon\,\mathcal{E}_0[u] + b\varepsilon\|u\|^2$, we need to relate $\mathcal{E}_0[u]$ to $\langle u, \tilde{L}_0 u\rangle$. From the identity $\langle u, \tilde{L}_0 u\rangle = \mathcal{E}_0[u] - \frac{1}{2}\mathrm{Tr}(S)\|u\|^2$ established above, we have $\mathcal{E}_0[u] = \langle u, \tilde{L}_0 u\rangle + \frac{1}{2}\mathrm{Tr}(S)\|u\|^2$. Substituting into the form bound and absorbing the $\|u\|^2$ terms (with updated constants $a, b$) gives
\begin{align*}
    -a\varepsilon \langle u, \tilde{L}_0 u\rangle - b\varepsilon\|u\|^2 \leq \langle u, W_\varepsilon u\rangle \leq a\varepsilon \langle u, \tilde{L}_0 u\rangle + b\varepsilon\|u\|^2,
\end{align*}
where we used $\langle u, \tilde{L}_0 u\rangle \geq 0$ (since $\tilde{L}_0$ is nonnegative). Adding $\langle u, \tilde{L}_0 u\rangle$ throughout:
\begin{align*}
    (1-a\varepsilon) \langle u, \tilde{L}_0 u\rangle - b\varepsilon\|u\|^2&\leq \langle u, \tilde{L}_\varepsilon u\rangle\leq (1+a\varepsilon)\langle u, \tilde{L}_0 u\rangle + b\varepsilon\|u\|^2.
\end{align*}
By the Courant--Fischer min--max principle \citep[Thm.~XIII.1]{reed1978iv}, the upper form bound gives $\sigma_k(\varepsilon)\leq (1+a\varepsilon)\sigma_k(0)+b\varepsilon$. For the lower bound, if $1-a\varepsilon\geq 0$, min--max gives $\sigma_k(\varepsilon)\geq (1-a\varepsilon)\sigma_k(0)-b\varepsilon$; if $1-a\varepsilon<0$, the right-hand side is nonpositive and the bound is trivial since $\sigma_k(\varepsilon)\geq0$. Hence
\begin{equation*}
    (1-a\varepsilon)\sigma_k(0)-b\varepsilon\leq \sigma_k(\varepsilon)\leq (1+a\varepsilon)\sigma_k(0)+b\varepsilon.
\end{equation*}
Rearranging,
\begin{equation*}
    |\sigma_k(\varepsilon)-\sigma_k(0)|\leq a\varepsilon\sigma_k(0)+b\varepsilon\leq C\varepsilon(1+\sigma_k(0)),\qquad C = \max\{a,b\}.
\end{equation*}
In particular, since $\sigma_0(0) = \sigma_0(\varepsilon) = 0$,  $\sigma_1(0) = \sigma_{\min}(S)$, $\sigma_1(\varepsilon)
\;\le\; \sigma_{\min}(S)+C\,\varepsilon$.

A complementary lower bound comes from convexity. Weyl's inequality for symmetric matrices shows $\sigma_{\min} (S+\nabla^2 r(x))\geq \sigma_{\min} (S)+\sigma_{\min} (\nabla^2 r(x))\geq \sigma_{\min} (S)-\|\nabla^2 r(x)\|\geq \sigma_{\min} (S)-\varepsilon$. By the Bakry--\'Emery curvature criterion \citep[Prop.~4.8.1]{bakry2013analysis}, this implies $\gamma(\varepsilon)=\sigma_1(\varepsilon)\geq \inf_x\sigma_{\min} (\nabla^2 V(x))\geq \sigma_{\min} (S)-\varepsilon$.

Combining the upper and lower bounds shows that the spectral gap of $-\mathcal L_\varepsilon$ differs from that of the Gaussian reference by at most $\mathcal O(\varepsilon)$, completing the proof.
\end{proof}

\subsection{Preparatory results for Theorem~\ref{thm:spectral-gap-stability}}
We begin by introducing some useful notation and preliminary results.
Under the hypotheses of Theorem~\ref{thm:spectral-gap-stability},
\[
    V(x)=\frac{1}{2}x^\top S x+r_S(x),
    \qquad
    \nabla r_S(0)=0,
    \qquad
    \delta_S=\sup_x\|\nabla^2 r_S(x)\|_2 ,
\]
the Fisher information matrix satisfies
\[
    \mathcal F_\pi
    =\mathbb E_\pi[\nabla^2V(X)]
    =S+\mathbb E_\pi[\nabla^2r_S(X)],
    \qquad
    \|\mathcal F_\pi-S\|_2\leq \delta_S .
\]
Rewrite the same potential around the $\mathcal F_\pi$-induced quadratic reference:
\[
    V(x)=\frac12x^\top\mathcal F_\pi x+\widetilde r(x),
    \qquad
    \widetilde r(x):=r_S(x)+\frac12x^\top(S-\mathcal F_\pi)x .
\]
Then
\[
    \widetilde\delta:=\sup_x\|\nabla^2\widetilde r(x)\|_2
    \leq \delta_S+\|S-\mathcal F_\pi\|_2
    \leq 2\delta_S .
\]
The centering is preserved: \(\nabla\widetilde r(0)=\nabla r_S(0)=0\).
For the rest of the proof, write \(\mathcal F=\mathcal F_\pi\), \(r=\widetilde r\), and \(\delta=\widetilde\delta\). This reduces the theorem to the perturbative argument around the Gaussian reference \(\mathcal N(0,\mathcal F_\pi^{-1})\).
For any antisymmetric matrix $J$, let
\[
    \mathcal{L}_J
    = \Delta - (I+J)\nabla V\cdot\nabla
    = \Delta - \nabla V \cdot \nabla - (J \nabla V) \cdot \nabla
\]
be the nonreversible Langevin generator on $L^2(\pi)$, where $\pi \propto e^{-V}$. The theorem concerns $J=J^*$.
The Courant-Fischer variational principle used in the proof of Proposition~\ref{prop:stability:spectrum} applies only to self-adjoint (reversible) generators, so for $\mathcal{L}_{J^*}$ we use resolvent perturbation techniques for sectorial operators \citep[Ch.~IV--V]{kato2013perturbation}.

As in the proof of Proposition~\ref{prop:stability:spectrum}, set $R = \nabla r$. Since \(\nabla r(0)=0\), we have $\nabla V = \mathcal F x + R$ and $\|R(x)\| \leq \delta\|x\|$.
Let $U:L^2(\pi)\to L^2(dx)$ denote the ground-state transform $(Uf)(x)=e^{-V(x)/2}f(x)$, and define
\[
    \tilde{L}_J := -U\mathcal{L}_J U^{-1}.
\]
Then
\begin{equation}\label{eq:conjugated-op}
    \tilde{L}_J = -\Delta + \frac{1}{4}|\nabla V|^2 - \frac{1}{2}\Delta V + (J \nabla V) \cdot \nabla.
\end{equation}
Here the zeroth-order term from conjugating $(J\nabla V)\cdot\nabla$ vanishes because $(J\nabla V)\cdot\nabla V=0$.
The spectra of $-\mathcal{L}_J$ and $\tilde{L}_J$ are identical, with $0$ as the eigenvalue of smallest real part.
Moreover, the irreversible part is skew-adjoint in $L^2(\pi)$, since $\nabla\cdot(J\nabla V\,e^{-V})=0$; hence $-\mathcal L_J$ is accretive and every eigenvalue of $\tilde L_J$ has nonnegative real part.
Finally, decompose $\tilde{L}_J = \tilde{L}_J^{(0)} + \mathcal{W}_J$, where
\begin{align*}
    \tilde{L}_J^{(0)}
    &= -\Delta + \frac{1}{4}|\mathcal F x|^2 - \frac{1}{2}\mathrm{Tr}(\mathcal F) + (J\mathcal F x) \cdot \nabla,\\
    \mathcal{W}_J
    &= \frac{1}{2}(\mathcal F x) \cdot R + \frac{1}{4}|R|^2 - \frac{1}{2}\Delta r
    +(JR) \cdot \nabla .
\end{align*}
Here $\tilde{L}_J^{(0)}$ is the Gaussian reference operator associated with \(V_0(x)=\frac12 x^\top\mathcal F x\), i.e., the Gaussian distribution \(\mathcal N(0,\mathcal F^{-1})\), and $\mathcal{W}_J$ collects the non-Gaussian perturbation terms.

We recall the key spectral properties of this Gaussian reference operator from \citet{Lelivre2013}.
\begin{lemma}[{\citealp[Theorem~1, Proposition~10]{Lelivre2013}}]\label{lem:gaussian}
For the Gaussian potential $V_0(x) = \frac{1}{2}x^\top \mathcal F x$, equivalently the reference distribution \(\mathcal N(0,\mathcal F^{-1})\), any $J^* \in \mathcal J_\mathrm{s}(\mathcal F)$ with diagonalizable drift matrix $(I+J^*)\mathcal F$ achieves spectral gap
\[
\gamma^{(0)}(J^*) = \frac{\mathrm{Tr}(\mathcal F)}{d}.
\]
Moreover, the equivalent matrix $\tilde B_{J^*}:=\mathcal F+\mathcal F^{1/2}J^*\mathcal F^{1/2}$ is diagonalizable, so its eigenvalues are semi-simple. The spectrum of $\tilde{L}_{J^*}^{(0)}$ is purely discrete with eigenvalues $\{\sum_{k=1}^d n_k \sigma_k((I+J^*)\mathcal F) : n_k \in \mathbb{N}_0\}$, and $\Re(\sigma_k((I+J^*)\mathcal F)) = \frac{\mathrm{Tr}(\mathcal F)}{d}$ for all $k$. The operator has a generalized Hermite eigenbasis, so the eigenvalues of the Gaussian reference operator $\tilde{L}_{J^*}^{(0)}$ listed above are semi-simple.
\end{lemma}

\begin{lemma}[{adapted from \citealp[Proposition~9]{Lelivre2013}}]\label{lem:compact}
For any antisymmetric $J$ and any Gaussian reference \(\mathcal N(0,Q^{-1})\) with precision matrix $Q \succ 0$, define
\[
    \tilde{L}_{J,Q}^{(0)}
    := -\Delta+\frac{1}{4}|Qx|^2-\frac{1}{2}\mathrm{Tr}(Q)+(JQx)\cdot\nabla .
\]
Then $\tilde{L}_{J,Q}^{(0)}$ is maximal accretive and sectorial with compact resolvent. Its kernel is one-dimensional, spanned by $e^{-x^\top Qx/4}$. In the proof we use $Q=\mathcal F$ and write $\tilde{L}_{J}^{(0)}=\tilde{L}_{J,\mathcal F}^{(0)}$.
\end{lemma}
\begin{rmk}[Coordinate convention]\label{rmk:coordinate-convention}
\citet[Proposition~9]{Lelivre2013} establishes these properties for the conjugated operator $-\tilde{\mathcal L}_J$ in isotropized coordinates $y = Q^{1/2}x$, where the kernel is $e^{-|y|^2/4}$.
Our operator $\tilde{L}_{J,Q}^{(0)}$ acts on $L^2(\mathbb R^d, dx)$ in the original $x$~coordinates.
The change of variables $y = Q^{1/2}x$ induces a unitary map $T\colon L^2(dy)\to L^2(dx)$ via $(Tf)(x) = |\!\det Q|^{1/4}f(Q^{1/2}x)$, and
$\tilde{L}_{J,Q}^{(0)} = T(-\tilde{\mathcal L}_J^{\mathrm{Lel}})T^{-1}$,
where $\tilde{\mathcal L}_J^{\mathrm{Lel}}$ denotes Leli\`evre et al.'s operator (with nonpositive spectrum).
All operator-theoretic properties cited above---compact resolvent, maximal accretivity, the sectorial property, and spectral structure---are invariant under unitary equivalence, so they transfer directly. The kernel transforms as $e^{-|y|^2/4}\mapsto e^{-x^\top Qx/4}$.
\end{rmk}

We also need the relative boundedness of the non-Gaussian part.

\begin{lemma}[Relative bound for the non-Gaussian part]\label{lem:rel_bound}
For each fixed antisymmetric \(J\), the perturbation $\mathcal{W}_J$ is $\tilde{L}_J^{(0)}$-bounded: there exists $C=C(\mathcal F,J,d) > 0$ such that
\begin{equation*}
    \|\mathcal{W}_J u\| \leq C\delta \bigl(\|\tilde{L}_J^{(0)} u\| + \|u\|\bigr)
\end{equation*}
for all $u$ in the domain of $\tilde{L}_J^{(0)}$.
\end{lemma}
\begin{proof}[Proof of Lemma~\ref{lem:rel_bound}]\label{proof:lem:rel_bound}
Fix \(J\) and write \(\mathcal W=\mathcal W_J\). Consider the harmonic oscillator $H_0 := -\Delta + \frac{1}{4}|\mathcal F x|^2$, so that
\[
    \tilde{L}_J^{(0)} = H_0-\frac{1}{2}\mathrm{Tr}(\mathcal F)+(J\mathcal F x)\cdot\nabla .
\]
In the notation of \citet{Lelivre2013}, using the annihilation and creation operators $a_j = \partial_{x_j} + \frac{(\mathcal F^{1/2}x)_j}{2}$ and $a_j^* = -\partial_{x_j} + \frac{(\mathcal F^{1/2}x)_j}{2}$, we have $H_0 = a^{*,T} \mathcal F a + \frac{1}{2}\mathrm{Tr}(\mathcal F)$ (cf.\ \citet[Equation~(54)]{Lelivre2013}). Standard harmonic oscillator estimates (see \citet[Chapter~X]{reed1972methods}) give
\[
    \||x|^2u\|+\sum_{j,k}\|x_j\partial_k u\|
    \leq C(\|H_0u\|+\|u\|).
\]

\textbf{Multiplication terms.} The terms $\frac{1}{2}(\mathcal F x) \cdot R$, $\frac{1}{4}|R|^2$, and $-\frac{1}{2}\Delta r$ are bounded pointwise by $C\delta(1 + |x|^2)$ using $\|R(x)\| \leq \delta\|x\|$ and $\|\nabla^2 r\|_\infty \leq \delta$. The harmonic oscillator estimate yields:
\[
\bigl\|\bigl[\tfrac{1}{2}(\mathcal F x) \cdot R + \tfrac{1}{4}|R|^2 - \tfrac{1}{2}\Delta r\bigr] u\bigr\| \leq C\delta (\|H_0 u\| + \|u\|).
\]

\textbf{First-order term.} For the term $(JR) \cdot \nabla$, use the same creation-annihilation estimate for $x_j\partial_k$.
Since $\|R(x)\| \leq \delta \|x\|$, we obtain
\[
\|(JR) \cdot \nabla u\| \leq \|J\| \delta \sum_{j,k} \|x_j \partial_k u\| \leq C\delta (\|H_0 u\| + \|u\|).
\]

\textbf{Relating $H_0$ to $\tilde{L}_J^{(0)}$.} The graph-norm estimate for the Gaussian quadratic operator \citep[Lemma~4 and Proposition~9]{Lelivre2013} gives, for fixed $\mathcal F,J$,
\[
    \|H_0u\|\leq C_{\mathcal F,J}\bigl(\|\tilde{L}_J^{(0)}u\|+\|u\|\bigr).
\]
Combining with the estimates for $\mathcal{W}$ completes the proof.
\end{proof}

\subsection{Proof of Theorem~\ref{thm:spectral-gap-stability}}\label{app:thm:spectral-gap-stability}

With these ingredients in hand, we can now write the proof of \Cref{thm:spectral-gap-stability}.
\begin{proof}
Apply the preceding notation with \(J=J^*\), and write \(\mathcal W:=\mathcal W_{J^*}\).
Since $\tilde{L}_{J^*}^{(0)}$ has compact resolvent (Lemma~\ref{lem:compact}) and $\mathcal{W}$ is $\tilde{L}_{J^*}^{(0)}$-bounded with relative bound $a = C\delta < 1$ for $\delta$ sufficiently small (Lemma~\ref{lem:rel_bound}), the perturbed operator $\tilde{L}_{J^*} = \tilde{L}_{J^*}^{(0)} + \mathcal{W}$ also has compact resolvent. To see this, it suffices to produce a single $z_0 \in \rho(\tilde{L}_{J^*})$ at which the resolvent is compact. The algebraic identity
\begin{equation*}
    \tilde{L}_{J^*} - z_0 = \bigl(I + \mathcal{W}(\tilde{L}_{J^*}^{(0)} - z_0)^{-1}\bigr)(\tilde{L}_{J^*}^{(0)} - z_0)
\end{equation*}
gives, whenever $I + \mathcal{W}(\tilde{L}_{J^*}^{(0)} - z_0)^{-1}$ is invertible,
\begin{equation}\label{eq:perturbed-resolvent}
    (\tilde{L}_{J^*} - z_0)^{-1} = (\tilde{L}_{J^*}^{(0)} - z_0)^{-1}\bigl(I + \mathcal{W}(\tilde{L}_{J^*}^{(0)} - z_0)^{-1}\bigr)^{-1}.
\end{equation}
Since $(\tilde{L}_{J^*}^{(0)} - z_0)^{-1}$ is compact and $(I + \mathcal{W}(\tilde{L}_{J^*}^{(0)} - z_0)^{-1})^{-1}$ is bounded, the product is compact.

It remains to show $\|\mathcal{W}(\tilde{L}_{J^*}^{(0)} - z_0)^{-1}\| < 1$ for a suitable $z_0$. By Lemma~\ref{lem:rel_bound}, for all $u \in D(\tilde{L}_{J^*}^{(0)})$:
\[
\|\mathcal{W}u\| \leq a\|\tilde{L}_{J^*}^{(0)} u\| + b\|u\|, \qquad a = C\delta,\; b \geq 0.
\]
Applying this to $u = (\tilde{L}_{J^*}^{(0)} - z_0)^{-1}v$ and using $\tilde{L}_{J^*}^{(0)}(\tilde{L}_{J^*}^{(0)} - z_0)^{-1} = I + z_0(\tilde{L}_{J^*}^{(0)} - z_0)^{-1}$:
\[
\|\mathcal{W}(\tilde{L}_{J^*}^{(0)} - z_0)^{-1}v\| \leq a\bigl(1 + |z_0|\,\|(\tilde{L}_{J^*}^{(0)} - z_0)^{-1}\|\bigr)\|v\| + b\,\|(\tilde{L}_{J^*}^{(0)} - z_0)^{-1}\|\,\|v\|.
\]
By the sectorial property of $\tilde{L}_{J^*}^{(0)}$ \citep[Proposition~9]{Lelivre2013}, its numerical range $W(\tilde{L}_{J^*}^{(0)}) := \{\langle u, \tilde{L}_{J^*}^{(0)} u\rangle / \|u\|^2 : u \in D(\tilde{L}_{J^*}^{(0)}),\, u \neq 0\}$ is contained in $\overline{\Sigma_\theta}$, which lies in the closed right half-plane. For any $z \notin \overline{W(T)}$, the Cauchy--Schwarz inequality gives $\|(T-z)u\| \geq |\langle (T-z)u, u\rangle| = |\langle Tu, u\rangle - z| \geq \mathrm{dist}(z, W(T))\|u\|$ for unit $u$, so $(T-z)$ is bounded below and $\|(T-z)^{-1}\| \leq 1/\mathrm{dist}(z, \overline{W(T)})$ \citep[Ch.~V, \S3.2]{kato2013perturbation}. For real $z_0 < 0$, every $w \in \overline{W(\tilde{L}_{J^*}^{(0)})}$ has $\Re(w) \geq 0$, so $\mathrm{dist}(z_0, \overline{W(\tilde{L}_{J^*}^{(0)})}) \geq |z_0|$ and therefore $\|(\tilde{L}_{J^*}^{(0)} - z_0)^{-1}\| \leq 1/|z_0|$. Thus
\[
\|\mathcal{W}(\tilde{L}_{J^*}^{(0)} - z_0)^{-1}\| \leq 2a + \frac{b}{|z_0|}.
\]
If $\delta$ is small enough that $2a = 2C\delta < 1$, then choosing $|z_0|$ large enough gives $\|\mathcal{W}(\tilde{L}_{J^*}^{(0)} - z_0)^{-1}\| < 1$. The Neumann series converges, $(I + \mathcal{W}(\tilde{L}_{J^*}^{(0)} - z_0)^{-1})^{-1}$ is bounded, and~\eqref{eq:perturbed-resolvent} confirms that $(\tilde{L}_{J^*} - z_0)^{-1}$ is compact.

Since $(\tilde{L}_{J^*} - z_0)^{-1}$ is compact, the spectrum of $\tilde{L}_{J^*}$ is purely discrete: every spectral point is an isolated eigenvalue of finite algebraic multiplicity, and the resolvent is meromorphic with finite-rank residues at each pole \citep[Ch.~III, \S6.5, Eq.~(6.32)]{kato2013perturbation}. The eigenvalue $0$ persists with eigenfunction $e^{-V/2} \in \ker(\tilde{L}_{J^*})$ for any $V$ and remains simple since $\pi \propto e^{-V}$ is the unique invariant measure of the ergodic diffusion.

Set $\gamma := \mathrm{Tr}(\mathcal F)/d$ and let $\theta < \pi/2$ denote the semi-angle of the sector $\Sigma_\theta = \{re^{i\phi}: r\geq 0, |\phi|\leq\theta\}$ containing the spectrum of $\tilde{L}_{J^*}^{(0)}$, with $\tan\theta \leq \|J^*\|/\sigma_{\min}(\mathcal F)$ by \citet[Equation~(55)]{Lelivre2013}.
Throughout the proof, all constants $C_1, \ldots, C_7$ depend only on $\mathcal F, J^*, d$ and are independent of $\delta$.

The proof proceeds in four steps: (1) a numerical range bound showing eigenvalues with small real part have small modulus; (2) a Riesz projection showing the ground state is the unique eigenvalue near the origin; (3) a resolvent estimate showing a strip is eigenvalue-free; (4) combining the previous steps to conclude.

\textbf{Step 1: Eigenvalues with small real part have small modulus.}
Let $\zeta$ be an eigenvalue of $\tilde{L}_{J^*}$ with unit eigenvector $u$, so $\zeta = \langle u, \tilde{L}_{J^*} u\rangle$. Since $\mathcal{W}$ has relative bound $<1$, the domains of $\tilde{L}_{J^*}$ and $\tilde{L}_{J^*}^{(0)}$ coincide, so
\[
\zeta = \langle u, \tilde{L}_{J^*}^{(0)} u\rangle + \langle u, \mathcal{W}u\rangle.
\]
By the sectorial property \citep[Proposition~9]{Lelivre2013}, the numerical range of $\tilde{L}_{J^*}^{(0)}$ lies in $\overline{\Sigma_\theta}$, i.e., $\langle u, \tilde{L}_{J^*}^{(0)} u\rangle / \|u\|^2 \in \overline{\Sigma_\theta}$ for all $u$ in the domain. Any point $w \in \overline{\Sigma_\theta}$ satisfies $|\mathrm{Im}(w)| \leq \tan\theta \cdot \Re(w)$ (since it lies in a cone of half-angle $\theta$), so
\begin{equation}\label{eq:num-range-sector}
|\mathrm{Im}\langle u, \tilde{L}_{J^*}^{(0)} u\rangle| \leq \tan\theta \cdot \Re\langle u, \tilde{L}_{J^*}^{(0)} u\rangle.
\end{equation}
To bound the perturbation term, note $\tilde{L}_{J^*}^{(0)} u = \zeta u - \mathcal{W}u$ and the relative bound (Lemma~\ref{lem:rel_bound}) give
$\|\tilde{L}_{J^*}^{(0)} u\| \leq |\zeta| + C\delta(\|\tilde{L}_{J^*}^{(0)} u\| + 1)$, hence $\|\tilde{L}_{J^*}^{(0)} u\| \leq (|\zeta| + C\delta)/(1-C\delta)$ for $C\delta < 1/2$. Substituting back:
\begin{equation}\label{eq:calWu-bound}
|\langle u, \mathcal{W}u\rangle| \leq \|\mathcal{W}u\| \leq C\delta(\|\tilde{L}_{J^*}^{(0)} u\| + 1) \leq \frac{C\delta}{1-C\delta}(|\zeta| + 1) \leq C'\delta(|\zeta| + 1),
\end{equation}
where $C' = C/(1-C\delta) \leq 2C$ since $C\delta < 1/2$.
We now combine these estimates. Taking imaginary parts of $\zeta = \langle u, \tilde{L}_{J^*}^{(0)} u\rangle + \langle u, \mathcal{W}u\rangle$ and applying the triangle inequality:
\[
|\mathrm{Im}(\zeta)| \leq |\mathrm{Im}\langle u, \tilde{L}_{J^*}^{(0)} u\rangle| + |\mathrm{Im}\langle u, \mathcal{W}u\rangle|.
\]
The first term is bounded by~\eqref{eq:num-range-sector}: $|\mathrm{Im}\langle u, \tilde{L}_{J^*}^{(0)} u\rangle| \leq \tan\theta \cdot \Re\langle u, \tilde{L}_{J^*}^{(0)} u\rangle$. To express this in terms of $\zeta$, take real parts of the same identity:
\[
\Re\langle u, \tilde{L}_{J^*}^{(0)} u\rangle = \Re(\zeta) - \Re\langle u, \mathcal{W}u\rangle \leq \Re(\zeta) + |\langle u, \mathcal{W}u\rangle| \leq \Re(\zeta) + C'\delta(|\zeta|+1),
\]
where the last step uses~\eqref{eq:calWu-bound}. The second term satisfies $|\mathrm{Im}\langle u, \mathcal{W}u\rangle| \leq |\langle u, \mathcal{W}u\rangle| \leq C'\delta(|\zeta|+1)$. Combining:
\[
|\mathrm{Im}(\zeta)| \leq \tan\theta\,[\Re(\zeta) + C'\delta(|\zeta|+1)] + C'\delta(|\zeta|+1).
\]
For $\Re(\zeta) \leq \epsilon$ (where $\epsilon = 2C_6\delta$, chosen in Step~4), we use $|\zeta| \leq |\Re(\zeta)| + |\mathrm{Im}(\zeta)| \leq \epsilon + |\mathrm{Im}(\zeta)|$ and set $A := (\tan\theta+1)C'\delta$:
\begin{align*}
|\mathrm{Im}(\zeta)| &\leq \tan\theta\cdot\epsilon + A\bigl(\epsilon + |\mathrm{Im}(\zeta)| + 1\bigr)\\
&= \tan\theta\cdot\epsilon + A(\epsilon + 1) + A\,|\mathrm{Im}(\zeta)|.
\end{align*}
Rearranging:
\[
(1-A)\,|\mathrm{Im}(\zeta)| \leq \tan\theta\cdot\epsilon + A(\epsilon + 1).
\]
For $\delta$ small enough that $A < 1/2$, this gives $|\mathrm{Im}(\zeta)| \leq C''\delta$ and hence $|\zeta| \leq \epsilon + C''\delta < \gamma/2$ for $\delta$ small.

\textbf{Step 2: Riesz projection near the origin.}
We show that exactly one eigenvalue of $\tilde{L}_{J^*}$ satisfies $|z| < \gamma/2$.
Let $\Gamma$ be the circle $\{|z| = \gamma/2\}$. For the unperturbed operator, $\Gamma$ encloses only the simple eigenvalue $0$ (the next eigenvalues have real part $\gamma$, hence modulus at least $\gamma>\gamma/2$). On $\Gamma$, the resolvent $(\tilde{L}_{J^*}^{(0)} - z)^{-1}$ is bounded: $\Gamma$ is at distance $\gamma/2$ from $0$ and at least $\gamma/2$ from all other eigenvalues, so $\sup_{z\in\Gamma}\|(\tilde{L}_{J^*}^{(0)} - z)^{-1}\| =: M_\Gamma < \infty$ by continuity on a compact subset of the resolvent set.
By Lemma~\ref{lem:rel_bound} and the identity $\tilde{L}_{J^*}^{(0)}(\tilde{L}_{J^*}^{(0)} - z)^{-1} = I + z(\tilde{L}_{J^*}^{(0)} - z)^{-1}$:
\begin{equation}\label{eq:calW-resolvent-bound}
    \|\mathcal{W}(\tilde{L}_{J^*}^{(0)} - z)^{-1}\| \leq C_1\delta\Big(1 + (1+|z|)\|(\tilde{L}_{J^*}^{(0)} - z)^{-1}\|\Big),
\end{equation}
which on $\Gamma$ (where $|z| = \gamma/2$) yields $\|\mathcal{W}(\tilde{L}_{J^*}^{(0)} - z)^{-1}\| \leq C_3\delta < 1$ for $\delta$ small.
The Neumann series converges on $\Gamma$, so the Riesz projection $P = \frac{1}{2\pi i}\oint_\Gamma (\tilde{L}_{J^*} - z)^{-1}\,dz$ is well-defined and satisfies $\mathrm{rank}(P) = \mathrm{rank}(P^{(0)}) = 1$ by stability of Riesz projections \citep[Chapter~IV, Theorem~3.16]{kato2013perturbation}. Since $0 \in \sigma(\tilde{L}_{J^*})$, it is the unique eigenvalue inside $\Gamma$.
By Step~1, any eigenvalue with real part below $\epsilon$ has modulus $O(\delta)<\gamma/2$ for $\delta$ small, hence lies inside $\Gamma$ and must be $0$.

\textbf{Step 3: Eigenvalue-free strip.}
Since $\tilde{L}_{J^*}^{(0)}$ has compact resolvent (Lemma~\ref{lem:compact}), its spectrum is purely discrete---every spectral point is an isolated eigenvalue of finite multiplicity. By Proposition~\ref{lem:gaussian}, these eigenvalues are organized into levels: $\{0\}$ at level $0$, and $n$-th level eigenvalues $\sum_{k=1}^d n_k \sigma_k((I+J^*)\mathcal F)$ with $\sum_k n_k = n$ satisfying real part \(n\gamma\) for $n \geq 1$.
In particular, the strip $\{z : 0 < \Re(z) < \gamma\}$ contains no eigenvalue of $\tilde{L}_{J^*}^{(0)}$.
We show that $\|\mathcal{W}(\tilde{L}_{J^*}^{(0)} - z)^{-1}\| < 1$ for all $z$ in the sub-strip $\mathcal S_\epsilon := \{z : \epsilon \leq \Re(z) \leq \gamma - \epsilon\}$, where $\epsilon > 0$ is chosen proportional to $\delta$.

We bound $(1+|z|)\|(\tilde{L}_{J^*}^{(0)} - z)^{-1}\|$ uniformly on $\mathcal{S}_\epsilon$ by splitting into two regions. Fix $\rho = \gamma/4$.
Write $\sigma_k^{(0)}:=\sigma_k((I+J^*)\mathcal F)$ for the first-level eigenvalues of the Gaussian reference operator.

\emph{Far from poles.} Let us set
\[
\mathcal{S}_\epsilon^{\mathrm{far}} = \{z \in \mathcal{S}_\epsilon : |z| \geq \rho \text{ and } |z - \sigma_k^{(0)}| \geq \rho \text{ for all first-level eigenvalues } \sigma_k^{(0)}\}.
\]
On this set, $z$ lies in the resolvent set of $\tilde{L}_{J^*}^{(0)}$, so $z \mapsto (1+|z|)\|(\tilde{L}_{J^*}^{(0)} - z)^{-1}\|$ is continuous. As $|z| \to \infty$ within $\mathcal{S}_\epsilon$, the real part stays in $[\epsilon, \gamma - \epsilon]$ while $|\mathrm{Im}(z)| \to \infty$, so $z$ eventually exits $\overline{\Sigma_\theta}$ and the resolvent estimate \citep[Ch.~V, \S3.2]{kato2013perturbation} gives $(1+|z|)\|(\tilde{L}_{J^*}^{(0)} - z)^{-1}\| \leq (1+|z|)/\mathrm{dist}(z, \overline{\Sigma_\theta}) \to 1/\cos\theta$. A continuous function with a finite limit at infinity attains a finite supremum $M = M(\mathcal F, J^*, d)$, independent of $\epsilon$. By~\eqref{eq:calW-resolvent-bound}, $\|\mathcal{W}(\tilde{L}_{J^*}^{(0)} - z)^{-1}\| \leq C_2\delta$ on $\mathcal{S}_\epsilon^{\mathrm{far}}$.

\emph{Near poles.} Let $\mathcal{S}_\epsilon^{\mathrm{near}} = \mathcal{S}_\epsilon \setminus \mathcal{S}_\epsilon^{\mathrm{far}}$, i.e., $z \in \mathcal{S}_\epsilon$ with $|z| < \rho$ or $|z - \sigma_k^{(0)}| < \rho$ for some first-level eigenvalue $\sigma_k^{(0)}$. Only the eigenvalues at level~$0$ (the origin) and level~$1$ (where real parts equal $\gamma$) border the strip; higher levels have real parts at least \(2\gamma\) and are at distance $\geq \gamma$ from $\mathcal{S}_\epsilon$.

By \Cref{lem:gaussian}, the relevant eigenvalues of the Gaussian reference operator are semi-simple. Hence the resolvent has a first-order pole at each such eigenvalue, with residue equal to the negative eigenprojection \citep[Ch.~III, \S6.5, Eq.~(6.32)]{kato2013perturbation}.

Near the eigenvalue $0$: the Laurent expansion gives $(\tilde{L}_{J^*}^{(0)} - z)^{-1} = -P_0/z + S_0(z)$, where $P_0$ is the eigenprojection onto the ground state and $S_0$ is holomorphic for $|z| < \rho$. Since $|z| \geq \epsilon$ and $|z| < \rho$ in this region, $(1+|z|)\|(\tilde{L}_{J^*}^{(0)} - z)^{-1}\| \leq (1+\rho)(\|P_0\|/\epsilon + \|S_0\|_\infty)$, so~\eqref{eq:calW-resolvent-bound} gives $\|\mathcal{W}(\tilde{L}_{J^*}^{(0)} - z)^{-1}\| \leq 2C_1\|P_0\|\delta/\epsilon + C_5\delta$.
The same estimate holds near each first-level eigenvalue $\sigma_k^{(0)}$ (with $\Re(\sigma_k^{(0)}) = \gamma$), where the resolvent has a first-order pole with eigenprojection $P_k$ and $|z - \sigma_k^{(0)}| \geq \epsilon$.
Combining all poles: $\|\mathcal{W}(\tilde{L}_{J^*}^{(0)} - z)^{-1}\| \leq C_6\delta/\epsilon + C_7\delta$ on $\mathcal{S}_\epsilon^{\mathrm{near}}$, where $C_6 := C_1\max(2\|P_0\|, \max_k(1+|\sigma_k^{(0)}|)\|P_k\|)$.

\textbf{Step 4: Choosing $\epsilon$ and concluding.}
Set $\epsilon := 2C_6\delta$. We require $\delta$ small enough that
\begin{enumerate}
    \item[(i)] $A = (\tan\theta+1)C'\delta < 1/2$ and $\epsilon + C''\delta < \gamma/2$ (Step~1: $|\zeta| < \gamma/2$),
    \item[(ii)] $C_3\delta < 1$ (Step~2: Neumann convergence on $\Gamma$),
    \item[(iii)] $C_7\delta < 1/2$ (Step~3, near poles: $C_6\delta/\epsilon + C_7\delta = 1/2 + C_7\delta < 1$),
    \item[(iv)] $C_2\delta < 1$ (Step~3, far from poles),
    \item[(v)] $\epsilon < \gamma/2$, i.e., $2C_6\delta < \gamma/2$ (sub-strip $\mathcal S_\epsilon$ is nonempty).
\end{enumerate}
All conditions are satisfied for $\delta < \delta_0 := \delta_0(\mathcal F, J^*, d)$.
The Neumann series
\[
(\tilde{L}_{J^*} - z)^{-1} = (\tilde{L}_{J^*}^{(0)} - z)^{-1}\sum_{n=0}^\infty \bigl(-\mathcal{W}(\tilde{L}_{J^*}^{(0)} - z)^{-1}\bigr)^n
\]
converges for all $z \in \mathcal S_\epsilon$, so $\mathcal S_\epsilon \cap \sigma(\tilde{L}_{J^*}) = \emptyset$.
By Step~2, the only eigenvalue with real part below $\epsilon$ is $0$. Therefore every nonzero eigenvalue has real part at least $\gamma - \epsilon = \gamma - 2C_6\delta$, and with $C := 2C_6$:
\begin{equation*}
    \gamma(J^*) := \inf\bigl\{\Re(\zeta) : \zeta \in \sigma(\tilde{L}_{J^*}),\; \zeta \neq 0\bigr\} \geq \frac{\mathrm{Tr}(\mathcal F)}{d} - C\delta,
\end{equation*}
where in this proof \(\mathcal F=\mathcal F_\pi\) and \(\delta=\widetilde\delta\leq2\delta_S\). Returning to the notation of the theorem gives
\[
    \gamma(J^*)\geq \frac{\mathrm{Tr}(\mathcal F_\pi)}{d}-2C\delta_S .
\]
Also,
\[
    \left|\frac{\mathrm{Tr}(\mathcal F_\pi)-\mathrm{Tr}(S)}{d}\right|
    \leq \|\mathcal F_\pi-S\|_2\leq\delta_S,
\]
which yields the equivalent lower bound relative to \(\mathrm{Tr}(S)/d\).
This completes the proof.
\end{proof}

\begin{rmk}[Non-diagonalizable references]
The diagonalizability assumption is used only in Step~3, where semi-simplicity makes the resolvent poles at the first spectral level of the Gaussian reference operator first order. If a relevant eigenvalue of this reference operator has a Jordan block of size \(m>1\), the corresponding pole may have order \(m\). The same perturbative picture then gives qualitative stability only at the weaker scale \(\delta^{1/m}\), rather than the linear \(O(\delta)\) loss proved above.
\end{rmk}

\section{Relating the Fisher information matrix to local or global geometry}\label{app:fisher-matrix-visual}
\begin{figure}[htbp]
\centering
    \begin{subfigure}[b]{0.7\textwidth}
        \includegraphics[width=\textwidth]{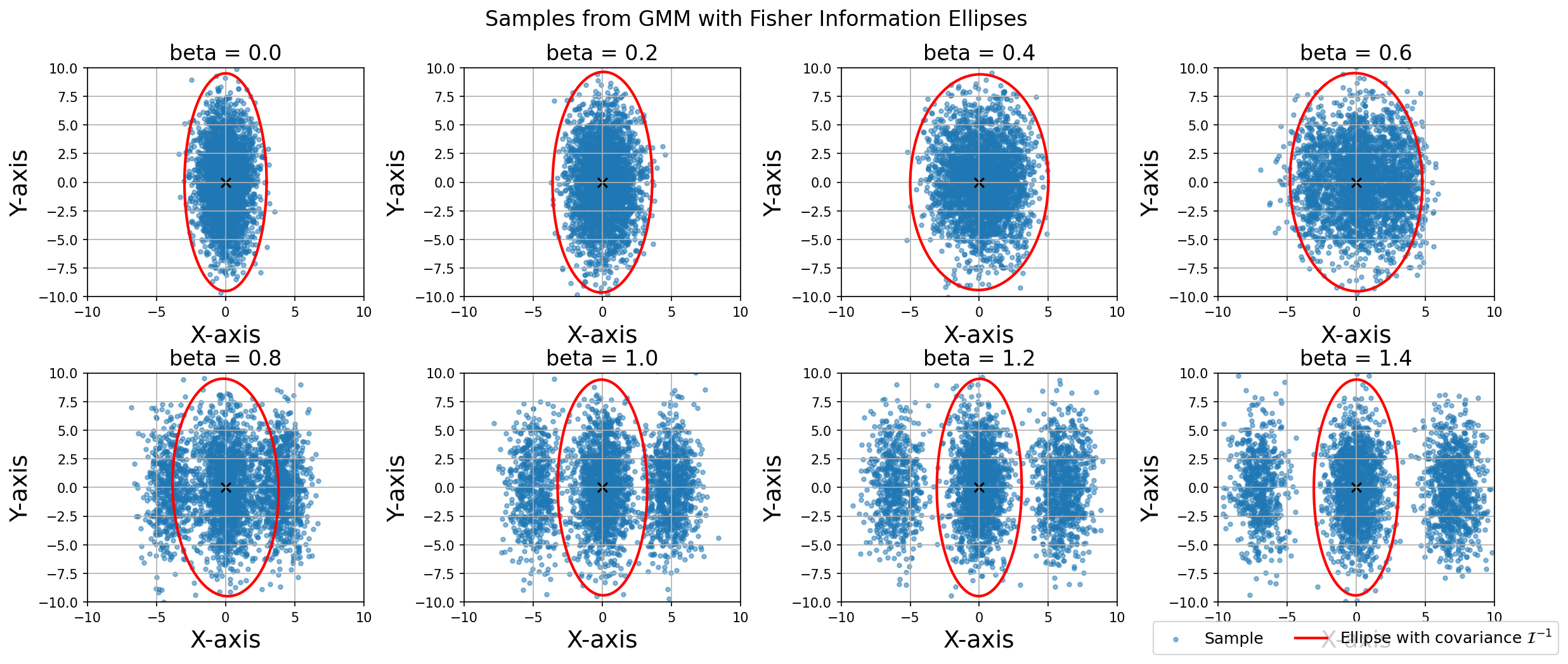}
    \end{subfigure}
\caption[Global and local Fisher matrix with mixed modes]{The Fisher information matrix interpolates between local and global geometry as modes merge. When modes are well separated, $\mathcal{F} \approx \sum_k w_k \mathcal{F}_k$ (weighted sum of per-component Fisher information matrices); as overlap increases, $\mathcal{F}$ transitions to reflect the merged mode structure.\label{fig:fisher_geometry2}}
\end{figure}

For a mixture $\pi = \sum_{k} w_k \pi_k$ with well-separated components, the density at most points is dominated by a single component $\pi_i$, so $\nabla \log \pi(x) \approx \nabla \log \pi_i(x)$ and the Fisher information matrix decomposes as $\mathcal{F} \approx \sum_k w_k \mathcal{F}_k$, where $\mathcal{F}_k := \mathbb{E}_{\pi_k}[\nabla \log \pi_k \, \nabla \log \pi_k^\top]$. As the modes overlap, this decomposition breaks down and \(\mathcal F\) transitions toward the geometry of the merged distribution.

\Cref{fig:fisher_geometry2} visualizes this transition using an ellipse associated with \(\mathcal F\). In this example, the off-diagonal entries vanish, so \(\mathcal F\) is diagonal and the ellipse is aligned with the coordinate axes. The lengths of its two semi-axes are determined by the two diagonal entries of \(\mathcal F\), thereby illustrating how the local geometry encoded by the Fisher matrix changes as the modes merge.

This also highlights the limitation of Fisher information-based preconditioning.
Although \(\mathcal F^{-1}\) adapts to local anisotropy, \(\mathcal F\) primarily reflects within-mode geometry and does not directly encode the inter-modal directions needed for global exploration.

\section{Additional experiments and numerical details}\label{app:add-experiment-details}
\subsection{Non-isotropic Gaussian}\label{app:numerical:iso-gaussian}
Additional results for MSE, bias, and variance across different sampling methods under fixed budget and fixed simulation time are provided in \Cref{fig:app:exp1:1} and~\Cref{fig:app:exp1:2}. \Cref{fig:app:exp1:3} further reports estimation error for other statistics corresponding to different proxies.
\begin{figure}[htbp]
\centering
\begin{minipage}[t]{0.89\textwidth}
    \centering
    \includegraphics[width=\textwidth]{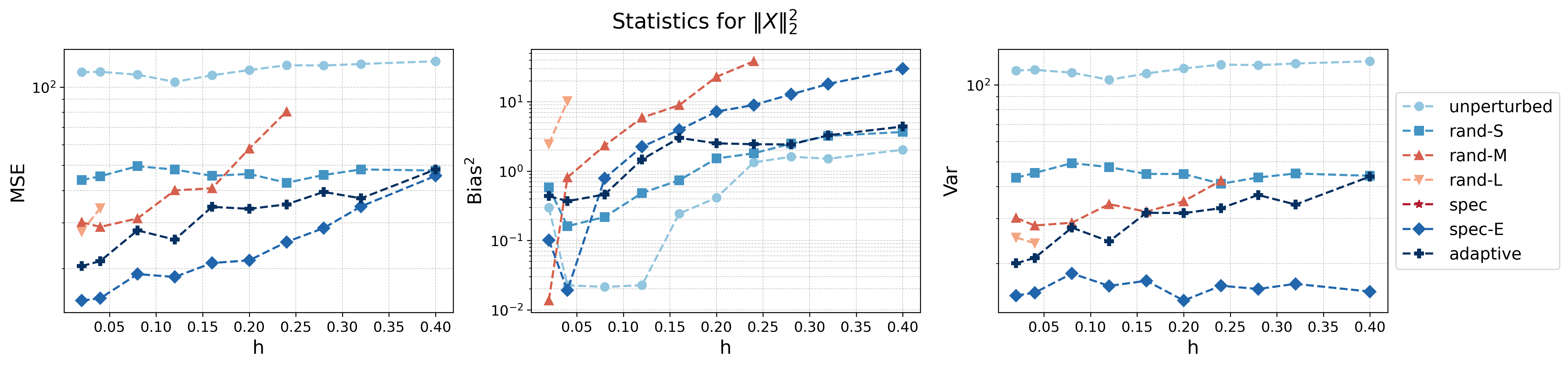}
    \includegraphics[width=\textwidth]{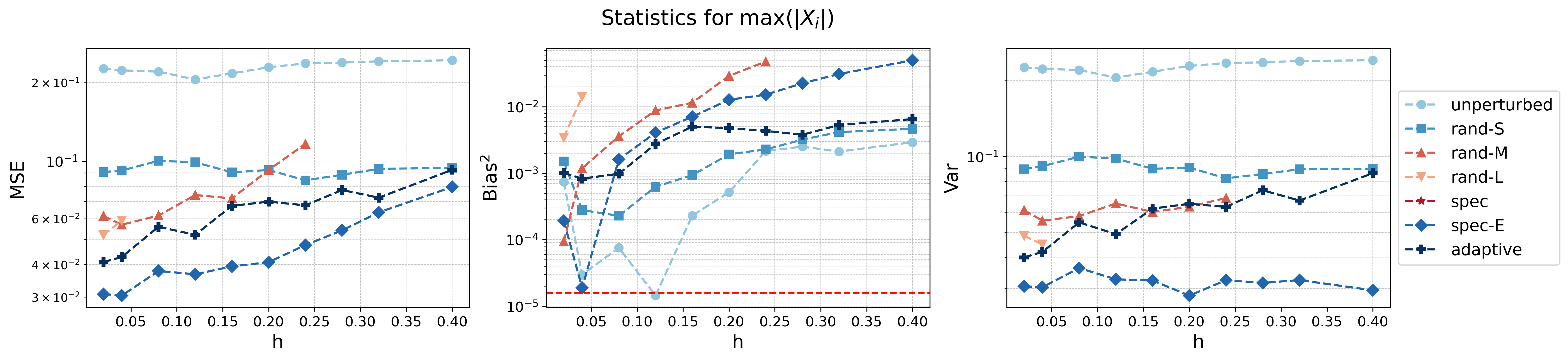}
    \includegraphics[width=\textwidth]{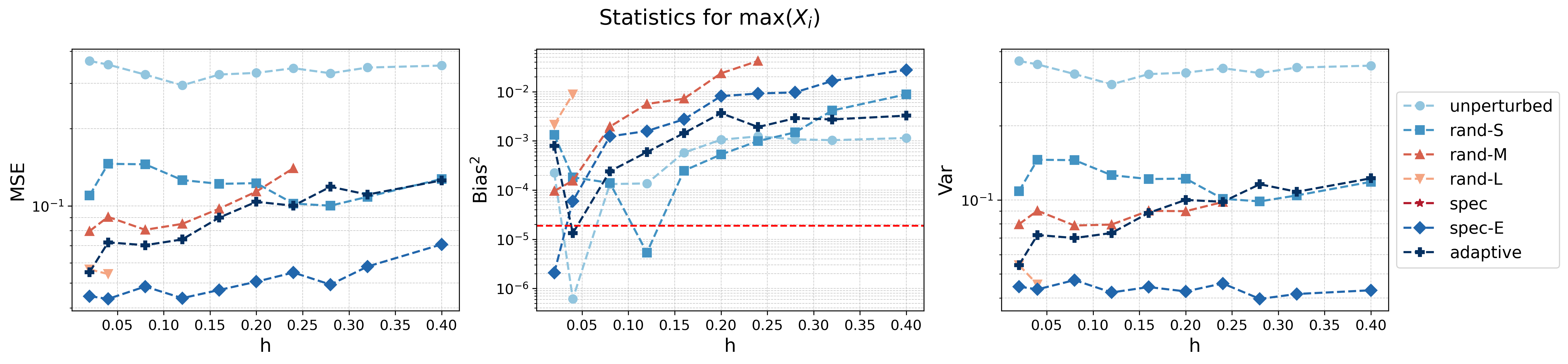}
    \includegraphics[width=\textwidth]{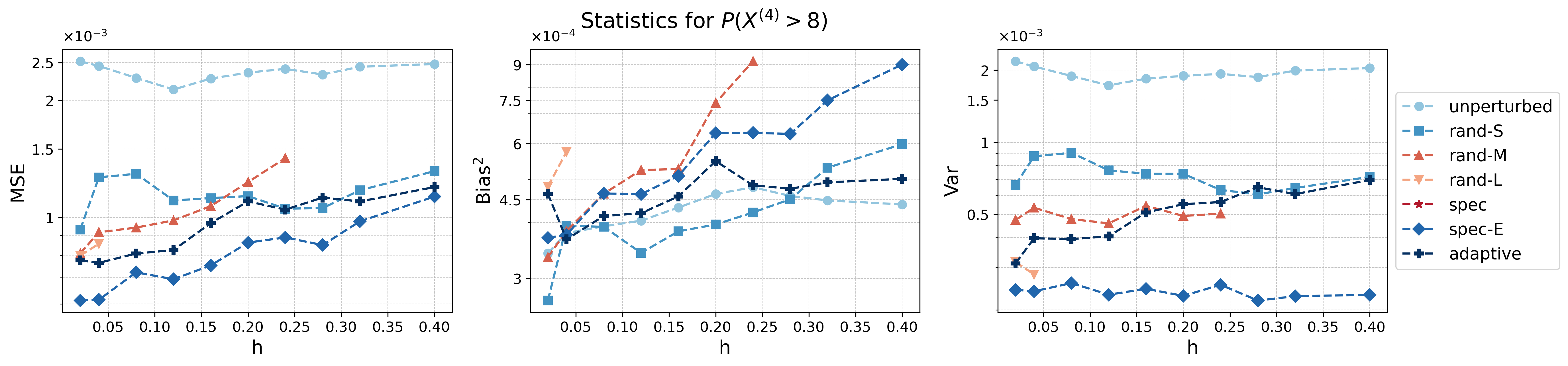}
\end{minipage}
\caption[Non-isotropic Gaussian, MSE vs. step size, fixed time]{Additional result for non-isotropic Gaussian, fixed simulation time $T = hK= 10{,}000$. Plots show MSE, squared bias, and variance for different observables for varying $h$. Statistics are computed using $M = 512$ replicated MCMC chains.
Our method attains the lowest MSE over a broad range of $h$, while
alternatives either converge more slowly or exhibit bias blow-ups at some $h$.\label{fig:app:exp1:1}}
\end{figure}
\begin{figure}[htbp]
\centering
\begin{minipage}[t]{0.89\textwidth}
    \centering
    \centering
    \includegraphics[width=\textwidth]{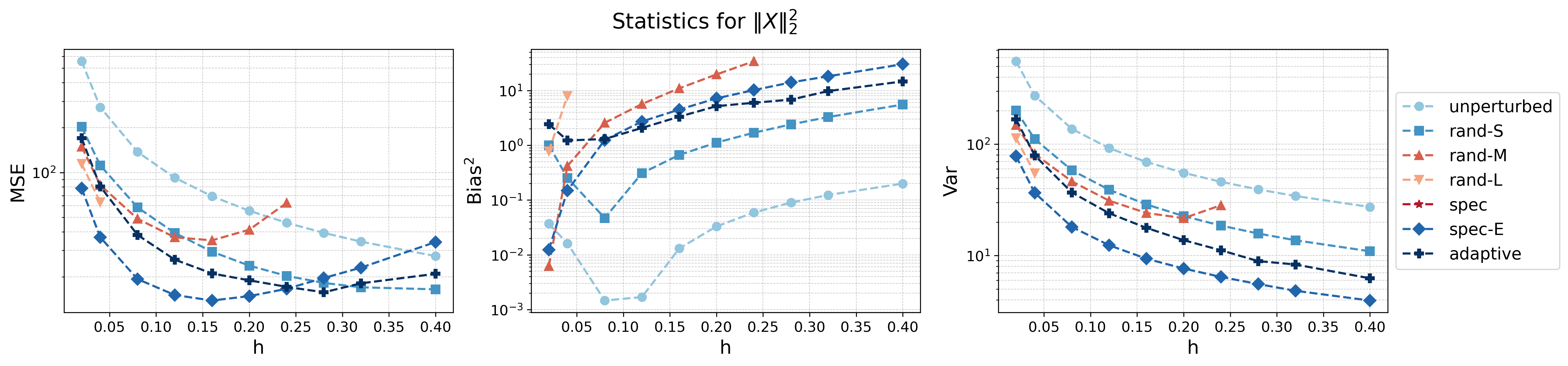}
    \includegraphics[width=\textwidth]{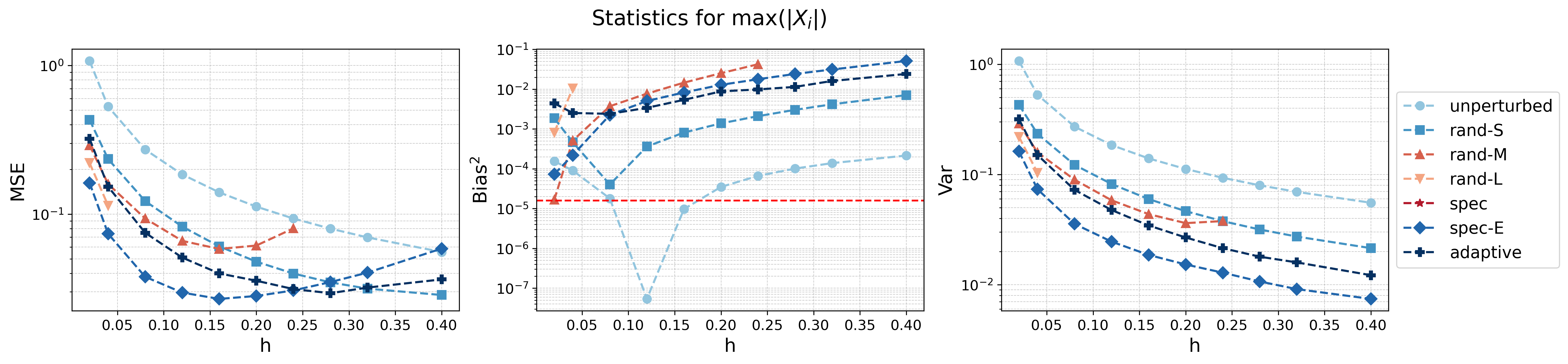}
    \includegraphics[width=\textwidth]{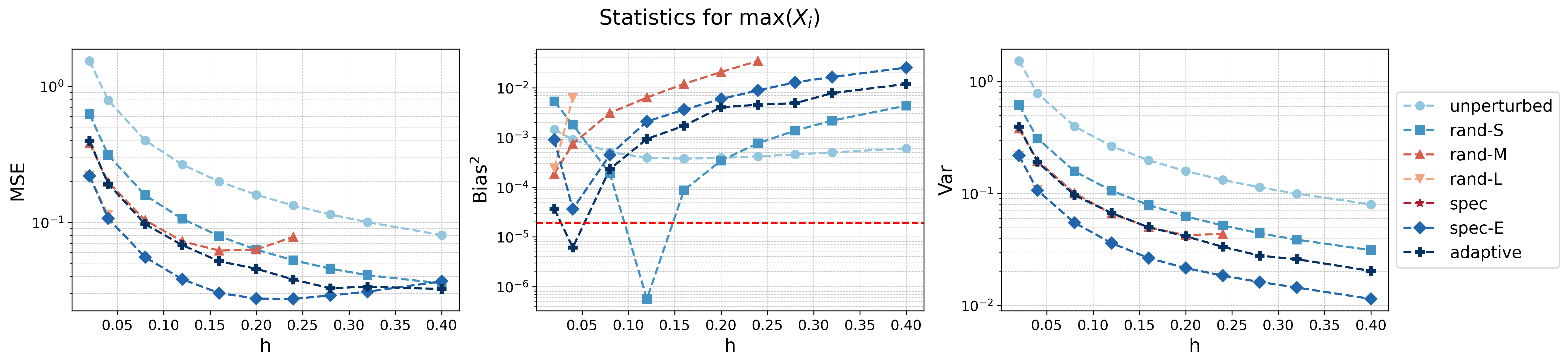}
    \includegraphics[width=\textwidth]{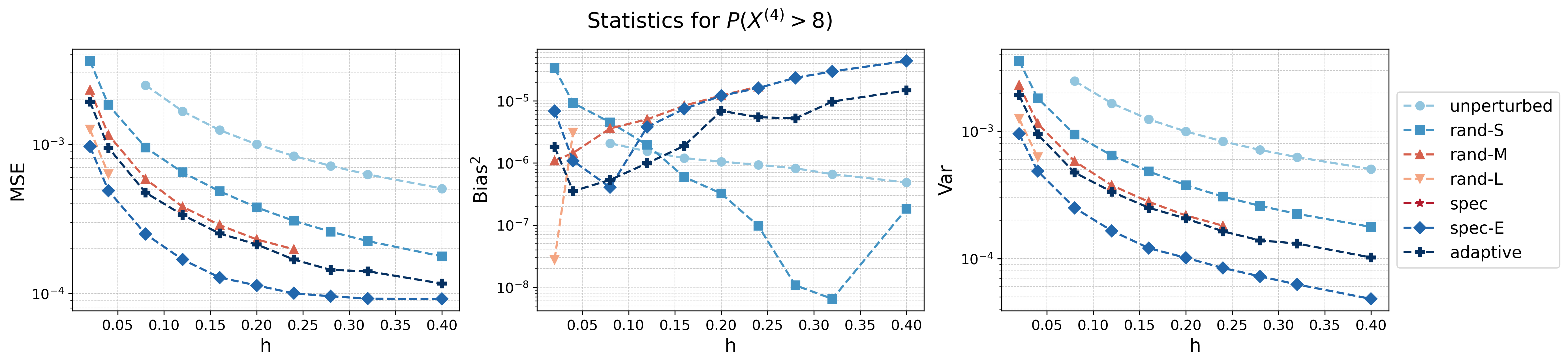}
\end{minipage}
\caption[Non-isotropic Gaussian, MSE vs. step size, fixed budget]{Additional result for non-isotropic Gaussian, fixed computational budget $K = 10^5$. Plots show MSE, squared bias, and variance of different observables  for varying $h$. Statistics are computed using $M = 512$ replicated MCMC chains.
Our method attains the lowest MSE over a broad range of $h$, while
alternatives either converge more slowly or exhibit bias blow-ups at some $h$.\label{fig:app:exp1:2}}
\end{figure}

\begin{figure}[htbp]
\centering
\begin{minipage}[t]{0.89\textwidth}
    \centering
    \includegraphics[width=\textwidth]{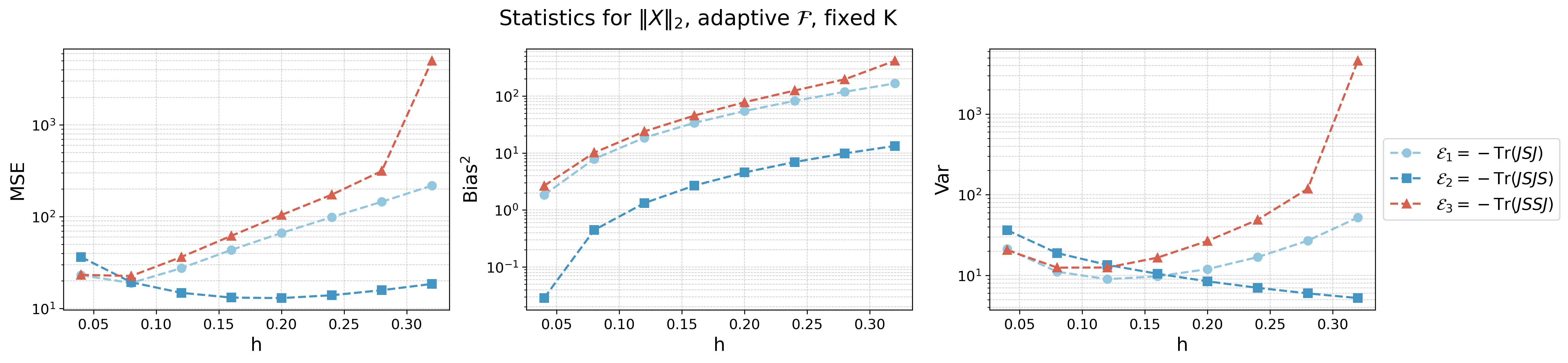}
    \includegraphics[width=\textwidth]{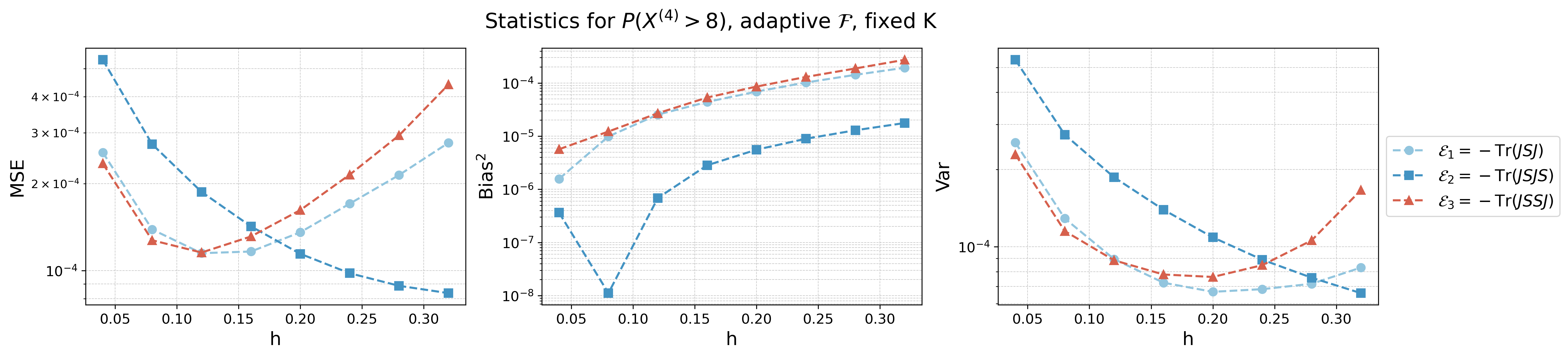}
    \includegraphics[width=\textwidth]{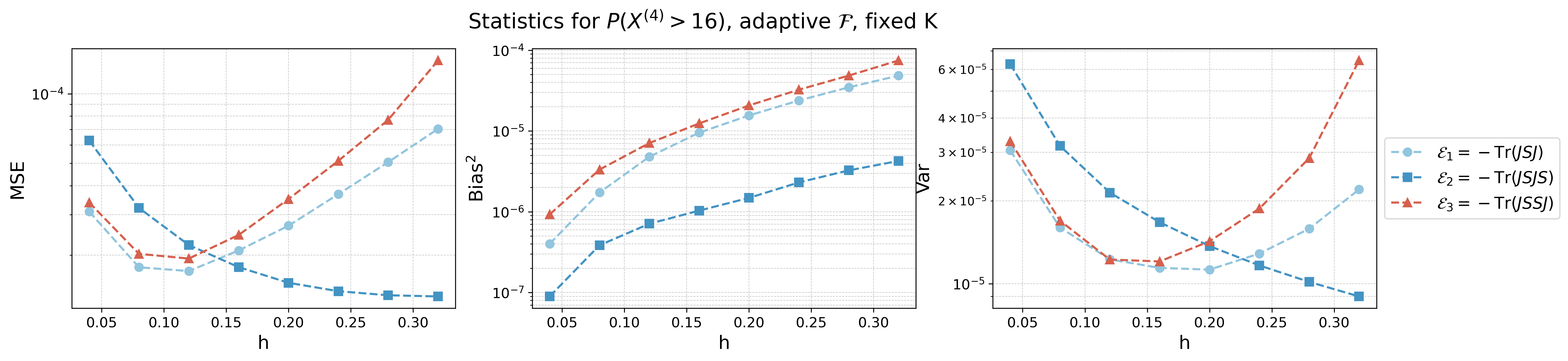}
\end{minipage}
\caption[Non-isotropic Gaussian, adaptive]{Additional results for non-isotropic Gaussian, fixed budget $K = 10^5$. Plots show MSE, squared bias, and variance of different observables for varying $h$. Here $\mathcal F=S$, $M = 512$.
The results show that different proxies for numerical error yield consistent behavior in the resulting statistical estimates. The adaptive Fisher-information-matrix scheme produces a stable but slower-converging algorithm; however, as $K$ increases, this adaptive approach also converges to the non-adaptive solution.\label{fig:app:exp1:3}}
\end{figure}

\subsection{Mixture of Gaussians}\label{app:numerical:mix-gaussian}
The target density is
\begin{equation}
    \pi(x) \;=\; \sum_{k=1}^{5} w_k \, \mathcal{N}(x \mid \mu_k, \Sigma_k),
\end{equation}
where the mixture parameters are:
\begin{align*}
    &\text{Means:} \quad
    \mu_1 = (-30, 0, 0)^\top, \quad
    \mu_2 = (-15, 0, 0)^\top, \quad
    \mu_3 = (0, 0, 0)^\top, \\
    &\hspace{4.5cm}
    \mu_4 = (15, 0, 0)^\top, \quad
    \mu_5 = (30, 0, 0)^\top, \\[0.5em]
    &\text{Weights:} \quad
    w = (0.3,\, 0.1,\, 0.2,\, 0.1,\, 0.3), \\[0.5em]
    &\text{Covariances:} \quad
    \Sigma_1 = \begin{bmatrix} 5 & 0 & 0 \\ 0 & 1 & \frac{1}{5} \\ 0 & \frac{1}{5} & \frac{1}{5} \end{bmatrix}, \quad
    \Sigma_2 = \begin{bmatrix} 10 & 0 & 0 \\ 0 & 1 & \frac{1}{5} \\ 0 & \frac{1}{5} & \frac{1}{5} \end{bmatrix}, \\
    &
    \hspace{2.7cm}\Sigma_3 = \begin{bmatrix} 5 & 0 & 0 \\ 0 & 1 & \frac{1}{5} \\ 0 & \frac{1}{5} & \frac{1}{5} \end{bmatrix}, \quad
    \Sigma_4 = \begin{bmatrix} 10 & 0 & 0 \\ 0 & 1 & \frac{1}{5} \\ 0 & \frac{1}{5} & \frac{1}{5} \end{bmatrix},
    \Sigma_5 = \begin{bmatrix} 5 &  0 & 0 \\ 0 & 1 & \frac{1}{5} \\ 0 & \frac{1}{5} & \frac{1}{5} \end{bmatrix}.
\end{align*}
See \Cref{fig:multi-gaussian-density} for demonstration.

\begin{figure}
    \centering
    \includegraphics[width=0.8\linewidth]{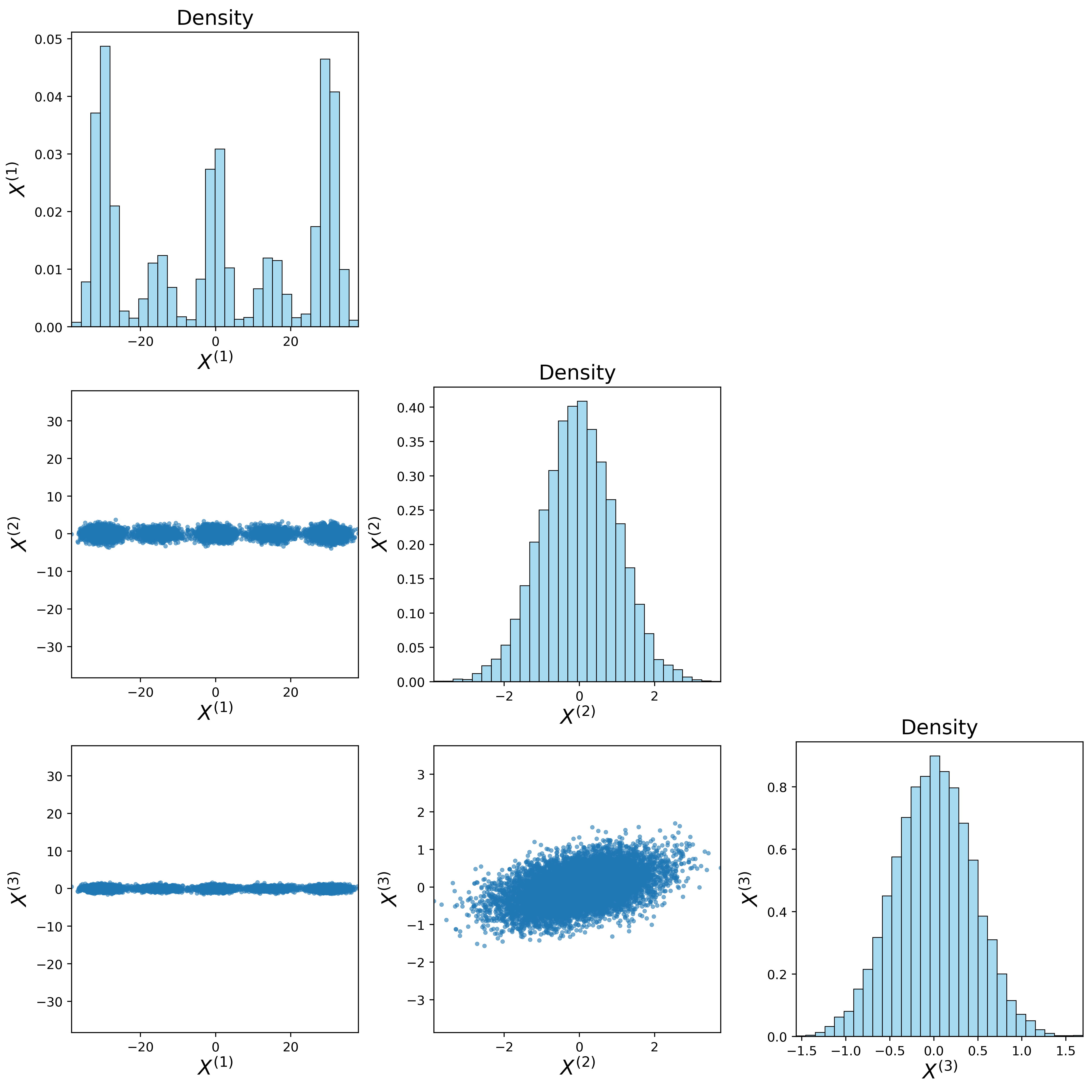}
    \caption{Mixture of Gaussians target distribution (five components in $\mathbb{R}^3$). Marginal densities and pairwise projections show alignment between global and local geometry.}
    \label{fig:multi-gaussian-density}
\end{figure}

Additional results for MSE, bias, and variance across different sampling methods under fixed budget and fixed simulation time are provided in \Cref{fig:app:exp2:1} and \Cref{fig:app:exp2:2}.
\begin{figure}[htbp]
\centering
\begin{minipage}[t]{0.8\textwidth}
    \centering
    \begin{subfigure}[b]{\textwidth}
        \includegraphics[width=\textwidth]{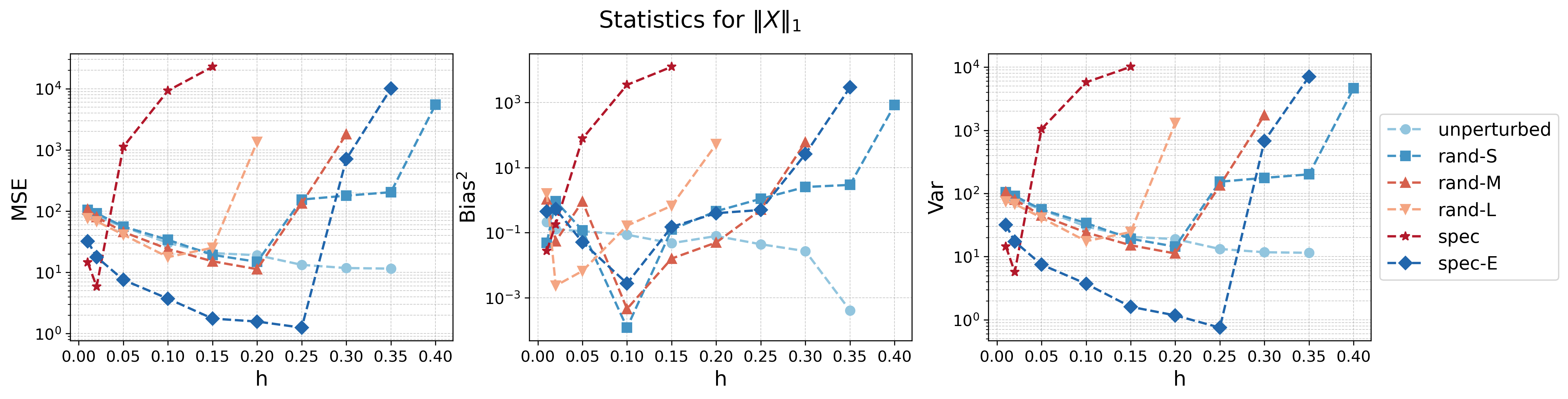}
        \includegraphics[width=\textwidth]{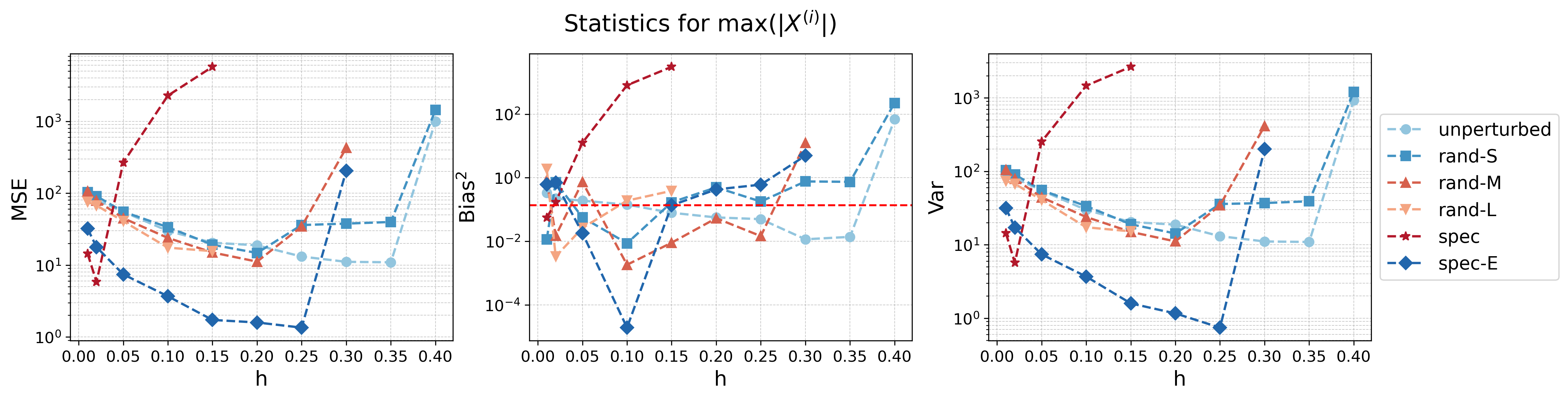}
        \includegraphics[width=\textwidth]{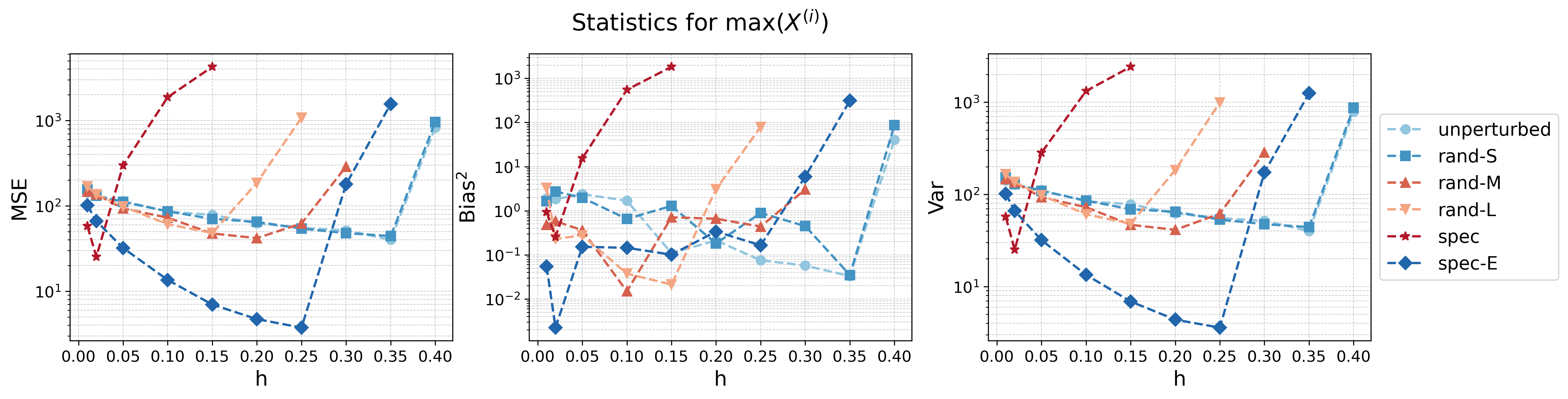}
        \includegraphics[width=\textwidth]{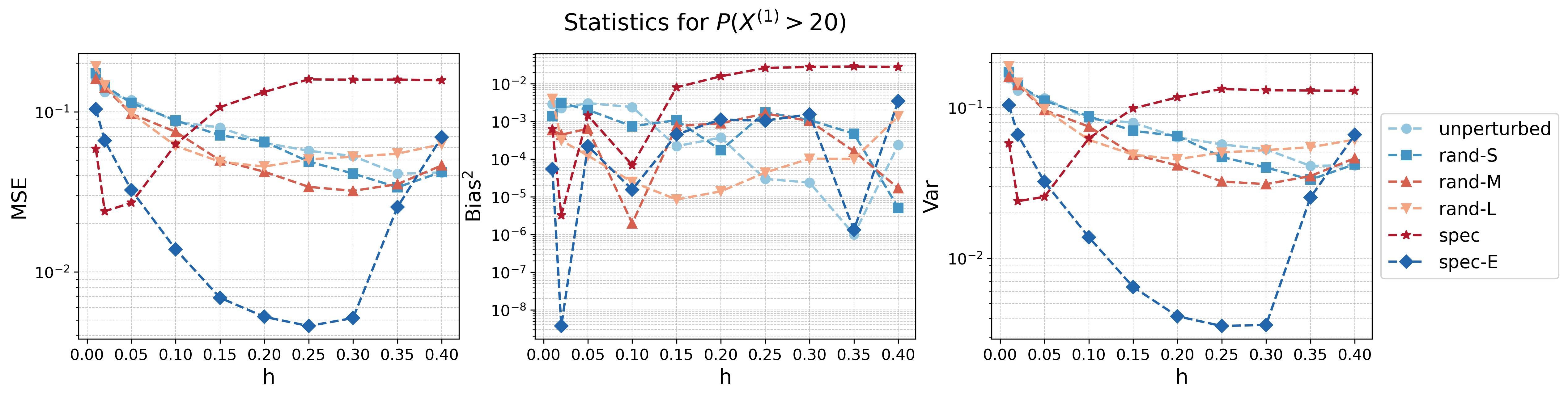}
        \includegraphics[width=\textwidth]{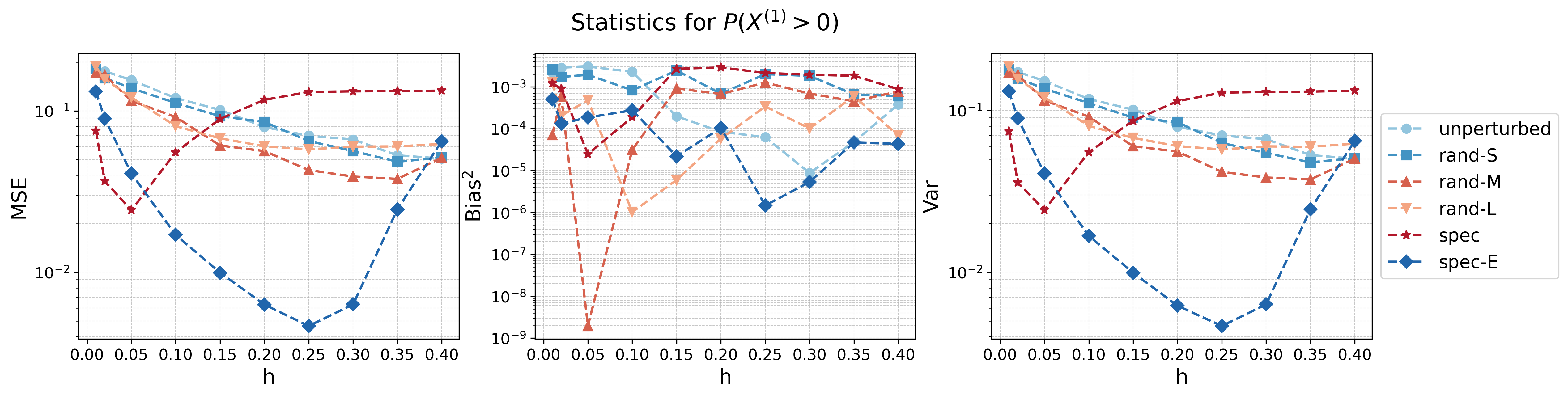}
    \end{subfigure}
\end{minipage}
\caption[Multiple Gaussian, MSE vs. step size]{Additional result for mixture of Gaussian under fixed total steps $K = 100000$. Plots show : MSE vs. step size under varying $h$, $M=128$.
Our method achieves the lowest MSE across a broad range of $h$, whereas competing schemes either converge more slowly or suffer bias blow-ups at certain step sizes. \label{fig:app:exp2:1}}
\end{figure}

\begin{figure}[htbp]
\centering
\begin{minipage}[t]{0.85\textwidth}
    \centering
    \begin{subfigure}[b]{\textwidth}
        \includegraphics[width=\textwidth]{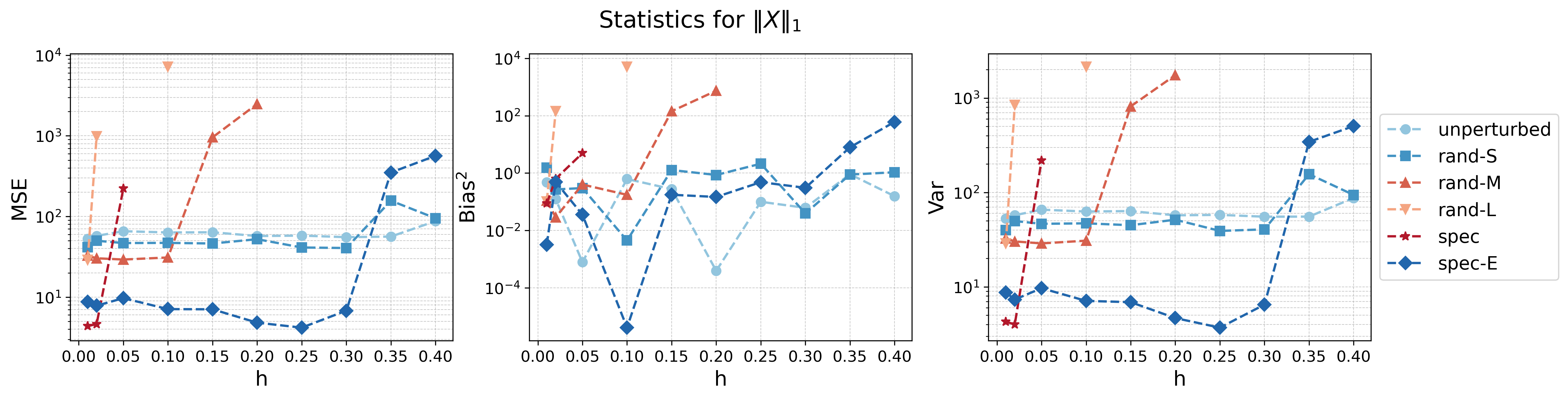}
        \includegraphics[width=\textwidth]{images/multi-gaussian/fix_time_3.png}
        \includegraphics[width=\textwidth]{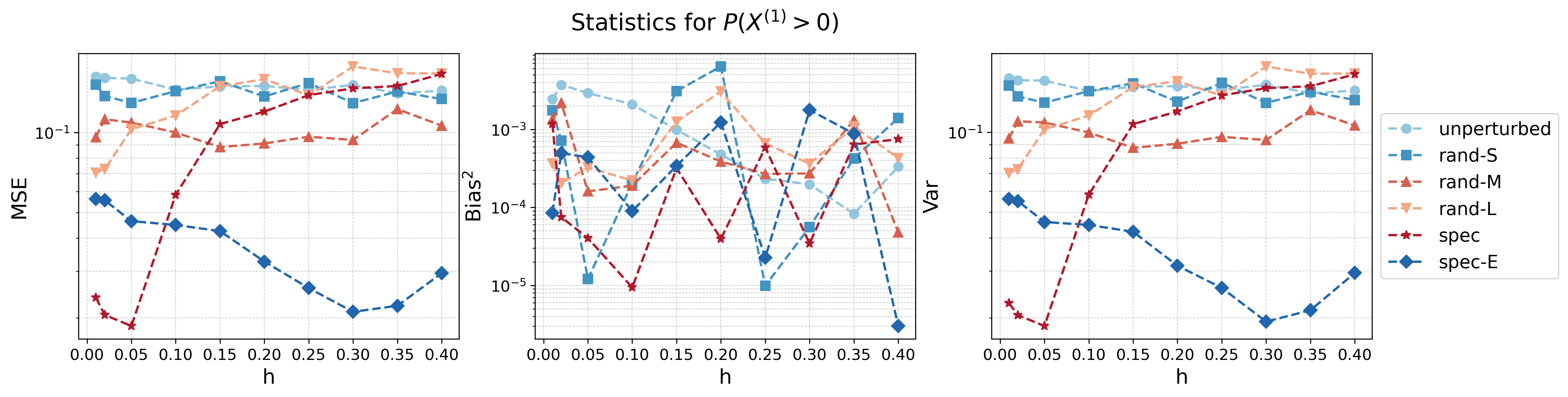}
    \end{subfigure}
\end{minipage}
\caption[Multiple Gaussian, MSE vs. step size]{Additional result for mixture of Gaussian under fixed total time $T = 4000$.
Plots show MSE vs. step size under varying $h$, $M=128$.
Our method achieves the lowest MSE across a broad range of $h$, whereas competing schemes either converge more slowly or suffer bias blow-ups at certain step sizes. \label{fig:app:exp2:2}}
\end{figure}

\subsection{Bayesian logistic regression}\label{app:numerical:blr}
The posterior of Bayesian logistic regression task is illustrated in \Cref{fig:blr_density}, which is uni-modal but non-symmetric.
Additional results for MSE, bias, and variance across different sampling methods under fixed budget and fixed simulation time are provided in \Cref{fig:app:exp3:1} and \Cref{fig:app:exp3:2}.
\begin{figure}[htbp]
\centering
    \includegraphics[width=0.6\textwidth]{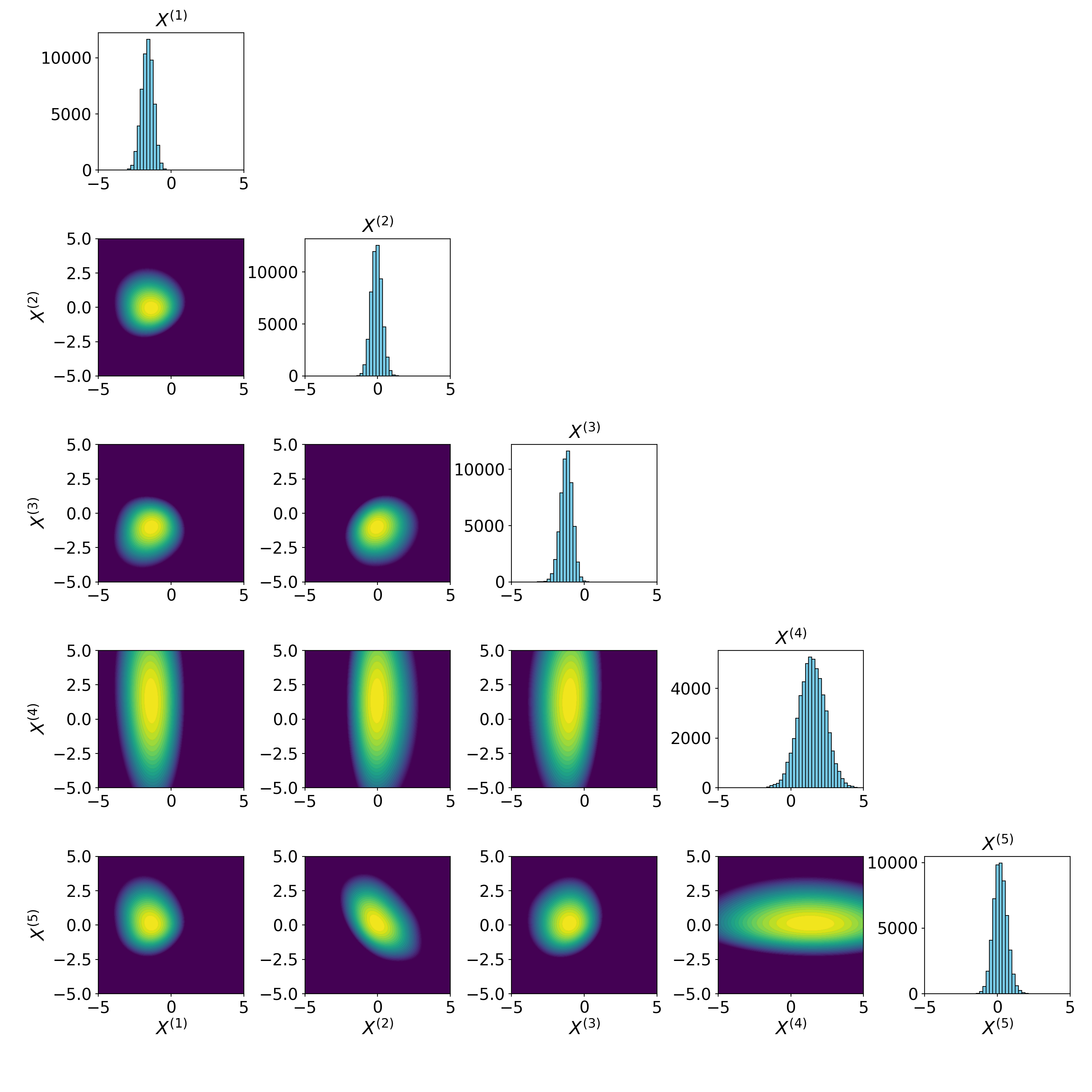}
\caption[Bayesian logistic regression, density]{Bayesian logistic regression. The off-diagonal panels show pairwise 2-D projections of the log-posterior (first 5 dims). The diagonal panel shows the posterior density.\label{fig:blr_density}}
\end{figure}

\begin{figure}[htbp]
\centering
\begin{minipage}[t]{0.8\textwidth}
    \centering
    \begin{subfigure}[b]{\textwidth}
        \includegraphics[width=\textwidth]{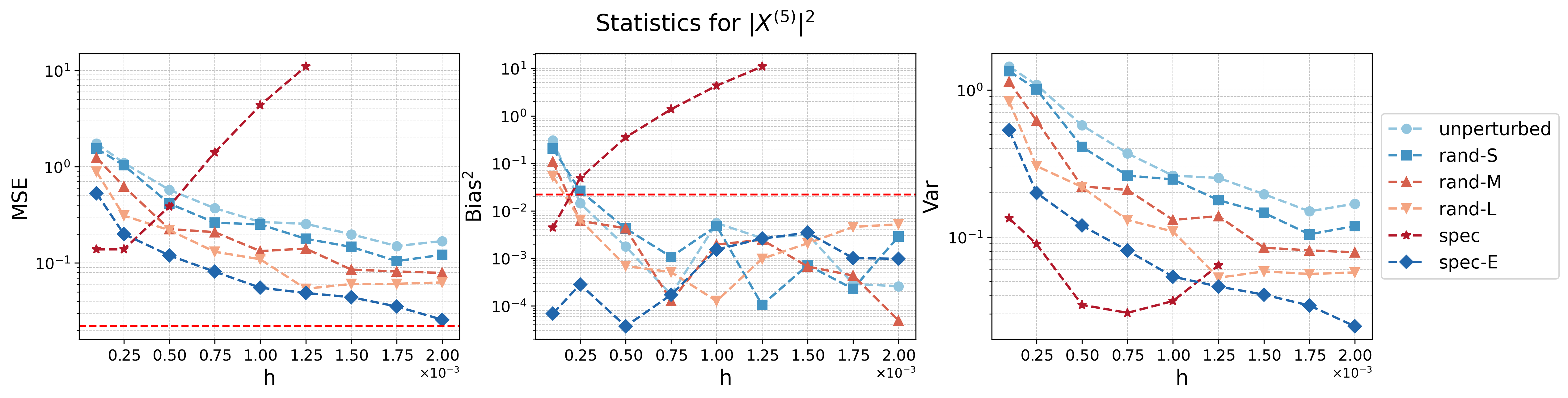}
        \includegraphics[width=\textwidth]{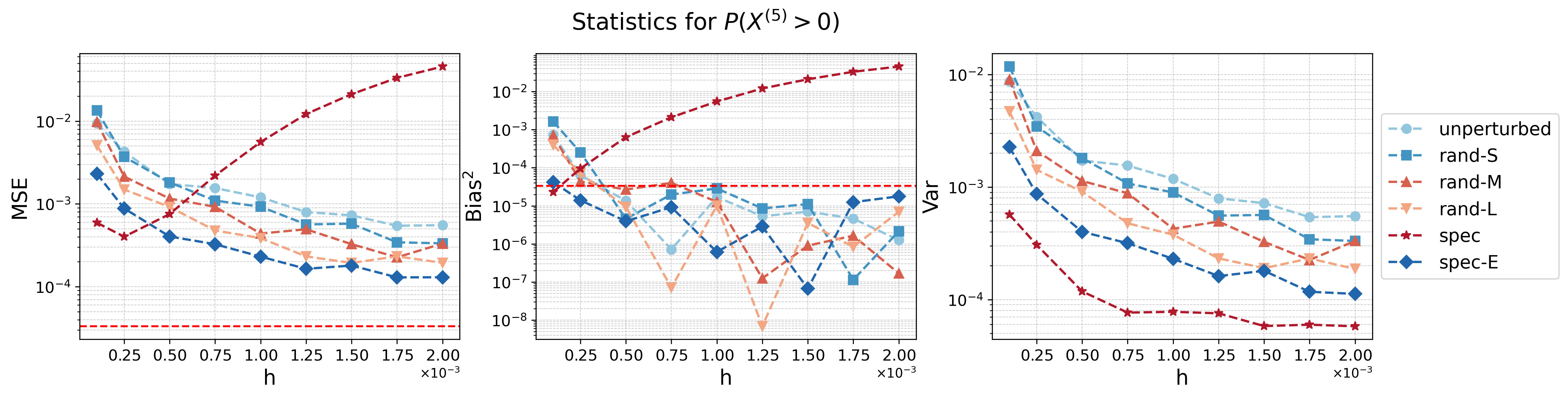}
        \includegraphics[width=\textwidth]{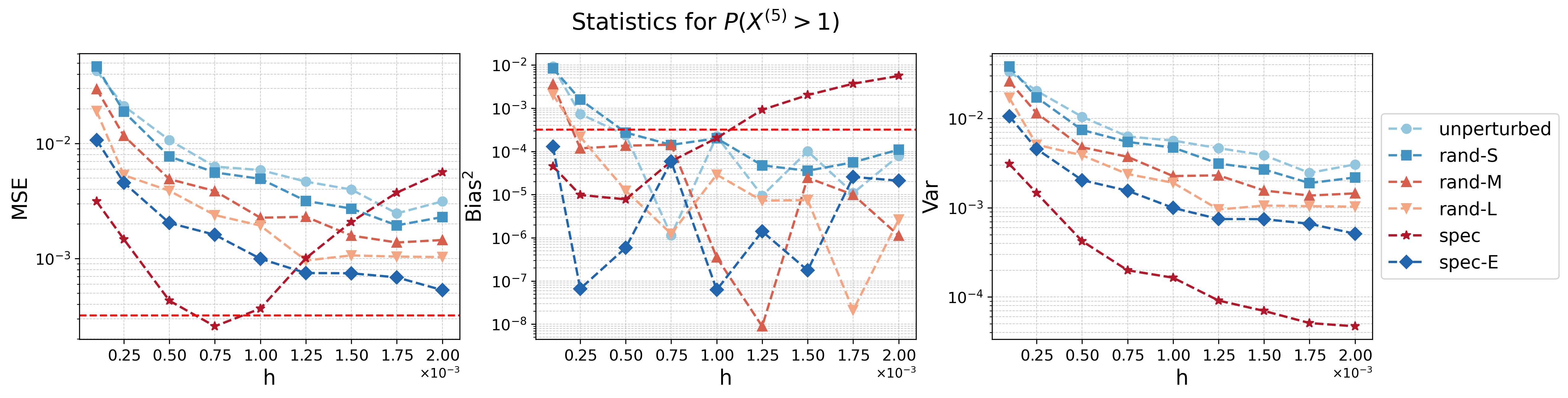}
        \includegraphics[width=\textwidth]{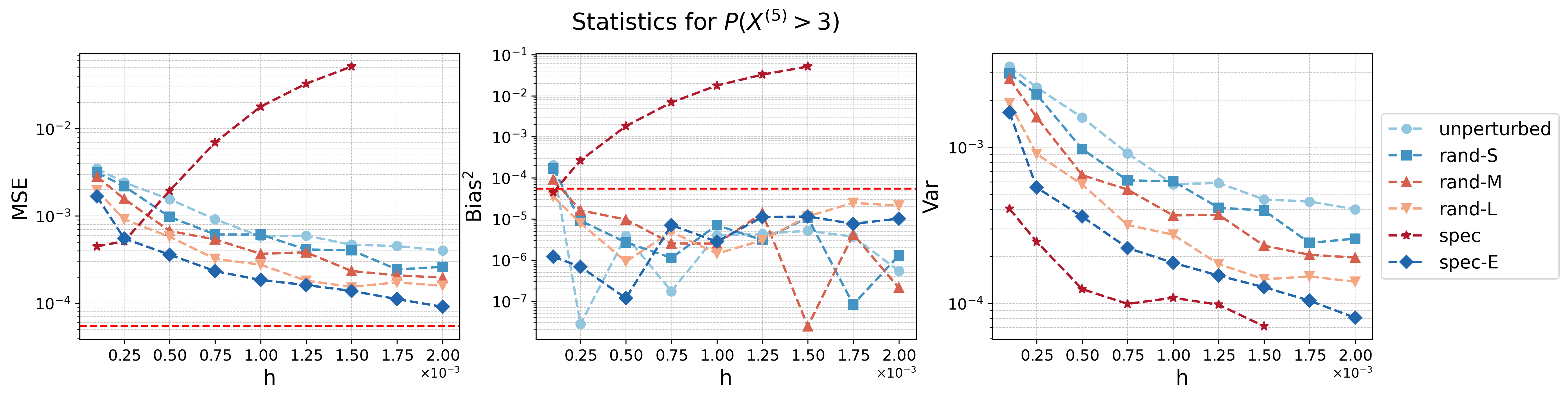}
        \includegraphics[width=\textwidth]{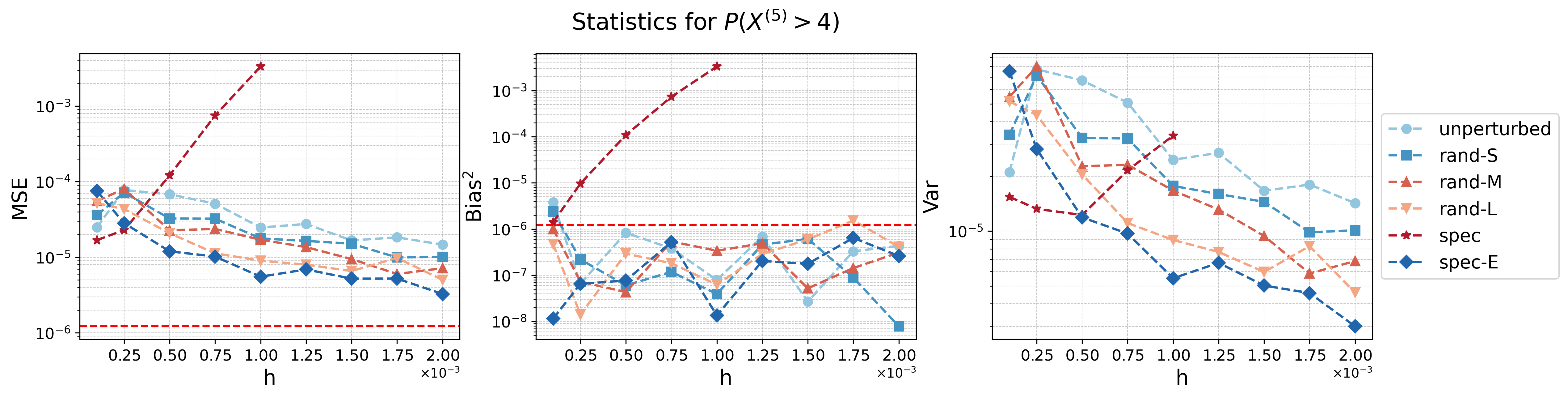}
    \end{subfigure}
\end{minipage}
\caption[Bayesian logistic regression, MSE vs. step size]{Additional results for Bayesian logistic regression: MSE vs. step size. For a panel of summary statistics, we report MSE, bias, and variance across step sizes. Our method attains the smallest (variance-dominated) MSE uniformly over $h$. \texttt{spec} is numerically unstable, while the other baselines are stable but exhibit larger estimation variance. \label{fig:app:exp3:1}}
\end{figure}

\begin{figure}[htbp]
\centering
\begin{minipage}[t]{0.7\textwidth}
    \centering
    \begin{subfigure}[b]{\textwidth}
        \includegraphics[width=\textwidth]{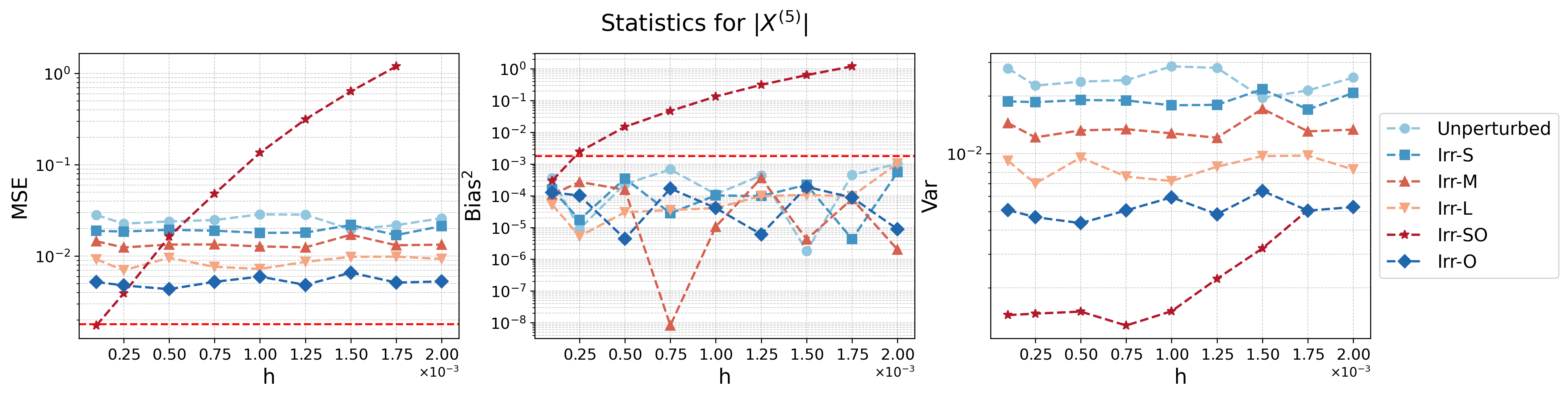}
        \includegraphics[width=\textwidth]{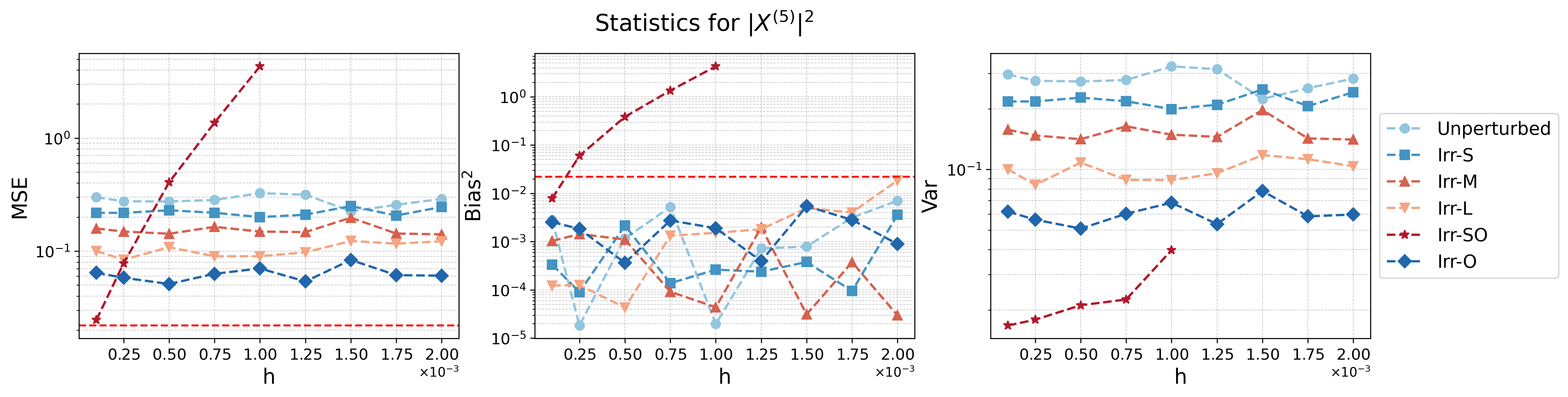}
        \includegraphics[width=\textwidth]{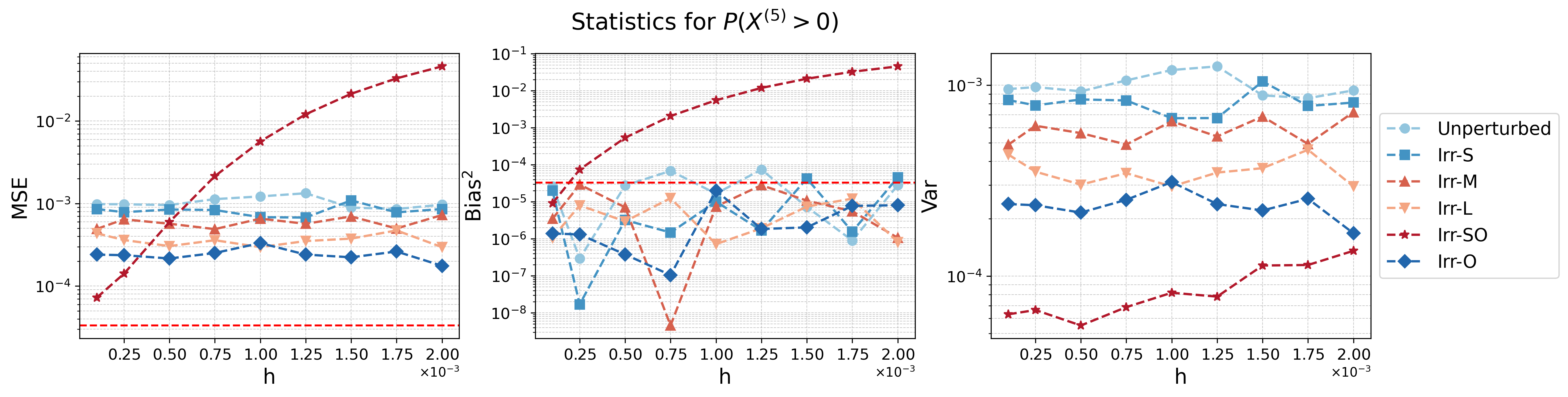}
        \includegraphics[width=\textwidth]{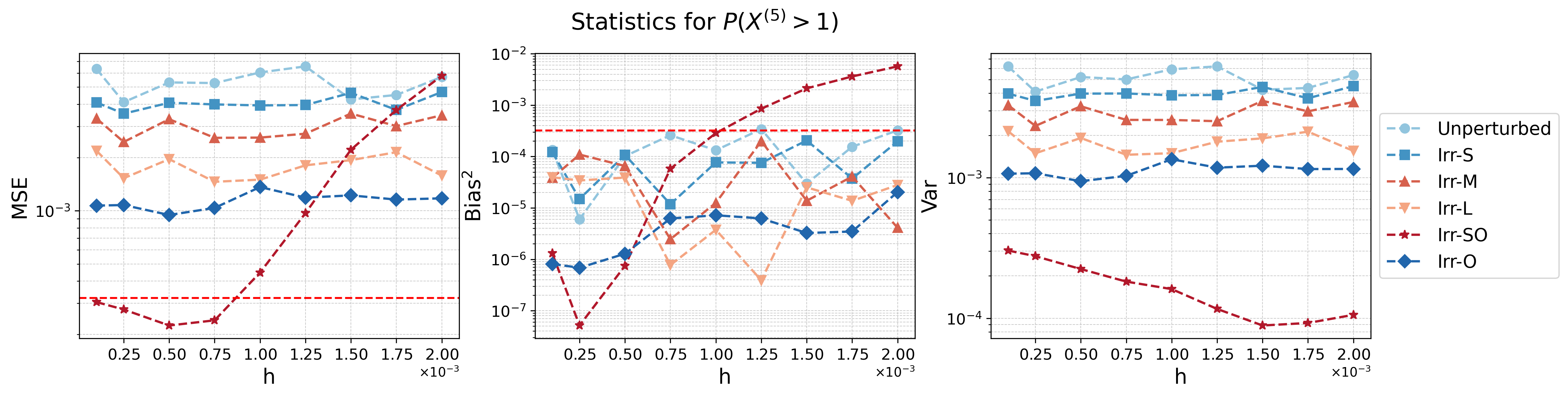}
        \includegraphics[width=\textwidth]{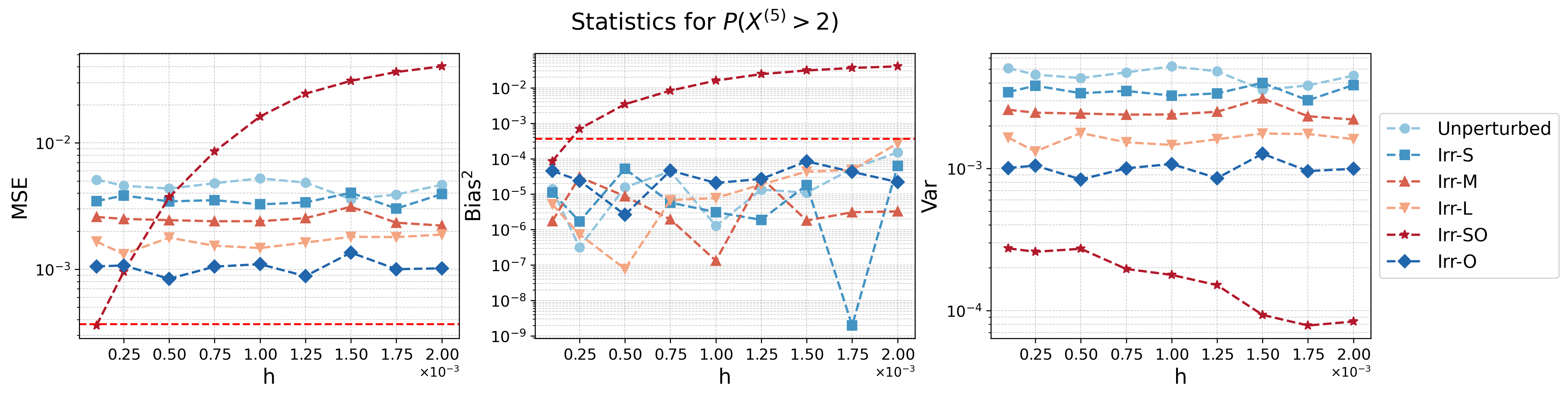}
        \includegraphics[width=\textwidth]{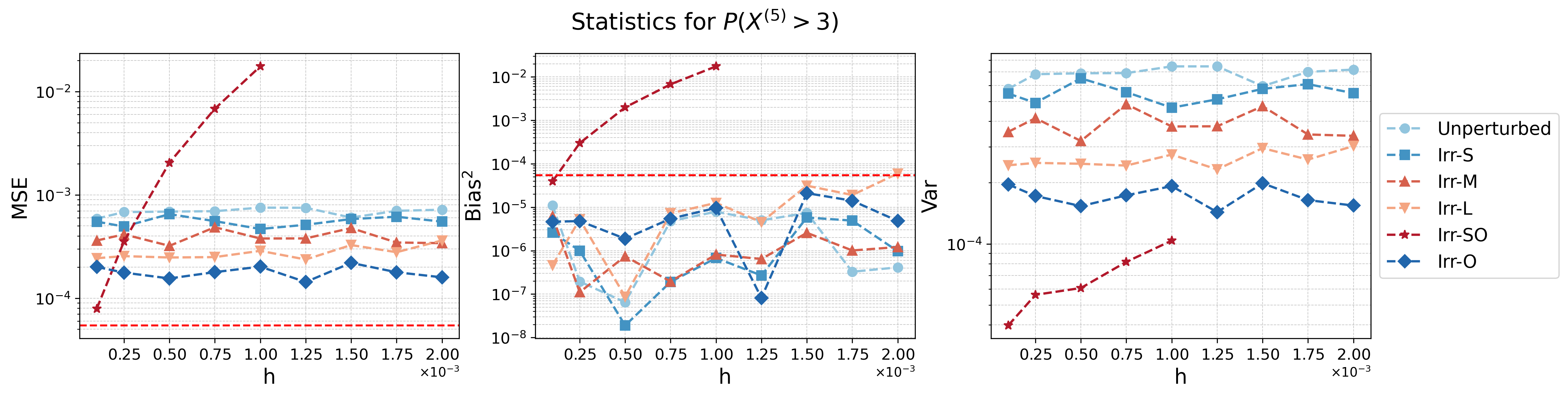}
        \includegraphics[width=\textwidth]{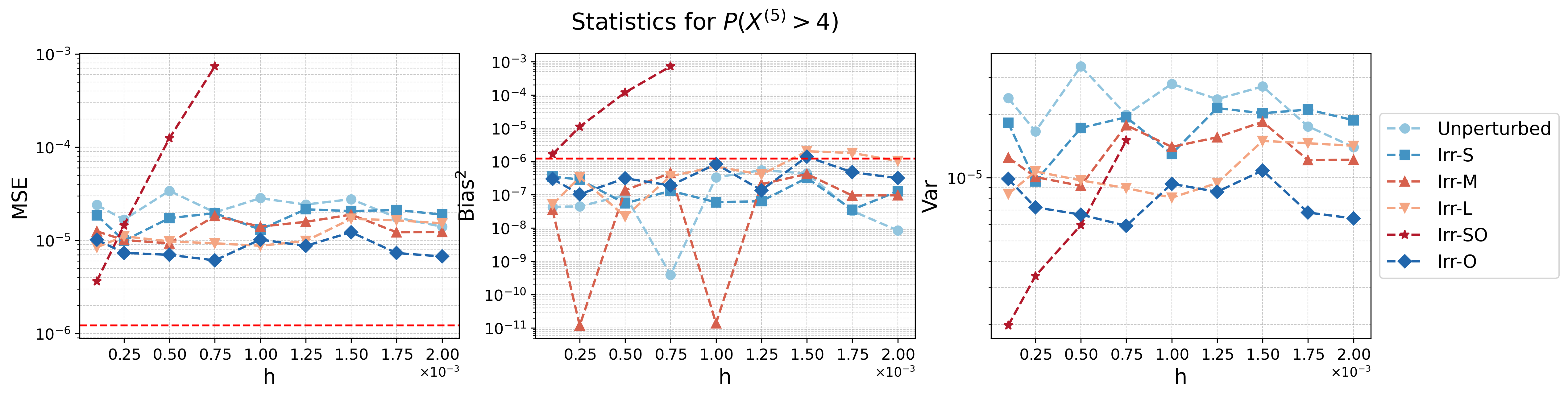}
    \end{subfigure}
\end{minipage}
\caption[Bayesian logistic regression, MSE vs. step size]{Bayesian logistic regression — MSE vs. step size. For a panel of summary statistics, we report MSE, bias, and variance across step sizes. Our method attains the smallest (variance-dominated) MSE uniformly over $h$. \texttt{spec} is numerically unstable, while the other baselines are stable but exhibit larger estimation variance.\label{fig:app:exp3:2}}
\end{figure}

\subsection{Independent component analysis}\label{app:numerical:ica}
The estimated Fisher information matrix $\hat{\mathcal{F}} \in \mathbb{R}^{9\times 9}$ (with $D = d^2 = 9$) used for constructing the spectral perturbation is
\begin{equation*}
\hat{\mathcal{F}} =
\begin{pmatrix}
1376.6 &  821.98 &  601.77 &   18.56 &   26.61 &    1.08 &    0.57 &   -5.64 &    7.88 \\
 821.98 & 1048.80 &  387.15 &   -0.57 &    5.98 &  -14.91 &   -7.88 &   -8.41 &  -16.20 \\
 601.77 &  387.15 &  544.71 &   12.76 &   17.55 &    3.13 &   24.72 &   17.63 &   16.26 \\
  18.56 &   -0.57 &   12.76 & 1379.31 &  817.66 &  596.03 &   20.69 &   26.94 &  -43.39 \\
  26.61 &    5.98 &   17.55 &  817.66 & 1036.79 &  380.94 &    8.62 &   13.46 &  -28.55 \\
   1.08 &  -14.91 &    3.13 &  596.03 &  380.94 &  539.21 &   34.54 &   32.86 &   -8.04 \\
   0.57 &   -7.88 &   24.72 &   20.69 &    8.62 &   34.54 & 1397.90 &  819.09 &  602.28 \\
  -5.64 &   -8.41 &   17.63 &   26.94 &   13.46 &   32.86 &  819.09 & 1030.88 &  380.12 \\
   7.88 &  -16.20 &   16.26 &  -43.39 &  -28.55 &   -8.04 &  602.28 &  380.12 &  536.94
\end{pmatrix}.
\end{equation*}

\section{Irreversible versus reversible perturbations}\label{app:robustness}

A natural question is whether the Fisher information matrix should instead be used to construct a reversible preconditioner.
The answer depends on how well the global Fisher information matrix \(\mathcal F\) represents the local geometry of the target.
For the Gaussian and mixture-of-Gaussian targets in \Cref{sec:Gaussian,sec:mix:of:Gaussian}, local and global curvature coincide, or nearly so.
In these settings, reversible preconditioning with \(\mathcal F^{-1}\) increases the spectral gap while introducing relatively little discretization error, and is therefore expected to perform well.

The BLR posterior is different: its curvature varies across the state space, so the local Hessian \(-\nabla^2\log\pi(x)\) can differ substantially from \(\mathcal F\).
A constant reversible preconditioner \(\mathcal F^{-1}\) modifies both the drift and the diffusion, so this mismatch can introduce persistent discretization bias.
An irreversible perturbation, by contrast, uses \(\mathcal F\) only to construct a skew-symmetric drift rotation and leaves the diffusion unchanged, making it more robust to global--local geometry mismatch.
The experiments below illustrate this distinction empirically on the BLR target: even with the true Fisher information matrix, our irreversible perturbation outperforms reversible preconditioning, and the gap widens in the adaptive setting, where inaccurate Fisher information matrix estimates can distort the reversible noise covariance.

All experiments use Euler--Maruyama discretization with step size \(\Delta t=2\times 10^{-3}\), \(K=20{,}000\) steps, \(M=32\) chains, and stochastic gradients with mini-batch size \(n=10\). We compare standard ULA with oracle and adaptive versions of the optimal position-independent reversible perturbations proposed by \citet{Ti2023} and with oracle and adaptive versions of our optimal position-independent irreversible perturbations \texttt{spec-E}. For adaptive methods, \(\mathcal F\) is estimated online using the streaming estimator from \Cref{sec:adaptive_FIM}, and the preconditioner or perturbation is updated every 100 steps.
For a paired comparison, all methods use the same Gaussian noise and mini-batch indices within each chain.
We report the MSE of running estimates of \(\mathbb E[\mathrm{sum}(X^{(i)})]\) and \(\mathbb E_\pi[\|X\|^2]\).

We emphasize that this comparison is empirical and target-dependent.
A more complete theoretical characterization of when irreversible perturbations are preferable to reversible preconditioning is beyond the scope of this paper and remains an interesting direction for future work.

\paragraph{Robustness to initialization}\label{app:robustness-init}
\Cref{fig:robustness-init-blr} compares standard ULA, oracle/adaptive reversible preconditioning, and oracle/adaptive irreversible perturbations, initialized at \(X_0=s\mathbf 1_d\) with \(s\in\{0,1,5,10\}\).
Adaptive reversible preconditioning stalls across initialization distances, while adaptive irreversible perturbation converges reliably and tracks its oracle version.
Even with the true Fisher information matrix, the irreversible perturbation outperforms reversible preconditioning, indicating that \(\mathcal F^{-1}\) is a poor constant reversible preconditioner for this spatially varying posterior.

\begin{figure}[ht]
    \centering
    \includegraphics[width=\textwidth]{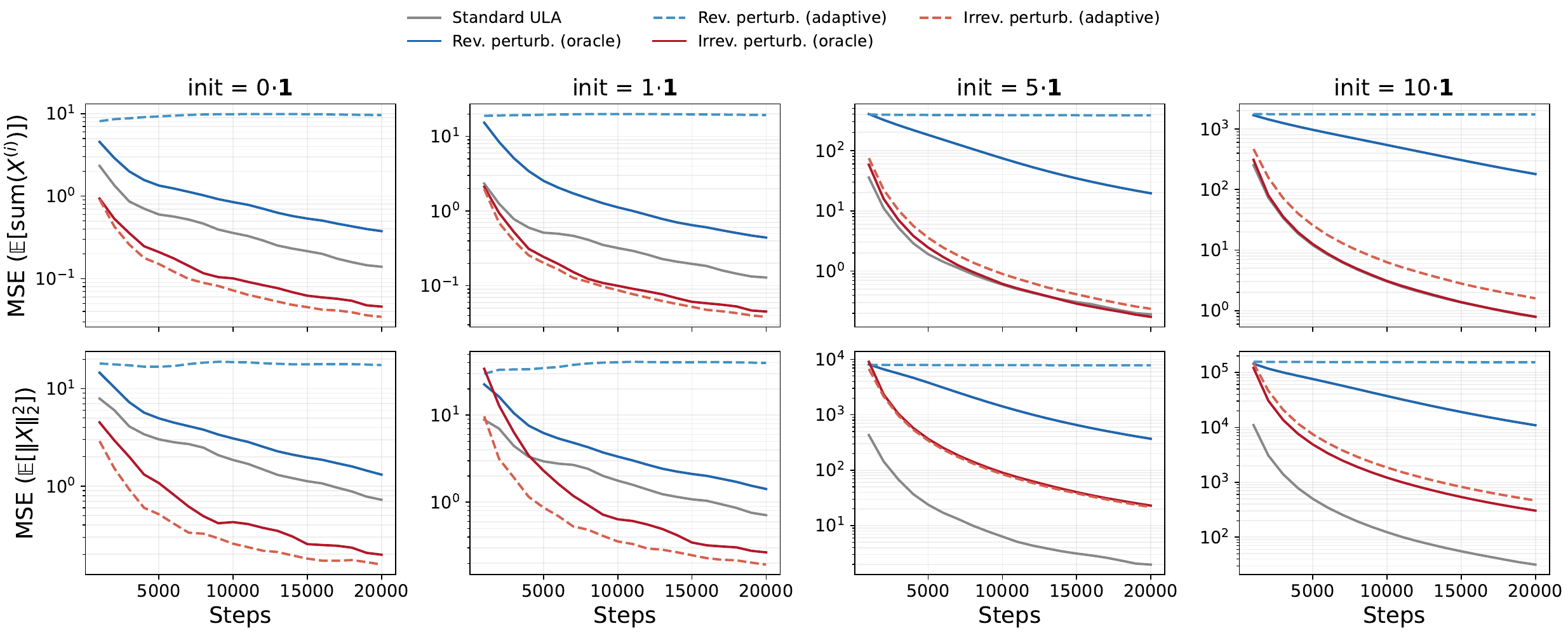}
    \caption{
        \textbf{Robustness to initialization (BLR, $d=20$).}
        MSE of $\mathbb{E}[\mathrm{sum}(X^{(i)})]$ (top) and $\mathbb{E}[\|X\|_2^2]$ (bottom) for five methods starting from $X_0 = s\cdot\mathbf{1}$ with $s \in \{0, 1, 5, 10\}$, statistics computed using $M=128$ chains.
        Adaptive reversible perturbation (dashed blue) completely stalls, while adaptive irreversible perturbation (dashed red) converges reliably.
    }
    \label{fig:robustness-init-blr}
\end{figure}

\paragraph{Trajectory behavior}\label{app:robustness-trajectory}

\Cref{fig:traj-oracle,fig:traj-adaptive} visualize representative trajectories projected onto the top two principal components of the posterior.
The oracle irreversible perturbation produces a spiraling trajectory that explores posterior level sets more broadly than reversible preconditioning.
In the adaptive case, reversible preconditioning collapses into a trapped cluster for distant initializations, whereas the adaptive irreversibly perturbed chain continues to explore and moves toward the posterior mean.

\begin{figure}[ht]
    \centering
    \includegraphics[width=\textwidth]{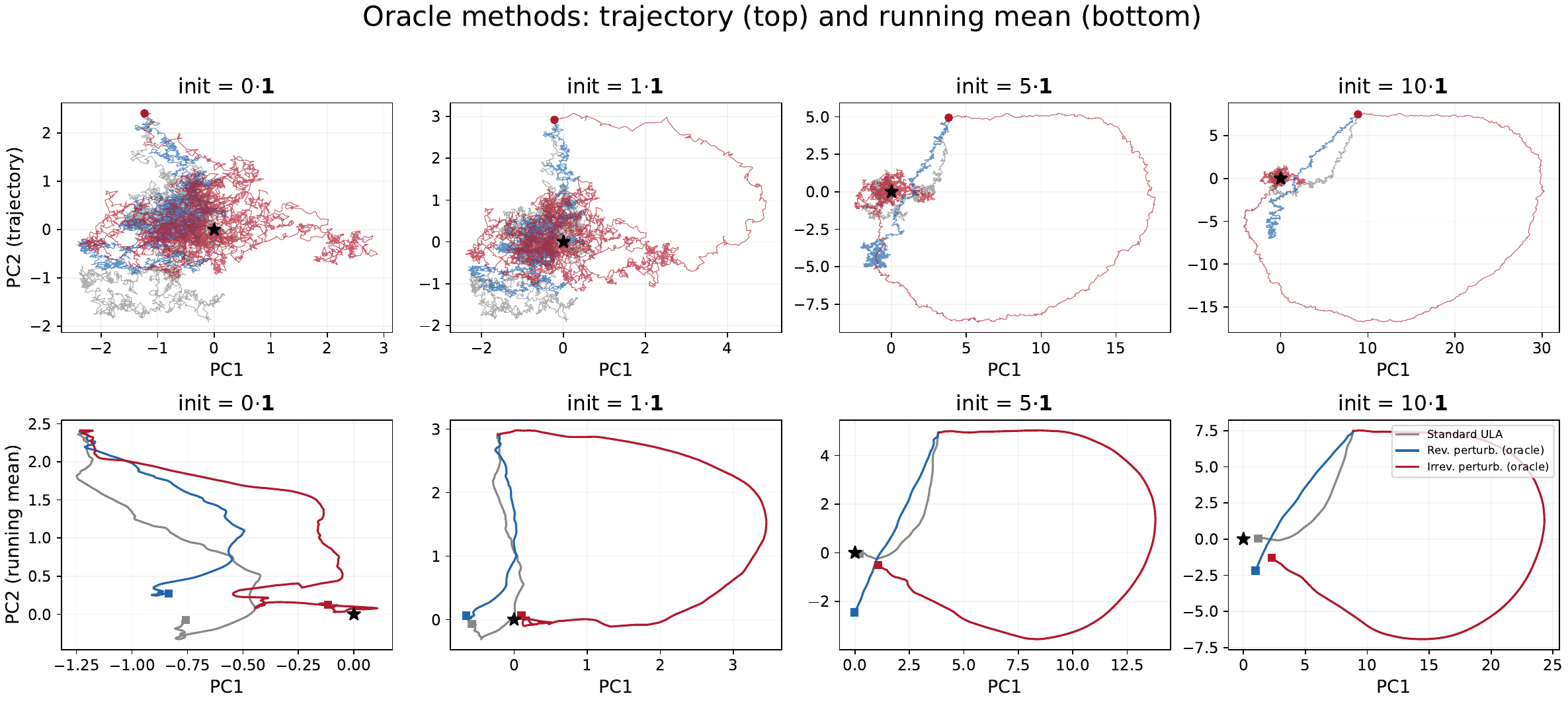}
    \caption{
        \textbf{Oracle trajectories (BLR, $d=20$, single chain).}
        Top: raw trajectories projected onto PC1--PC2. Bottom: running cumulative mean (square = final position, star = posterior mean).
        The irreversible chain (red) spirals toward the mode; the reversible chain (blue) takes a more direct but slower path.
    }
    \label{fig:traj-oracle}
\end{figure}

\begin{figure}[ht]
    \centering
    \includegraphics[width=\textwidth]{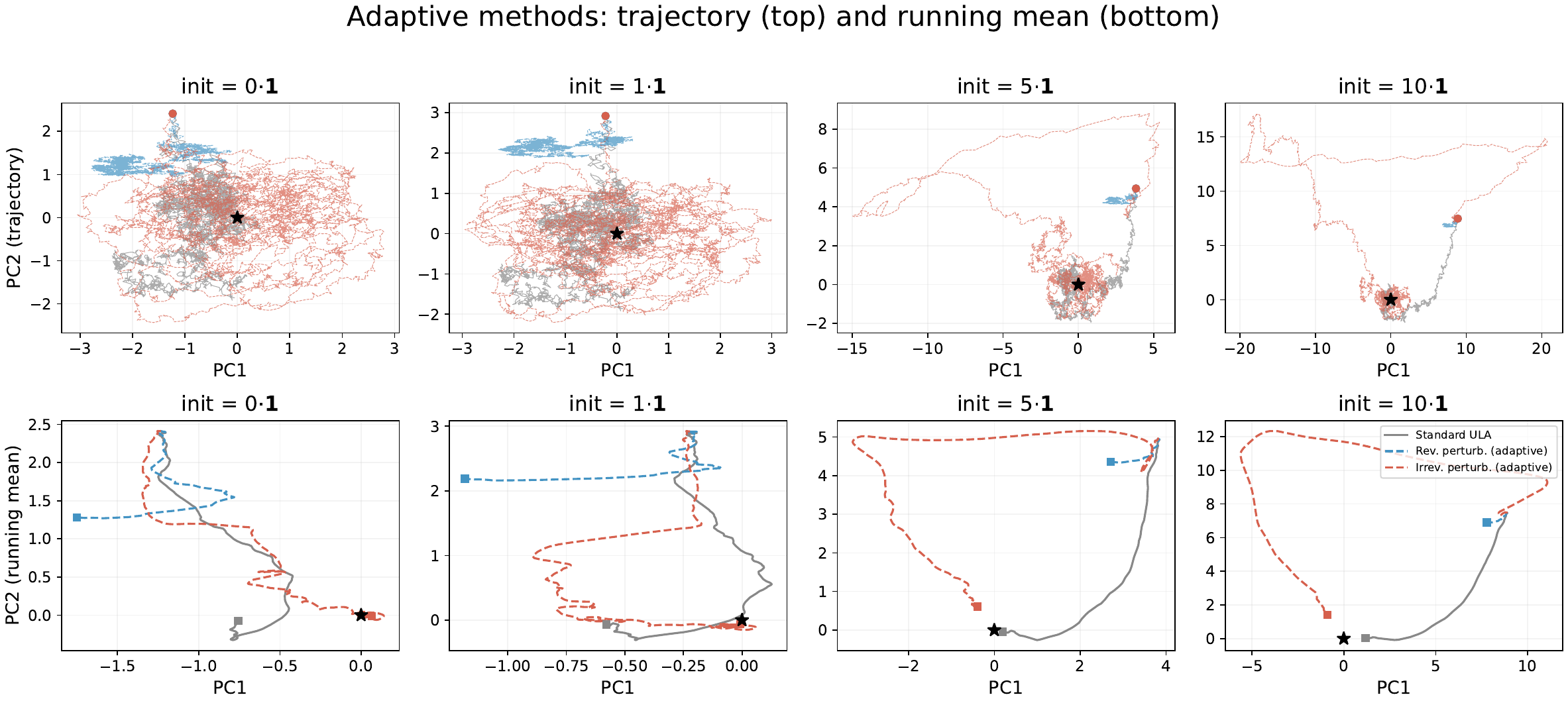}
    \caption{
        \textbf{Adaptive trajectories (BLR, $d=20$, single chain).}
        The adaptive reversible chain (blue) collapses into a trapped cluster at $s \geq 5$, while the adaptive irreversible chain (red) explores widely and converges.
    }
    \label{fig:traj-adaptive}
\end{figure}

\paragraph{Robustness to inaccurate Fisher estimates}\label{app:robustness-fim}

To isolate the effect of FIM quality, we corrupt the Fisher information matrix as
\[
    \mathcal{F}_\alpha = \alpha \, \mathcal{F}_{\mathrm{true}} + (1 - \alpha) \, I_d, \qquad \alpha \in \{0, 0.25, 0.5, 0.75, 1\},
\]
and initialize all chains at $X_0 = \mathbf{0}$.
\Cref{fig:robustness-fim-curves-blr,fig:robustness-fim-summary-blr} show that irreversible perturbations consistently outperform reversible preconditioning, both with oracle and corrupted Fisher matrices.
The adaptive irreversible perturbation achieves the lowest MSE across corruption levels, while adaptive reversible preconditioning again stalls.

\begin{figure}[ht]
    \centering
    \includegraphics[width=\textwidth]{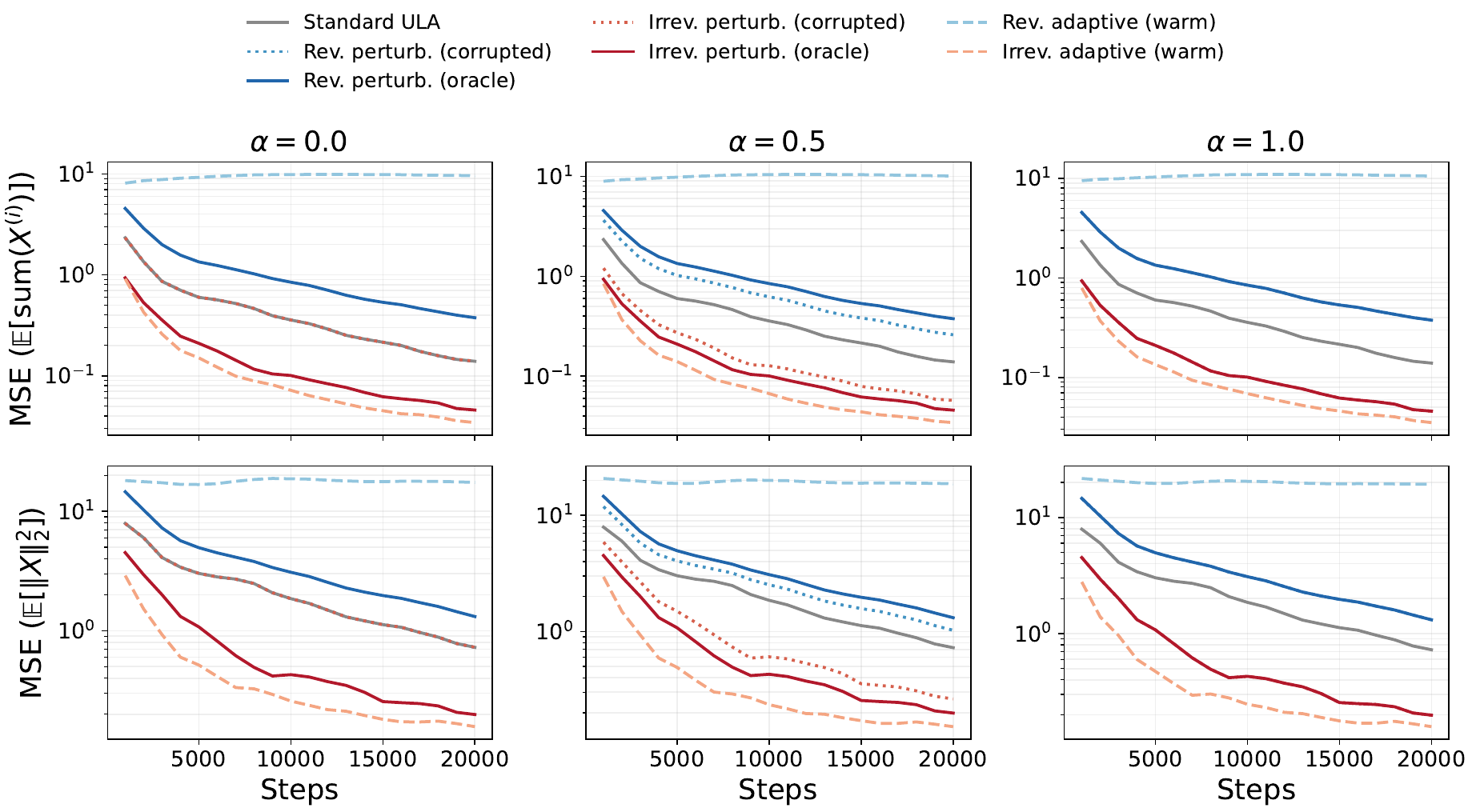}
    \caption{
        \textbf{Robustness to corrupted FIM (BLR, $d=20$): MSE decay curves.}
        Columns correspond to $\alpha \in \{0, 0.5, 1\}$.
        Adaptive irreversible (dashed orange) is the best method at all corruption levels, while adaptive reversible (dashed blue) stalls.
    }
    \label{fig:robustness-fim-curves-blr}
\end{figure}

\begin{figure}[ht]
    \centering
    \includegraphics[width=0.85\textwidth]{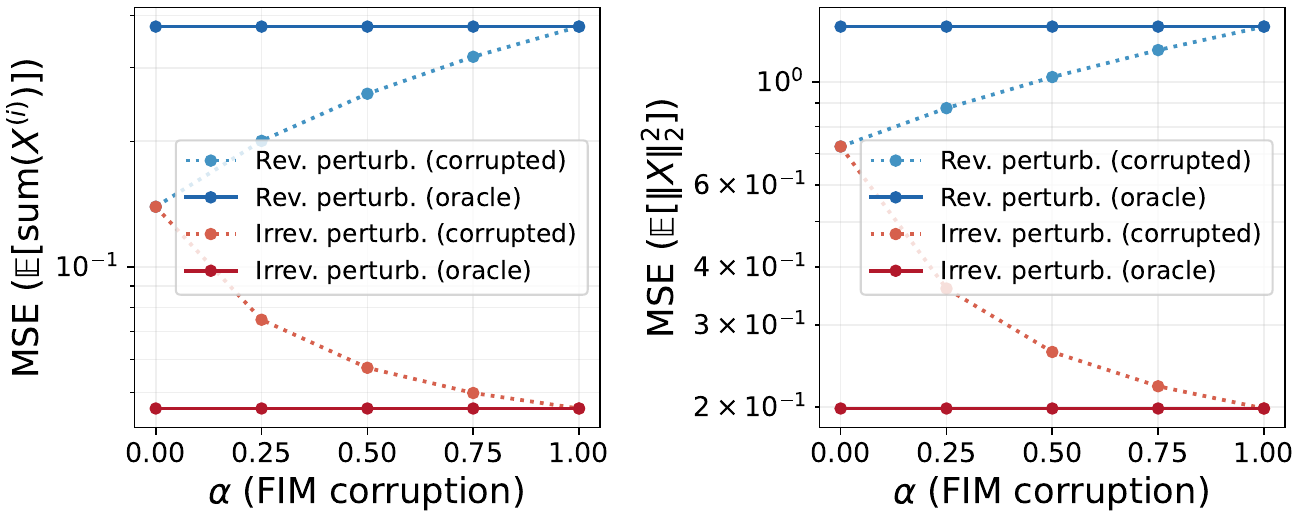}
    \caption{
        \textbf{Final MSE vs.\ FIM corruption level (BLR, $d=20$).}
        Irreversible perturbation (red) consistently outperforms reversible perturbation (blue), both with oracle and corrupted FIM.
        The gap widens as FIM quality worsens ($\alpha \to 0$).
    }
    \label{fig:robustness-fim-summary-blr}
\end{figure}

\paragraph{Analysis of adaptive failure modes}\label{app:robustness-analysis}

Failure of the adaptive reversible method (and success of the adaptive \textit{irreversible} method) can be ascribed to how
the two methods use an inaccurate Fisher information matrix estimate.
Reversible preconditioning inverts \(\hat{\mathcal F}\) and uses \(\hat{\mathcal F}^{-1}\) for both drift and noise; when \(\hat{\mathcal F}\) is ill-conditioned, this suppresses exploration in important directions and prevents recovery.
An irreversible perturbation based on \(\hat{\mathcal F}\) may also produce a poorly calibrated rotational term in the drift, but the diffusion remains isotropic, so the chain continues to explore and the Fisher estimate can improve.
This drift--diffusion decoupling can explain the greater robustness of irreversible perturbations in adaptive schemes.

\bibliography{irr_references}

\end{document}